\documentclass[a4paper]{article}

\usepackage[T1]{fontenc}
\usepackage[utf8]{inputenc}
\usepackage[english]{babel}
\usepackage{graphicx}
\usepackage{epsfig}
\usepackage{amsmath}
\usepackage{amssymb}
\usepackage{amsthm}
\usepackage{mathtools}
\usepackage{url}
\usepackage[breaklinks]{hyperref}
\usepackage{float}
\usepackage{pifont}
\usepackage[table]{xcolor}
\usepackage{lmodern}
\usepackage{xspace}
\usepackage{bbold}
\usepackage[capitalize,nameinlink]{cleveref}
\usepackage{framed}
\usepackage{multirow}
\usepackage{enumitem}
\usepackage{subcaption}
\usepackage{arydshln}
\usepackage[misc]{ifsym}
\usepackage{epstopdf}

\definecolor{lightgray}{gray}{.85}

\title{\vskip-2em Model~Order~Reduction \\ for \\ Gas~and~Energy~Networks}
\author{%
Christian Himpe\thanks{ORCiD: \href{http://orcid.org/0000-0003-2194-6754}{0000-0003-2194-6754}, Contact: \href{mailto:himpe@mpi-magdeburg.mpg.de}{\nolinkurl{himpe@mpi-magdeburg.mpg.de}},
Computational Methods in Systems and Control Theory Group at the Max Planck Institute for Dynamics of Complex Technical Systems, Sandtorstra{\ss}e~1, D-39106 Magdeburg, Germany}~\,$^{\textrm{\Letter}}$~~~%
Sara Grundel\thanks{ORCiD: \href{http://orcid.org/0000-0002-0209-6566}{0000-0002-0209-6566}, Contact: \href{mailto:grundel@mpi-magdeburg.mpg.de}{\nolinkurl{grundel@mpi-magdeburg.mpg.de}},
Computational Methods in Systems and Control Theory Group at the Max Planck Institute for Dynamics of Complex Technical Systems, Sandtorstra{\ss}e~1, D-39106 Magdeburg, Germany}~~~
Peter Benner\thanks{ORCiD: \href{http://orcid.org/0000-0003-3362-4103}{0000-0003-3362-4103}, Contact: \href{mailto:benner@mpi-magdeburg.mpg.de}{\nolinkurl{benner@mpi-magdeburg.mpg.de}},
Computational Methods in Systems and Control Theory Group at the Max Planck Institute for Dynamics of Complex Technical Systems, Sandtorstra{\ss}e~1, D-39106 Magdeburg, Germany}
}

\date{}

\DeclareMathOperator{\R}{\mathbb{R}}

\DeclareMathOperator{\D}{\mathrm{d}\!}
\DeclareMathOperator{\diag}{diag}
\DeclareMathOperator{\tr}{tr}
\DeclareMathOperator{\svd}{SVD}
\DeclareMathOperator{\evd}{EVD}
\DeclareMathOperator{\tsvd}{tSVD}
\DeclareMathOperator{\tevd}{tEVD}
\DeclareMathOperator{\id}{id}
\newcommand{\splus}{{\scriptstyle+}}
\newcommand{\T}{\mathsf{\scriptscriptstyle T}}

\newcommand{\orelse}{\textbar\,\,}

\newcommand{\morgen}{\texttt{morgen}}

\newtheoremstyle{thm}{\topsep}{\topsep}{\normalfont \itshape}{}{\normalfont \bfseries}{}{\newline}{}
\theoremstyle{thm}

\newtheorem{proposition}{Proposition}
\makeatletter
\def\@cite#1#2{[{#1\if@tempswa ,~#2\fi}]}% NEW
\makeatother
\begin{document}

\maketitle

{\textbf{Abstract}:\\
To counter the volatile nature of renewable energy sources, gas networks take a vital role.
But, to ensure fulfillment of contracts under these circumstances,
a vast number of possible scenarios, incorporating uncertain supply and demand, has to be simulated ahead of time.
This many-query gas network simulation task can be accelerated by model reduction,
yet, large-scale, nonlinear, parametric, hyperbolic partial differential(-algebraic) equation systems, modeling natural gas transport,
are a challenging application for model order reduction algorithms.

For this industrial application, we bring together the scientific computing topics of:
mathematical modeling of gas transport networks,
numerical simulation of hyperbolic partial differential equation,
and parametric model reduction for nonlinear systems.
This research resulted in the \morgen{} (Model Order Reduction for Gas and Energy Networks) software platform,
which enables modular testing of various combinations of models, solvers, and model reduction methods.
In this work we present the theoretical background on systemic modeling and structured, data-driven, system-theoretic model reduction for gas networks,
as well as the implementation of \morgen{} and associated numerical experiments testing model reduction adapted to gas network models.
}

\vskip1em


\begin{table}[h]
\begin{tabular}{rl|rl}
 $p$ & Pressure &  $N_s$ & Number of supply nodes \\
 $q$ & Mass-flux & $N_0$ & Number of internal nodes \\
 $\bar{p}$ & Steady-state pressure  & $N_d$ & Number of demand nodes \\
 $\bar{q}$ & Steady-state mass-flux & $N_c$ & Number of compressors \\
 $s_p$ & Supply node pressure  & $N_p$ & Dimension of pressure space \\
 $d_q$ & Demand node mass-flux & $N_q$ & Dimension of mass-flux space \\
 $s_q$ & Supply node mass-flux & $n_p$ & Dimension of reduced pressure space \\
 $d_p$ & Demand node pressure  & $n_q$ & Dimension of reduced mass-flux space
\end{tabular}
\caption{List of recurring symbols.}
\end{table}

\pagebreak

\section{Introduction} 
%
% Motivation
Rapid transient simulations of gas flow in pipeline networks are essential for safe operations of gas networks as well as reliable delivery of denominations.
Yet, in a volatile supply and demand environment, due to increasing renewable energy sources,
the time horizon for planning dispatch and load forecasting shortens while more sources of accountable uncertainties,
such as effects of weather on energy consumption and production are introduced; to a lesser degree this is a long standing challenge~\cite{Ash93}.
An example is the interconnection of gas and power grids through gas-fired power plants~\cite{CheFBetal15,ZloCB15,MakVZetal16}.
Thus, more simulations for the uncertainty quantification of dynamic gas network behavior need to be completed in less time by the gas grid operators.
However, available compute power is (and was~\cite{VosZ93}) never sufficient.
To this end we evaluate customized model reduction techniques for an established class of gas network models.

This work and the associated software platform is an effort to determine which model reduction methods are suitable for enabling \emph{digital twin}s~\cite{LuGAetal19,Sar19,morHarHW18} of gas networks.
Depending on the mathematical model and quantities of interest, the twin may contain redundant or superfluous information with respect to the simulations.
Therefore, model reduction compresses the twin to a matched surrogate model, which is sufficiently accurate in the chosen operating region.

The swift numerical simulation of gas network twins by reduced order modeling is highly relevant,
not only due to the transition towards renewables at the time of writing,
which is underlined by the research projects \emph{MathEnergy}\footnote{\url{https://mathenergy.de}} (Mathematical Key Technologies for Evolving Energy Grids)~\cite{CleBBetal21} that the authors are part of,
and \emph{TRR154}\footnote{\url{https://trr154.fau.de}} (Mathematical Modelling, Simulation and Optimization using the Example of Gas Networks)~\cite{LanLMetal15},
but also because of the intriguing numerical problem of model reduction for hyperbolic, nonlinear, coupled, parametric, multiscale partial differential-algebraic equation systems.

% \textit{PSIG}\footnote{\url{https://psig.org}} (Pipeline Simulation Interest Group)

%Literature Review

%% Modelling of gas networks
If relevant intraday demand changes occur, established steady / stationary / static simulations may not be sufficient anymore~\cite{FarR16}. 
The basic model for the simulation of unsteady / dynamic / transient flow processes in gas pipelines is based on the one-dimensional (isothermal) Euler equations,
originally introduced in~\cite{Guy67},
and popularized in~\cite{Osi84} as well as in~\cite{KraSVetal84a,KraSVetal84} around the same time.
A practical extension in the context of gas networks is the repetitive modeling approach~\cite{DymLJ07},
which enables a modular construction.
For extensive details on gas network modeling,
see the works~\cite{Szi86,Osi87,FugGGetal15,SchSW15,MisFH16,DomHLetal17,BenGHetal18,StoM18},
and for a concise summary of the overall approach we recommend~\cite{AzeJ07}.
Furthermore, a system-theoretic approach to gas networks is discussed in~\cite{DorT08,GugH20},
and results on boundary reachability (controllability) and observability for this class of models have been derived in~\cite{AzeJL13,AzeJL15}.

% Model reduction for gas networks review
In terms of complexity reduction for gas network models, earlier works applied techniques such as
combining parallel pipelines~\cite{StrW70}, singular perturbation~\cite{morSte87} and symbolic simplifications~\cite{morMohHHetal04}.
Younger works introduced projection-based model reduction methods from fluid dynamics,
proper orthogonal decomposition (POD)~\cite{morGruHKetal13,morGruJHetal14,morGruHR16},
and system-theoretic methods~\cite{morLuMM16} (matrix interpolation), \cite{morAal10,morLuMM17} (balanced truncation), or
Pad\'e-type approximations~\cite{morLilM17,morEggKLetal18} (moment matching).

In this work, we conceptually combine these previous approaches,
by using system-theoretic but data-driven methods that are structure-preserving.
The utilized data-driven assembly of the system-theoretic operators, central to the employed methods,
is also a partial answer to the challenges posed in~\cite[Remark~5.10]{HanLMetal17};
while structure preservation means in this context, retaining (particularly not mixing) the discretized physical variables in the reduced order model.
Furthermore, we note that from this work's point of view,
\cite{ZloCB15,ZloDBetal15,SunZ19} are concerned rather with (valuable) model simplifications than model reduction.

%% Contribution
To avoid the analytically most complex aspects of the gas network model -- the nonlinearities --
one could linearize the model around an operating point.
Yet, the different nonlinearities (i.e. friction, compressibility and compressors) are unlikely to have a compatible operating point for a wide range of scenarios.
Furthermore, linearized and simplified models of gas flow have limitations with simulations of real scenarios~\cite{Hen18}, \cite[Ch.~7]{van04};
hence, we use a nonlinear model.
Since there is no general theory for model reduction of nonlinear systems,
and a high degree of modularity in the gas network modeling process,
model reduction algorithms have to be compared heuristically to determine their applicability.
% morgen
As a result of this reasoning and a demand for gas network simulation software tools~\cite{DomGHetal15,morHimGB21}, a platform named \morgen{}
(\textbf{m}odel \textbf{o}rder \textbf{r}eduction for \textbf{g}as and \textbf{e}nergy \textbf{n}etworks)
was designed with the goal to compare different models, solvers and reductors.
The \morgen{}\footnote{``Morgen'' is also the German language word for ``tomorrow''.} platform is an open-source project,
and designed in a configurable, modular, and extensible manner, so that modeling, discretization or model reduction specialists can compare their methods.

In summary, this work contributes a full, but also fully modular, modeling, model reduction and simulation open-source software stack for gas networks,
and potentially other energy network systems (i.e. district heating networks, water networks), which brings together research results from various disciplines.

% Outline
Overall, this work is organized as follows:
In \cref{sec:model} the gas network model, simplifications, non-pipe elements,
a relation to port-Hamiltonian theory, and obtaining a steady-state initial condition are described.
\cref{sec:reduction} and \cref{sec:methods} outline the general model reduction idea and propose five reduction method classes.
The design and features of the \morgen{} platform are summarized in \cref{sec:morgen},
followed by three sets of numerical experiments in \cref{sec:numex}.
We conclude by an outlook (\cref{sec:outlook}) and an evaluation of our findings in \cref{sec:conclusion}.

%%%%%%%%%%%%%%%%%%%%%%%%%%%%%%%%%%%%%%%%%%%%%%%%%%%%%%%%%%%%%%%%%%%%%%%%%%%%%%%%
%%%%%%%%%%%%%%%%%%%%%%%%%%%%%%%%%%%%%%%%%%%%%%%%%%%%%%%%%%%%%%%%%%%%%%%%%%%%%%%%
%%%%%%%%%%%%%%%%%%%%%%%%%%%%%%%%%%%%%%%%%%%%%%%%%%%%%%%%%%%%%%%%%%%%%%%%%%%%%%%%

\section{The Transient Gas Network Model}\label{sec:model}
The goal of this section is to describe the partial differential-algebraic equation model of a gas network as an input-output system
that maps boundary values to quantities of interest.
First, the model for a single pipeline is summarized, which is based on the isothermal Euler equations of gas dynamics~\cite{Osi84,Osi96}.
Then, it is generalized to a network of pipes, and simplified compressors are added.
Additionally, a connection to energy-based modeling is made.

Even though further non-pipe elements are common in gas networks,
such as resistors, coolers, heaters, valves and control valves~\cite{FugGGetal15,MokPM19},
we prioritized compressors to focus on the model reduction aspect on a macro scale.
Moreover, the practical numerical problems of scale homogenization, spatial discretization,
index reduction and steady-state approximation are discussed in this section.

\pagebreak

\subsection{The Gas Pipeline Model}
The principal building blocks of gas transport networks are pipelines or ducts.
Since the length of pipes exceeds their diameter by far ($L>500d$, \cite{KraSVetal84}),
a spatial one-dimensional model suffices.
We model gas flow in a (cylindrical) pipe of length $L$ connecting two junctions by the isothermal Euler equations:
\begin{align}\label{eq:pipeline}
\begin{split}
 \frac{1}{\gamma_0 \, z_0} \, \partial_t p &= -\frac{1}{S} \, \partial_x q, \\
 \partial_t q &= -S \, \partial_x p - \Big(\frac{S \, g \,\partial_x h}{\gamma_0 \, z_0} p + \frac{\gamma_0 \, z_0 \lambda_0}{2 \, d \, S} \frac{|q| \, q}{p}\Big),
\end{split}
\end{align}
which determine the evolution of the coupled pressure $p(x,t)$ and mass-flux $q(x,t)$ variables.
The physical dimension of the pipe enters as its diameter $d$ and the derived cross-section area $S = \frac{\pi}{4}d^2$,
which is assumed constant, ignoring the influence of temperature and pressure on the pipe walls.
These coupled partial differential equations (PDE) can also be characterized as a nonlinear, two-dimensional, first-order hyperbolic system of conservation laws:
the pressure $p$ preserves continuity, while the mass-flux $q$ conserves momentum.

Following~\cite{Kiu94,HerMS10,PamBD16} and~\cite[Sec.~2.1]{BenGHetal18}, the inertia term has been neglected due to a low Mach number $m \ll 1$,
which leads to the \emph{ISO2} model in the \emph{TRR154} classification~\cite{DomHLetal17},
also known as \emph{friction-dominated model}~\cite[Sec.~3.2.1]{BroGH11}.
Furthermore, we assume a turbulent flow with a Reynolds number exceeding \linebreak $\operatorname{Re} \gg 10^5$ \cite{Guy67,EhrS05},
neither line breaks or valve closings happen intraday (to preclude associated shocks~\cite{DorF11}),
and low-frequency boundary values~\cite{PamBD16,AzePA19},
which in this work are the supply pressure and demand mass-flux,
due to frequent use in literature, and use-cases like \emph{guaranteed demand pressures}~\cite{Her07,BerS19}.

In~\eqref{eq:pipeline}, the linear reaction term describes the effect of gravity (with standard gravity $g \equiv 9.80665 \frac{\text{m}}{\text{s}^2}$) due to the pipe height $h$,
while the nonlinear reaction term models loss of momentum due to friction at the pipe walls,
specified by the (Darcy-Weisbach) friction factor $\lambda_0 := \lambda(d, k, \operatorname{Re}_0)$,
given a pipe roughness $k$, and an estimated mean Reynolds number $\operatorname{Re}_0$, see~\cite[Sec.~2.2]{BenGHetal18}\footnote{Additionally to~\cite{BenGHetal18},
the IGT formula~\cite{DomHLetal17}, \cite[Sec.~15.2.3]{MokPM19} is implemented in \morgen{}.}.
This friction term is principal to the accuracy of the gas pipeline model~\cite{OsiC10,Mis12,DomHLetal17,Hen18}.

In this model variant, a (globally) constant mean compressibility factor \linebreak $z_0 := z(p_0,T_0) \in \R$ is assumed~\cite{Osi96,DorT08,HerDDetal09,PfeFGetal15},
as well as a constant gas state $\gamma_0:= R_S T_0$,
whereas the temperature $T_0$ and the specific gas constant $R_S$ are treated as parameters (see~\cref{sec:param}).
To this end, the steady-state pressure $\bar{p} =: p_0$ is used to compute $z_0$,
via heuristic formulas based on the Virial expansion~\cite{Cha09}, \cite[Sec.~2.3]{BenGHetal18}\footnote{Additionally to~\cite{BenGHetal18},
the {DVGW-G-2000} equation~\cite[Ch.~9]{MisFH16} is implemented in \morgen{}.}.

\subsection{Homogenizing Scales}\label{sec:scal}
The SI-based units for pressure and mass-flux are [Pa] and [kg/s], respectively.
This introduces a difference in scales of five orders of magnitude between the variables $p$ and $q$,
and hence induces numerical problems.
To counter this multiscale structure, we simply rescale the pressure from [Pa] to [bar] which conveniently comprises a factor of $10^5$.
Nonetheless, the model still consists of two interacting physical variables, hence the model still has to be treated as a coupled system,
however, without numerical multiscale issues.

\pagebreak

\subsection{The Gas Network Model}\label{sec:network}
Given the model for a single pipe from the previous section, a (gas) network of pipes can be encoded by a finite directed graph,
which is a tuple $\mathcal{G} = (\mathcal{N},\mathcal{E})$ of finite sets symbolizing nodes $\mathcal{N}$, and oriented edges $\mathcal{E}$.
The edges correspond to pipes, while the nodes represent the junctions connecting pipes.
The connectivity of the network is the relationship between edges and junctions,
given by the incidence matrix $\mathcal{A} \in \{-1,0,1\}^{|\mathcal{N}| \times |\mathcal{E}|}$,
a map from edges to nodes, such that:
\begin{align*}
 \mathcal{A}_{ij} = \begin{cases} -1 & \mathcal{E}_j ~ \text{connects \textbf{from}} ~ \mathcal{N}_i, \\
                        \phantom{-}0 & \mathcal{E}_j ~ \text{connects \textbf{not}} ~ \mathcal{N}_i, \\
                        \phantom{-}1 & \mathcal{E}_j ~ \text{connects \textbf{to}} ~\mathcal{N}_i. \end{cases} 
\end{align*}
Note, that the orientation of the edges is not enforcing the dynamic flow direction of the gas,
but is necessary to determine the complexity and boundary of the overall networked model~\cite{morGruJHetal14,BenGHetal18}.

We introduce the notation $|\mathcal{A}|$ for the component-wise absolute value of a matrix.
Using this absolute value, the following partial incidence matrices associating edges entering and leaving nodes respectively are defined similar to a Heaviside function:
\begin{align*}
 \mathcal{A}_R := {\textstyle \frac{1}{2}} (\mathcal{A} + |\mathcal{A}|), \quad
 \mathcal{A}_L := {\textstyle \frac{1}{2}} (\mathcal{A} - |\mathcal{A}|).
\end{align*}

Next, based on this connectivity, certain conservation properties are enforced to maintain a network balance,
and thus ensure physical relevance of the gas network model.
Specifically, the Kirchhoff laws are applied to the network in vectorized (or rather matricized) form~\cite{Szo12,morBenBG18}:
\begin{enumerate}

 \item \emph{The sum of in- and outflows (mass-flux) at every node (junction) is zero:} \linebreak
       This means that no gas gets lost in transport from one pipe to the next, with the exception of boundary nodes.
       Hence, a vector of flows $q \in \R^{|\mathcal{E}|}$ applied to the incidence matrix equals the (out-)flow at the boundary (discharge) nodes $d_q : \R \to \R^{|\mathcal{N}_D|}$,
       which are mapped into the network via \linebreak $\mathcal{B}_d \in \{0,1\}^{|\mathcal{N}| \times |\mathcal{N}_D|}$:
  \begin{align*}
   \mathcal{A} \, q(t) = \mathcal{B}_d \, d_q(t),
  \end{align*}
       with $\mathcal{N}_D \subset \mathcal{N}$ denoting the subset of boundary nodes,
       which only connect \textbf{from one} node respectively, but \textbf{not to any} node.

 \item \emph{The sum of directed pressure drops in every fundamental loop is zero:} \linebreak
       Fortunately, an equivalent representation~\cite[Ch.~7.3]{Zer00} can be used, which resolves implicitly.
       It remains to ensure that the nodal pressures at the in-flow boundary (supply) nodes are associated to the boundary function \linebreak $s_p : \R \to \R^{|\mathcal{N}_S|}$,
       which are mapped to the network via $\mathcal{B}_s \in \{0,1\}^{|\mathcal{N}_S| \times |\mathcal{E}|}$:
  \begin{align*}
   \mathcal{A}_0^\T \, p(t) + \mathcal{B}_s^\T \, s_p(t) = A_{0,R}^\T \, p(t) - A_{0,L}^\T \, p(t),
  \end{align*}
       with $\mathcal{N}_S \subset \mathcal{N}$ denoting the subset of boundary nodes,
       which only connect \textbf{to one} node respectively, \textbf{but not from} any node,
       and the reduced incidence matrix $\mathcal{A}_0 \in \{-1,0,1\}^{|\mathcal{N}_0| \times |\mathcal{E}|}$, $|\mathcal{N}_0| = |\mathcal{N}| - |\mathcal{N}_S|$, with all rows associated to supply nodes removed.

\end{enumerate}

Given a connected and directed graph representing a gas network topology,
with the dynamic flow in the pipe edges modeled by the PDE~\eqref{eq:pipeline},
then yields a partial differential-algebraic equation (PDAE) due to the above constraints.

\pagebreak

\subsection{Discretization and Index Reduction}\label{sec:disc}
Next, we delineate the discretization of the spatial differential operators and reduction of the (P)DAE index in the networked system,
yielding a system of Ordinary Differential Equations (ODE).
Eventually, the remaining discretization of the temporal differential operators is addressed.

We explicitly do not use the decoupling approaches from \cite{morBanGB20} or \cite{morBanAGetal20},
as the former employs linearization and hence does not fit this setting,
while compared to the latter, our equivalent analytic index reduction is more convenient here.

The partial differential(-algebraic) equation is discretized using the method of lines:
First in space, then in time, yielding a (nonlinear) dynamic system.
For the spatial discretization a first-order upwind finite difference scheme is utilized \cite{ThoT87,AzePA19}.
We select (only) two points for each of the $k$ pipes with length $L_k$,
namely the start ($\,\,\cdot^R$) and end point ($\,\,\cdot^L$):
\begin{align*}
 \partial_x p_k &\approx \frac{p_k^R - p_k^L}{L_k}, && k\in \mathcal{E},\\
 \partial_x q_k &\approx \frac{q_k^R - q_k^L}{L_k}, && k\in \mathcal{E}.
\end{align*}
The matter of short, long and varying lengths $L_k$ is addressed in \cref{sec:tdisc}.

For each pipe, this leads to the following equations:
\begin{align*}
 \frac{1}{\gamma_0 \, z_0} \, \partial_t p_k^* ={}&-\frac{1}{S_k} \frac{q_k^R - q_k^L}{L_k}, \\
 \partial_t q_k^* ={}& -S_k \,\frac{p_k^R - p_k^L}{L_k} - \Big( \frac{S_k \, g \, (h_k^R - h_k^L)}{\gamma_0 \, z_0 \, L_k} p_k^*
                                                          + \frac{\gamma_0 \, z_0 \, \lambda_{0,k}}{2 \, d_k \, S_k} \frac{|q_k^*| \, q_k^*}{p_k^*}\Big).
\end{align*}
Now, different choices for $(\,\,\cdot^*)$ are surmisable.
Subsequently, two specific combinations of $p^*$ and $q^*$ will be discussed:
The midpoint discretization~\cite{morGruHKetal13,morGruJHetal14,ZloDBetal15,morGruHR16,morBenBG18},
and the left-right discretization~\cite{morGruJ15,morRog15,GruH20} resulting in (implicit) ODEs.
For an error analysis of these two discretization variants, see~\cite{StoM18}.
In the following, we describe a unified approach of deriving these index-reducible discretizations.

For notational ease in the coming subsections, a vectorized form for the above (networked) system including its constraints is given by:
\begin{align}
\begin{split}
 d_0 \, \partial_t p^* ={}& D_p (q^R - q^L), \\
        \partial_t q^* ={}& D_q (\mathcal{A}_0^\T p + \mathcal{B}_s^\T s_p) - \Big(D_q D_g \, d_0 \, p^* + D_f \, \frac{|q^*| \, q^*}{d_0 \, p^*}\Big), \label{eq:pcon} \\
     \mathcal{A}_0 q^* ={} &\mathcal{B}_d \, d_q,
\end{split}
\end{align}
using, thus resolving, the constraint $\mathcal{A}_0^\T p + \mathcal{B}_s^\T \, s_p ={} p^R - p^L$,
as well as \linebreak $d_0 := \frac{1}{\gamma_0 z_0} \in \R$ and the diagonal matrices:
\begin{align*}
 D_{p,kk} := -\frac{1}{S_k \, L_k}, \;\;
 D_{q,kk} := -\frac{S_k}{L_k}, \;\;
 D_{g,kk} := g \, (h_k^R - h_k^L), \;\;
 D_{f,kk} := \frac{\lambda_{0,k}}{2 \, d_k \, S_k}.
\end{align*}
Note, that $(d_0 \cdot p^*)$ corresponds to the global average density, $(S_k^{-1} \cdot q_k^*)$ to the local flow rate,
and depending on the choices for $p^*$ and $q^*$,
the model's analytic and numerical character will differ.

\pagebreak

\subsubsection{Midpoint Discretization}\label{sec:midpoint}
In case of the midpoint discretization, we set $p_k^*$ and $q_k^*$ to the mean of its associated edge's endpoints:
\begin{align*}
 p_k^* &= \frac{p_k^R + p_k^L}{2} =: p_k^+, \\
 q_k^* &= \frac{q_k^R + q_k^L}{2} =: q_k^+.
\end{align*}
Furthermore, we define $q^- := \frac{1}{2}(q^R - q^L)$,
and note, that in vectorized form, $p^+ = \frac{1}{2}(|\mathcal{A}_0^\T| p+|\mathcal{B}_s^\T| s_p)$.

Together with the algebraic constraints from \cref{sec:network},
a DAE system in the variables $p$, $q^+$, and $q^-$ arises:
\begin{subequations}\label{eq:midpoint}
\begin{align}
 d_0 \, {\textstyle \frac{1}{2}}(|\mathcal{A}_0^\T|\dot{p} +|\mathcal{B}^\T_s|\dot{s}_p) &= D_p \, 2 \, q^-, \label{eq:midpoint_a} \\
\begin{split}
 \dot{q}^+ &= D_q (\mathcal{A}_0^\T p + \mathcal{B}^\T_s s_p) - \Big( D_q D_g d_0 {\textstyle \frac{1}{2}}(|\mathcal{A}_0^\T| p + |\mathcal{B}^\T_s| s_p) \\
 	& \qquad\qquad\qquad\qquad\quad + D_f \frac{|q^+| \, q^+} {d_0 \frac{1}{2}(|\mathcal{A}_0^\T| p +|\mathcal{B}^\T_s| s_p)} \, \Big),
\end{split} \label{eq:midpoint_b} \\
 0 &= \mathcal{A}_0 q^+ + |\mathcal{A}_0|q^- -  \mathcal{B}_d d_q. \label{eq:midpoint_c} 
\end{align}
\end{subequations}
Since we aim to obtain an ODE, we need to transform this DAE system.
The complexity of deriving this transformation is quantified by the DAE's index.
From the various DAE index concepts, we use the tractability index $\tau$ \cite{Mae02},
for which the midpoint discretization guarantees $\tau \leq 2$ \cite{morGruJHetal14}.

This DAE can be decoupled into an ODE by rewriting it in the variables $p$ and $q^+$.
To this end, 
\begin{enumerate}

\item the pressure boundary condition implicitly resolves~\eqref{eq:pcon}.

\item By multiplying the differential equation~\eqref{eq:midpoint_a} by $(|\mathcal{A}_0| D_p^{-1})$ from the left,
the remaining algebraic constraint~\eqref{eq:midpoint_c} is removed by replacing $\mathcal{A}_0 q^-$ by $(-\mathcal{A}_0 q^+ + \mathcal{B}_d d_q)$ in~\eqref{eq:midpoint_a}.
Since $D_p$ is a diagonal matrix, this is also numerically feasible.

\end{enumerate}
We also pre-multiply \eqref{eq:midpoint_b} with the inverse of the diagonal matrix $D_q$.
Altogether, we obtain:
\begin{subequations}\label{eq:irmidpoint}
\begin{align}
 |\mathcal{A}_0| ({\textstyle \frac{1}{4}} D_p^{-1} d_0) |\mathcal{A}_0^\T| \,\, \dot{p}\phantom{^+} &= -\mathcal{A}_0 q^+ \! + \mathcal{B}_d d_q - |\mathcal{A}_0| ({\textstyle \frac{1}{2}} D_p^{-1} d_0) |\mathcal{B}^\T_s|\dot{s}_p, \\
\begin{split}
 D_q^{-1} \dot{q}^+ &= \;\; \mathcal{A}_0^\T \, p \; + \, \mathcal{B}^\T_s s_p - \Big( D_g d_0 {\textstyle \frac{1}{2}} (|\mathcal{A}_0^\T| p + |\mathcal{B}^\T_s| s_p)  \\
         & \qquad\qquad\qquad\quad\;\;\; + D_q^{-1} D_f \frac{|q^+| \, q^+}{d_0 {\textstyle \frac{1}{2}}(|\mathcal{A}_0^\T| p +|\mathcal{B}^\T_s| s_p)} \Big).
\end{split}
\end{align}
\end{subequations}
This system of a pressure and mass-flux variable now consists of only differential equations.
Notably, the first equation of the ODE system contains a temporal derivative of the input function $s_p$,
which practically would need to be approximated numerically, for example by finite differences.
However, we will assume that all inputs are sums of step functions, so that effectively $\dot{s}_p \equiv 0$,
which is reasonable as we assume exclusively low-frequency boundary values.

\pagebreak

\subsubsection{Endpoint Discretization}\label{sec:endpoint}
For the endpoint discretization, also called left-right discretization,
we set $p_k^*$ and $q_k^*$ to the left and right endpoints, respectively:
\begin{align*}
 p_k^* &= p_k^R, \\
 q_k^* &= q_k^L.
\end{align*}
Since $(\mathcal{B}_s^\T + |\mathcal{B}_s^\T|) \, s_p = 0$, we can write $p^R = \mathcal{A}_{0,R}^\T \, p$.
With the algebraic constraints from \cref{sec:network},
a DAE system in the variables $p$, $q^R$, and $q^L$ results:
\begin{subequations}\label{eq:endpoint}
\begin{align}
 d_0 \mathcal{A}_{0,R}^\T \, \dot{p} &= D_p (q^R - q^L), \label{eq:leftright_a} \\
\begin{split}
 \dot{q}^L &= D_q (\mathcal{A}_0^\T p + \mathcal{B}^\T_s p_s) - \Big( D_q D_g d_0 \, \mathcal{A}_{0,R}^\T \, p + D_f \frac{|q^L| \, q^L}{d_0 \mathcal{A}_{0,R}^\T \, p}\Big),
\end{split} \label{eq:leftright_b} \\
 0 &= \mathcal{A}_{0,R} \, q^R + \mathcal{A}_{0,L} \, q^L -  \mathcal{B}_d \, d_q. \label{eq:leftright_c}
\end{align}
\end{subequations}
As for the midpoint discretization, we want to derive a system of ODEs.
For the endpoint discretization, it is shown in~\cite{morGruJ15,morRog15},
that the tractability index is $\tau = 1$, if all edges connecting supply nodes are directed from the supply,
and each component of the graph is connected to at least one supply.
This implies, no two supplies are to be directly connected.
Similar to \cref{sec:midpoint}, this DAE can be decoupled into an ODE by rewriting it in the variables $p$ and $q^L$.
Applying equivalent steps to~\eqref{eq:endpoint} as for the midpoint decoupling~\eqref{eq:midpoint} yields:
\begin{subequations}\label{eq:irendpoint}
\begin{align}
\!\!\!\! (\mathcal{A}_{0,R} D_p^{-1} d_0 \, \mathcal{A}_{0,R}^\T)  \, \dot{p}\phantom{^L} &= -\mathcal{A}_0 \, q^L \! + \mathcal{B}_d \, d_q, \label{eq:odelr_a} \\
\begin{split}
 D_q^{-1} \dot{q}^L &= \mathcal{A}_0^\T \, p + \mathcal{B}^\T_s s_p - \Big( D_g d_0 \, \mathcal{A}_{0,R}^\T \, p + D_q^{-1} D_f \frac{|q^L| \, q^L}{d_0 \mathcal{A}_{0,R}^\T \, p} \Big),
\end{split} \label{eq:odelr_b}
\end{align}
\end{subequations}
using $\mathcal{A}_{0,R} + \mathcal{A}_{0,L} = \mathcal{A}_0$ in~\eqref{eq:odelr_a}.

An advantage of this endpoint discretization, in addition to the absence of a derivative of an input function,
is the input-free friction term in~\eqref{eq:odelr_b}.

\subsubsection{Temporal Discretization}\label{sec:tdisc}
After spatial discretization and index reduction of the gas network model, a system of stiff nonlinear ODEs (in time) remains.
The remaining temporal differential operator(s) are discretized using time-stepping schemes.

Due to the hyperbolicity of the pipeline model,
the temporal resolution $\Delta t$ used for the discrete integration interdepends on the spatial resolution $\Delta x$
employed for the discretization of the spatial differential operators.
Formally, this is expressed by the Courant-Friedrichs-Levy (CFL) condition~\cite{HelMY14,Ban15,WiilC20},
which states that the propagation of information cannot be faster than its conveyor:
\begin{align*}
 \lambda_{\text{CFL}} := v_{\text{max}} \frac{\Delta t}{\Delta x} < 1,
\end{align*}
with the (dimensionless) CFL constant $\lambda_{\text{CFL}}$,
symbolizing the ratio of temporal and spatial discretization step-width scaled by the peak gas velocity $v_{\text{max}}$.
Since the flow is subsonic, $v_{\text{max}}$ could be estimated from the (linearized) characteristics~\cite{HelMY14},
or via the boundary values\footnote{e.g.~{\scriptsize~\url{https://petrowiki.org/Pipeline_design_consideration_and_standards\#Gas_line_sizing}}}.
However, we fix this maximum gas speed to $v_{\text{max}} = 20\frac{\operatorname{m}}{\operatorname{s}}$ (practically this is configurable in \morgen{}).

\pagebreak

% Refinement
Due to this relation of the space and time discretization and a pre-selected application-specific sampling frequency $\Delta t$ of the output trajectory (i.e., every~$60$s),
$\Delta x$ has to be adapted accordingly.
The spatial discretization by finite differences in the previous sections ignores pipeline length,
as each pipe is assigned only one (finite) difference.
This means, pipes are potentially too long or too short with respect to a nominal length $\Delta x$,
determined by the CFL condition $\Delta x = (1-\varepsilon) \, v_{\text{max}} \, \Delta t$, ($0 < \varepsilon \ll 1$).
Thus, too long pipes are subdivided into virtual pipes of nominal length,
while too short pipes, including a potential remainder of too long pipes,
are ``rounded'' to a full nominal-length pipe,
yet with a friction term scaled by the fraction of the short pipe's length compared to the nominal length,
\begin{align*}
 \widetilde{D}_{f,kk} = \begin{cases} D_{f,kk} & L_k = \Delta x, \\ \frac{L_k}{\Delta x} D_{f,kk} & L_k < \Delta x.\end{cases}
\end{align*}
This approach assumes that delays due to the forced virtual length of an actually short pipe is insignificant,
hence this simple homogenization of pipe lengths may be improved by replacing short pipes with friction-less shortcuts and a static pressure-drop,
as used in the quasi-static model~\cite{HerMS10,morGruHKetal13} and similar to the subsequent compressor model in \cref{sec:comp}.

Overall, we refine each pipe into a sequence of pipes of a selected nominal length -- a graph level refinement --
which is determined using the CFL condition.
We note here, that this methodology is aimed at ensuring a certain minimum length for each pipe,
as the shortest pipe may dictate an unnecessarily finely resolved time discretization.
In terms of a pipe's maximum length, suggestions are for example: $5$km (\cite{morGruHKetal13}),
or $10$km (\cite{ZloCB15}).

As the model is composed of a stiff, linear (hyperbolic) and a nonlinear component,
an implicit solution of the linear part using a diagonally implicit Runge-Kutta (DIRK) method,
and an explicit solution of the nonlinear part via an strong stability preserving (SSP) method,
by an IMEX (IMplicit-EXplicit) solver, as proposed in~\cite[Sec.~3.2.2]{Ban15}, is targeted.
The actual quadrature rules used to compute the transient solutions are detailed in \cref{sec:solvers}.

\subsection{Simplified Compressors}\label{sec:comp}
Beyond pipes, gas networks comprise a variety of non-pipe elements, of which the most important are compressors.
Compressors increase gas pressure to counteract cumulative effects of retarding forces (friction, gravity, inertia etc.),
and are grouped into stations with many possible configurations.
For our purposes, we just allow fixed configurations on a macro scale~\cite{MakVZetal19} per scenario,
which leads to a compressor being modeled as a special kind of edge that boosts the pressure from its suction inlet to its discharge outlet.

Compressors are typically modeled ``ideally'', based on their power consumption, for example as a special node type; in-depth discussions can be found in~\cite{Her07,SchSW15}.
Due to this consumption model, such ideal compressor units are useful for energy utilization optimization tasks~\cite{EhrS05},
yet, for a simplified transient simulation aspect already too complicated.
A more practical approach is taken in~\cite{ZloCB15,ZloDBetal15,MakVZetal16}, where a compression ratio $\alpha_i \geq 1$ is used to scale the pressure in each node,
or pipe~\cite{SunZ19}, which means $\alpha_i > 1$ indicates compression / a compressor, otherwise ($\alpha_i = 1$) a pipe.

Here, we utilize a likewise simple compressor model similar to~\cite{VufMC15,GolRT76},
for which we assume it is propelled by an external energy source, for example, given a compressor electrification, by excess renewable power~\cite{EriEA19},
or that the off-take in gas is insignificant.
But instead of using compression ratios, a constant (or parametric) target pressure is prescribed,
modeling \emph{discharge pressure control}~\cite{StrW70}.

The following affine compressor model is a variant of the compressor presented in~\cite{Ste07} and used in~\cite{AzeJ07}.
We model a simplified compressor by a level, short pipe which increases the pressure at its outflow to a specified value $\bar{p}_c$ (and without friction, $\lambda_{ij} \equiv 0$).
Given the pipe from nodes $i$ to $j$ is treated as such a ``compressor pipe'', with target pressure of $\bar{p}_c$,
the variables $p_{ij}$ and $q_{ij}$ are given by the differential equations:
\begin{align}\label{eq:comp}
\begin{split}
 \dot{p}_{ij}(t) &= q_j - q_i, \\
 \dot{q}_{ij}(t) &= \bar{p}_c - p_i.
\end{split}
\end{align}
The target pressure $\bar{p}_c$ could be a step function $\bar{p}_c(t)$, and hence a control input~\cite{SunZ19},
which could be accompanied by a discharge mass-flux output.

A compressor could also be interpreted as an actual pipe with ``negative friction'',
and we considered using such nonphysical pipes as compressor model,
but a difficult transformation between friction and pressure increase would have to be calibrated for every model variant (including friction factor formulas) and updated with every change in any model.

\subsection{Parametrization}\label{sec:param}
For the considered pipeline model~\eqref{eq:pipeline}, two scalar parameters are of practical interest:
The temperature of the gas $T_0$ (in [K]), which is assumed to be constant throughout the network,
and the global specific gas constant $R_S$ (in [J/(kg\,K)]).

Due to mainly underground on-shore pipelines~\cite[Ch.~45]{MisFH16}, and coolers in compressor stations~\cite{MakVZetal19},
using an isothermal model is a reasonable simplification.  
However, temperature is relevant as a global parameter, since the use of an isothermal model ``freezes'' the dynamic energy (temperature) component in the original Euler equations (in time),
and while intraday ambient temperature variation can be neglected for simulations with a $24$-hour time horizon,
the temperature difference of a hot summer day and a cold winter day
should be taken into account by a parameter representing an average temperature.

The specific gas constant, on the other hand, is determined by the gas composition, which may also vary.
Again, local variations during an intraday simulation are neglected in this work, yet the average gas mixture of natural gas with,
for example, hydrogen or bio-gas is relevant, so a parameter for the average specific gas constant is introduced.
Together, the parameter-space $\Theta$ is given by:
\begin{align*}
 \theta := \begin{pmatrix} T_0 \\ R_S \end{pmatrix} \in \Theta \subset \R^2,
\end{align*}
and note that $\theta$ is used in the model (only) as $d_0 = \frac{1}{\gamma_0 z_0} = \frac{1}{(T_0 R_S) z_0} = \frac{1}{\theta_0 \theta_1 z_0}$.
Yet, lumping into a single parameter would impede physical interpretation.

Applied to the respective components of the input-output model of \cref{sec:iomodel}, this leads to parameter-dependent quantities,
which need to be regarded accordingly by the model reduction as discussed in \cref{sec:pmor}.

\pagebreak

\subsection{Input-Output Model}\label{sec:iomodel}
After spatial discretization and index reduction, we end up with a square input-output system,
a system with the same number of inputs and outputs, consisting of an ordinary differential equation, an output function and an initial value:
\begin{align}\!\!
 \setlength\arraycolsep{1pt}\overbrace{\begin{pmatrix} E_p(\theta) & 0 \\ 0 & E_q \end{pmatrix}}^E \!\! \begin{pmatrix} \dot{p} \\ \dot{q} \end{pmatrix} &\!=\!
 \setlength\arraycolsep{1pt}\overbrace{\begin{pmatrix} 0 & A_{pq} \\ A_{qp} & 0 \end{pmatrix}}^A \!\! \begin{pmatrix} p \\ q \end{pmatrix} \!+\!
 \setlength\arraycolsep{1pt}\overbrace{\begin{pmatrix} 0 & B_{pd} \\ B_{qs} & 0 \end{pmatrix}}^B \!\! \begin{pmatrix} s_p \\ d_q \end{pmatrix} \!+\!
 \overbrace{\begin{pmatrix} 0 \\ F_c \end{pmatrix} \!+\! \begin{pmatrix} 0 \\ f_q(p,q,s_p,\theta) \end{pmatrix}}^f, \notag \\
 \begin{pmatrix} s_q \\ d_p \end{pmatrix} &\!=\! 
 \setlength\arraycolsep{1pt}\underbrace{\begin{pmatrix} 0 & C_{sq} \\ C_{dp} & 0 \end{pmatrix}}_C \!\! \begin{pmatrix} p \\ q \end{pmatrix}, \label{eq:iosys} \\
 \begin{pmatrix} p_0 \\ q_0 \end{pmatrix} &\!=\! \begin{pmatrix} \bar{p}(\bar{s}_p,\bar{d}_q) \\ \bar{q}(\bar{s}_p,\bar{d}_q) \end{pmatrix}, \notag
\end{align}
with parameter independent linear vector field components $A$ and $B$, parametric mass matrix $E(\theta)$,
and nonlinear friction and gravity retarding term $f_q(p,q,s_p,\theta)$.
The actual composition of the dynamical system components depends on the discretization and index reduction, cf.~\cref{sec:disc},
while the linear output function $C$ consists of $C_{sq} = -\mathcal{B}_s$ and $C_{dp} = \mathcal{B}_d^\T$.
The load vector $F_c \in \R^{N_q}$ accumulates the respective discharge pressures $\bar{p}_c$, as described in \cref{sec:comp}, for all compressors,
and the initial state is given by a steady state, whose computation is detailed in \cref{sec:steady}, depending on given steady state boundary values $\bar{s}_p$, $\bar{d}_q$.
Altogether, the gas network model is a generalized linear system $(E,A,B,C)$, together with a nonlinear part $f$.

A control system formulation of transient gas network simulation was already formulated in~\cite{KraSVetal84}, and recently in~\cite{WiilC20}.
Also in~\cite{DorT08}, a so-called ``systemic interpretation'' is discussed;
inspired by~\cite[Figure~2]{DorT08}, we schematically illustrate~\eqref{eq:iosys} in \cref{fig:schematic}.

The specific structure and grouping for the single pipeline model~\eqref{eq:pipeline},
the index-reduced spatially discrete network models~\eqref{eq:irmidpoint}, \eqref{eq:irendpoint},
and the input-output \linebreak model~\eqref{eq:iosys}, is justified by the numerical processing:
only two model components depend on the parameters, temperature $T_0$ and specific gas constant $R_S$,
as well as on the compressibility factor $z_0$, namely the mass matrix $E_p$ and the jointly treated retarding forces gravity and friction $f_q$.
Hence, the linear part of the right-hand side vector field is non-parametric and compressibility-agnostic.

Overall, this system maps input boundary values, in the scope of this work, pressure at the inlets and mass-flux at the outlets,
via the internal state, to output quantities of interest, here, mass-flux at the inlets and pressure at the outlets:
\begin{align*}
 \begin{pmatrix} s_p \\ d_q \end{pmatrix} \;\; & \stackrel{S}{\longrightarrow} \;\; \begin{pmatrix} s_q \\ d_p \end{pmatrix} \\
 \raisebox{.75ex}{\text{$\searrow$}\;} &\begin{pmatrix} p \\ q \end{pmatrix} \raisebox{.75ex}{\;\text{$\nearrow$}}
\end{align*}

To this type of input-output system we can now apply (data-driven) system-theoretic model reduction methods,
which preserve the input-to-output mapping~$S$, but explicitly not the internal state $\begin{pmatrix} p & q\end{pmatrix}^\T$.
Lastly, we note that based on~\cite{HueMGetal18}, we added a model fact sheet in the \nameref{sec:appendix}.

\pagebreak

\begin{figure}[h!]\centering
\includegraphics[height=.28\textheight]{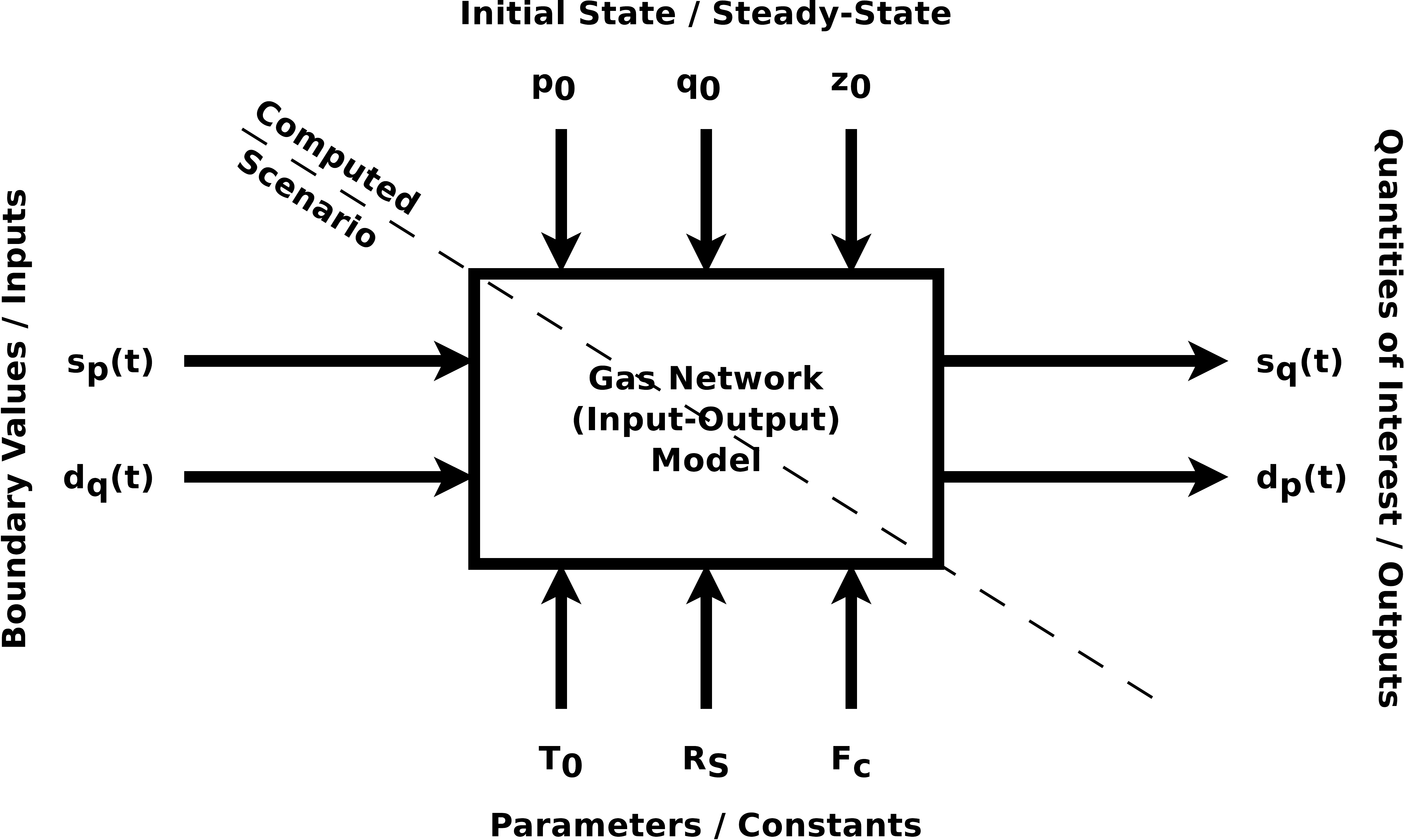}
\caption{Schematic illustration of gas network input-output model.}
\label{fig:schematic}
\end{figure}

\subsection{Steady-State Computation}\label{sec:steady}
After spatial discretization, the dynamic simulation becomes an initial value problem.
Yet, only the boundary values of the network model are known a-priori.
This means the internal state at time $t=0$ is unknown.
We assume simulations always start at a steady-state $\bar{p}$, $\bar{q}$ for which $\partial_t p = \partial_t q = 0$,
given some (initial) boundary values $\bar{s}_p$, $\bar{d}_q$.
The internal state is then computable as a steady-state problem.
Since the employed model is nonlinear, we approximate the steady-state by a two-step procedure:
\begin{enumerate}[leftmargin=2em]

 \item[1a.] Linear mass-flux steady-state: $A_{pq}\,\bar{q} = -B_{pd}\,\bar{d}_q$.

 \item[1b.] Linear pressure steady-state: $A_{qp}\,\hat{p} = -\Big(B_{qs}\,\bar{s}_p + F_c\Big)$.

 \item[2.\phantom{b}] Corrected pressure steady-state: $A_{qp}\,\,\bar{p} = -\Big(B_{qs}\,\bar{s}_p + F_c + f_q(\hat{p},\bar{q},\bar{s}_p,\theta)\Big)$.

\end{enumerate}
Step~2 can be repeated until an error threshold is met by using the previously approximated pressure steady-state.
Practically, the linear problems in Step~1 and Step~2 are solved by a QR-based least-norm method~\cite{Boy08}.
Note, that Step~1a and Step~1b can be solved in parallel and that the QR decomposition of Step~1b can be recycled in Step~2 because of the chosen model structure.

While this method works well for rooted-tree pipe-networks,
it is not sufficient for cyclic networks with multiple supply nodes and non-pipe elements such as compressors.
In this case, the resulting state after a limited number of the above algorithm's iterations 
is used as an initial value for the first order IMEX integrator detailed in \cref{sec:imex1},
which time-steps until a steady-state is sufficiently approximated.
This approximate steady-state, associated to a fixed set of boundary values and parameters,
is used as initial value for the simulations:
\begin{align*}
 \begin{pmatrix} p_0 \\ q_0 \end{pmatrix} &= \begin{pmatrix} \bar{p}(\bar{s}_p,\bar{d}_q) \\ \bar{q}(\bar{s}_p,\bar{d}_q) \end{pmatrix}.
\end{align*}
While other time steppers are applicable,
the first order IMEX solver is related to the initial (two-step) algebraic approximation,
due to the synthesis of the linear/input/source and nonlinear/reaction terms.

\pagebreak % TEMP

\subsection{Port-Hamiltonian Structure}\label{sec:ph}
An interesting class of models are port-Hamiltonian systems,
which have \linebreak already been used for gas network modeling~\cite{morLilM17,morEggKLetal18}.
Such port-controlled-Hamiltonian models result from a system-theoretic approach to energy-based modeling,
and are square, passive, stable and feature certain symmetries, besides their physical interpretability~\cite{OrtVMetal01,BeaMXetal18}.
To exploit results from port-Hamiltonian theory in the context of data-driven model reduction,
we regiment the previous modeling approach into the port-Hamiltonian framework.

A linear input-state-output port-Hamiltonian model~\cite[Ch.~4]{VanJ14} has the form:
\begin{align}\label{eq:ph}
\begin{split}
 E \, \dot{x}(t) &= \overbrace{(J - R)\,\; Q}^A x(t) + \overbrace{(G - P)}^B \, u(t), \\
            y(t) &= \underbrace{(G + P)^\T Q}_C x(t),
\end{split}
\end{align}
with a symmetric positive definite mass matrix $E = E^\T$, $E > 0$,
a skew-symmetric energy flux $J = -J^\T$,
a symmetric, positive, semi-definite energy dissipation $R = R^\T$, $R \geq 0$,
a symmetric, positive definite energy storage $Q$, $Q > 0$,
resistive port matrix $P$ and control port matrix $G$\footnote{Typically, the symbol $B$ is used for this port matrix.}.

Here, we generalize the energy dissipation $R \in \R^{N \times N}$ to a nonlinear mapping $R : \R^N \to \R^{N \times N}$,
this means the linear constraints become~\cite{EggG20}:
\begin{align}
 R &= R^\T \;\; \rightarrow \;\; \langle R(x) x', x'' \rangle = \langle x', R(x) x'' \rangle, \quad \forall x, x', x'' \in \R^N \label{eq:res1}, \\
 R &\geq 0 \;\; \;\;\; \rightarrow \;\; \;\;\; \langle R(x) x', x' \rangle \geq 0, \quad \forall x, x' \in \R^N \label{eq:res2}.
\end{align}

With this set up, we test the two index-reduced gas network model discretizations presented in \cref{sec:endpoint} and \cref{sec:midpoint}
for compliance with the above port-Hamiltonian properties.
\begin{proposition}
 The endpoint discretization \eqref{eq:irendpoint} is a port-Hamiltonian model.
\end{proposition}

\begin{proof} ~\\
We define the port-Hamiltonian state as $x := \begin{pmatrix} p & q^L \end{pmatrix}^\T$, which induces the remaining components.
The mass matrix
\begin{align*}
 E = \begin{bmatrix} (\mathcal{A}_{0,R} D_p^{-1} d_0(\theta) \, \mathcal{A}_{0,R}^\T) & 0 \\ 0 & D_q^{-1} \end{bmatrix} = E^\T > 0
\end{align*}
is symmetric positive definite, if its diagonal blocks are.
Given that the $D_*$ are diagonal, and thus symmetric, as well as positive definite, both blocks are symmetric positive definite.
The energy flux
\begin{align*}
 J = \begin{bmatrix} 0 & -\mathcal{A}_0 \\ \mathcal{A}_0^\T & 0 \end{bmatrix} = -J^\T
\end{align*}
is skew-symmetric by definition.
The energy dissipation (see $f_q$)
\begin{align*}
 R(x) := \begin{bmatrix} 0 & 0 \\ 0 & \big( D_g \, \diag\big(\frac{\scriptstyle d_0(\theta) \mathcal{A}_{0,R}^\T \, p}{\scriptstyle q^L}\big) + D_q^{-1} D_f \diag\big(\frac{\scriptstyle |q^L|}{\scriptstyle d_0(\theta) \mathcal{A}_{0,R}^\T \, p}\big)\big)  \end{bmatrix},
\end{align*}
with the $\diag : \R^N \to \R^{N \times N}$ operator mapping a vector $v$ to a diagonal matrix $D$ such that $v_k \mapsto D_{kk}$,
and element-wise (fraction) nonlinearities,
results in one non-zero diagonal block and thus fulfills \eqref{eq:res1}.
The condition \eqref{eq:res2} is fulfilled since in the friction term of $R$,
the absolute value of the mass-flux (numerator),
and the nodal pressure variable (denominator) are always non-negative.

Here, the energy storage represents the scale homogenization from \cref{sec:scal},
\begin{align*}
 Q = \begin{bmatrix} (10^5 \cdot I_{N_p}) & 0 \\ 0 & (10^{-5} \cdot I_{N_q}) \end{bmatrix} = Q^\T > 0
\end{align*}
which is a diagonal matrix of positive entries,
and due to same block structure in $E$ also fulfills $Q^\T E = E^\T Q$.
Lastly, the port matrix configuration
\begin{align*}
 P := \begin{bmatrix} 0 & 0 \\ B_S^\T & 0 \end{bmatrix}, \quad
 G := \begin{bmatrix} 0 & B_D \\ 0 & 0 \end{bmatrix},
\end{align*}
complies to the port-Hamiltonian form.
\end{proof}

Some remarks are in order on this result:
From the previous proof it is also immediately clear that the midpoint discretization cannot be a port-Hamiltonian model,
due to the input dependence of the energy dissipation.
Furthermore, this derivation tests if the endpoint discretization has the mathematical port-Hamiltonian structure,
but does not verify a physical energy-based model.

The somewhat nonphysical treatment of the gravity term as dissipating instead of storing (\cite{EggG20}),
is done with regard to the parametrization.
Including the parametric gravity term as a retarding or damping force,
and thus keeping the linear energy flux parameter-free, enables the previous steady-state computation.

Compressors, as modeled in \cref{sec:comp} can be included by an additional summand inside the energy dissipation component,
i.e. $\frac{F_C}{q}$, similar to the gravity term.
This exhibits an unphysical negative sign inside the dissipation, as a compressor introduces energy.
Furthermore, this compressor model requires to remove components from the $A_{qp}$ block of the system matrix \eqref{eq:comp},
and thus perturbs the skew-symmetry of $J$.

Lastly, this notation for the dissipation can also be used for linearization,
by constraining the argument of $R$ to the steady state~$\bar{x}$,
\begin{align*}
 \widetilde{R} := R(\bar{x}) \approx R(x).
\end{align*}

Given the port-Hamiltonian model with a nonlinear resistive term,
an (approximate) adjoint system can be derived by treating $R$ as its image -- a diagonal matrix.
Transposing the (primal) port-Hamiltonian system's \eqref{eq:ph} transfer function $h(s) = (G + P)^\T Q (E s - (J - R) Q)^{-1} (G - P)$,
and exploiting the system components properties, yields the dual system:
\begin{align}\label{eq:dual}
\begin{split}
 E \, \dot{x}(t) &= \overbrace{Q \big(-J - R\big(x(t)\big)\big)}^{A^\T} x(t) + \overbrace{Q (G + P)}^{C^\T} u(t), \\
            y(t) &= \underbrace{(G - P)^\T}_{B^\T} x(t).
\end{split}
\end{align}
Hence, for the (nonlinear) endpoint discretization, its observability can be (approximately) quantified by the dual system's reachability, as for linear systems.
Conceptually, this could also be done with the midpoint discretization, as it supplies the same model components.
However, it has no theoretical justification, as a dual system may not be accessible for (general) nonlinear systems.

\pagebreak % TEMP

%%%%%%%%%%%%%%%%%%%%%%%%%%%%%%%%%%%%%%%%%%%%%%%%%%%%%%%%%%%%%%%%%%%%%%%%%%%%%%%%
%%%%%%%%%%%%%%%%%%%%%%%%%%%%%%%%%%%%%%%%%%%%%%%%%%%%%%%%%%%%%%%%%%%%%%%%%%%%%%%%
%%%%%%%%%%%%%%%%%%%%%%%%%%%%%%%%%%%%%%%%%%%%%%%%%%%%%%%%%%%%%%%%%%%%%%%%%%%%%%%%

\section{Model Order Reduction for Gas Networks}\label{sec:reduction}
In this section, we summarize the principal approach behind all presented model reduction methods that are extended and tested in this work.
The structure of the model laid out in \cref{sec:model} is given by~\eqref{eq:iosys}.
For large (expansive) networks, the differential equations in $p$ and $q$ become high dimensional,
which impedes their solution \cite[Sec.~7]{GugH20} and hence repeated simulations of scenarios.
The aim of model reduction is to reduce the dimensionality of the differential equations,
by computing subspaces of the phase space on which the trajectories evolve suitably similar (with regard to the quantities of interest).
Furthermore, the reduced order model shall have the same form as the original model,
and since two physical quantities are (bi-directionally) coupled in this system,
the model reduction for interconnected systems~\cite{morSanM09} approach is used, yielding reduced operators for each subsystem:
\begin{align*}\!\!
 \setlength\arraycolsep{0pt}\begin{pmatrix} \!\widetilde{E}_p(\theta) & 0 \\ 0 & \widetilde{E}_q \end{pmatrix} \!\! \begin{pmatrix} \dot{\tilde{p}} \\ \dot{\tilde{q}} \end{pmatrix} &\!\!=\!\!
 \setlength\arraycolsep{0pt}\begin{pmatrix} 0 & \widetilde{A}_{pq}\!\! \\ \!\widetilde{A}_{qp} & 0 \end{pmatrix} \!\! \begin{pmatrix} \bar{p} \splus \tilde{p} \\ \bar{q} \splus \tilde{q} \end{pmatrix} \!\!+\!\!
 \setlength\arraycolsep{0pt}\begin{pmatrix} 0 & \widetilde{B}_{pd}\!\! \\ \!\widetilde{B}_{qs} & 0 \end{pmatrix} \!\! \begin{pmatrix} s_p \\ d_q \end{pmatrix} \!\!+\!\!
 \begin{pmatrix} 0 \\ \widetilde{F}_c \end{pmatrix} \!\!+\!\! \begin{pmatrix} 0 \\ \tilde{f}_q(\bar{p} \splus \tilde{p},\bar{q} \splus \tilde{q},\!s_p,\!\theta)\!\!\end{pmatrix}\!\!, \\
 \begin{pmatrix} s_q \\ d_p \end{pmatrix} \approx
 \begin{pmatrix} \tilde{s}_q \\ \tilde{d}_p \end{pmatrix} &\!\!=\!\!
 \setlength\arraycolsep{0pt}\begin{pmatrix} 0 & \widetilde{C}_{sq}\! \\ \!\widetilde{C}_{dp} & 0 \end{pmatrix} \!\! \begin{pmatrix} \bar{p} \splus \tilde{p} \\ \bar{q} \splus \tilde{q} \end{pmatrix}, \\
 \begin{pmatrix} \tilde{p}_0 \\ \tilde{q}_0 \end{pmatrix} &\!\!=\!\! \begin{pmatrix} 0 \\ 0 \end{pmatrix},
\end{align*}
centered around the steady-state $\begin{pmatrix} \bar{p} & \bar{q}\end{pmatrix}^\T$.
This structure preserving model order reduction was already used in~\cite{morGruJHetal14,morBenBG18} in the context of model reduction for gas networks,
while the centering has been used in~\cite{AzeJ07} for gas network simulation and in~\cite{morHim17} for nonlinear model order reduction.
In the following, the general ansatz to obtain these reduced quantities (denoted by $\tilde{\,\cdot\,}$) is summarized.

\subsection{Projection-Based Model Reduction}
The reduced order model is computed by projecting the high-dimensional dynamics evolving in the (coupled) pressure and mass-flux phase spaces (of dimension $N_p$ and $N_q$)
to low(er)-dimensional subspaces (of dimension $n_p$ and $n_q$),
which capture the principal components of the respective trajectories.
Given suitable discrete projection mappings from the original space to the reduced space $V_*^\T$ and mappings from the reduced space back to the original space $U_*$:
\begin{align*}
 U_p : \R^{n_p} \to \R^{N_p}, \quad V_p^\T : \R^{N_p} \to \R^{n_p}: \quad V_p^\T \cdot U_p = \id_{n_p}, \\
 U_q : \R^{n_q} \to \R^{N_q}, \quad V_q^\T : \R^{N_q} \to \R^{n_q}: \quad V_q^\T \cdot U_q = \id_{n_q}.
\end{align*}
Thus, the reduced trajectory results from applying $V_*$ to the original trajectory's steady-state deviation,
while the original trajectory is approximately recovered by applying $U_*$ to the reduced trajectory:
\begin{align*}
 \begin{pmatrix} \tilde{p} \\ \tilde{q} \end{pmatrix} \coloneqq \begin{pmatrix} V_p^\T \, (p - \bar{p}) \\ V_q^\T \, (q - \bar{q}) \end{pmatrix} \to
 \begin{pmatrix} \bar{p} + U_p \, \tilde{p} \\ \bar{q} + U_q \, \tilde{q} \end{pmatrix} \approx \begin{pmatrix} p \\ q \end{pmatrix};
\end{align*}
the initial condition is also reduced by application of $V_*$.
Similarly, the components of the reduced system result from applying the $U_*$ map to the argument of the respective operators,
and the $V_*$ map to the result of the operation.

For the linear operators, the matrices $E_*$, $A_*$, $B_*$ and $C_*$ and the vector $F_c$,
this leads conveniently to pre-computable reduced matrices and vector respectively,
\begin{align*}
 \widetilde{A}_{pq} &\coloneqq V_p^\T \cdot A_{pq} \cdot U_q \in \R^{n_p \times n_q}, \: && \widetilde{A}_{qp} \coloneqq V_q^\T \cdot A_{qp} \cdot U_p \in \R^{n_q \times n_p}, \\
 \widetilde{B}_{pd} &\coloneqq V_p^\T \cdot B_{pd} \in \R^{n_p \times N_s}, \: && \widetilde{B}_{qs} \coloneqq V_q^\T \cdot B_{qs} \in \R^{n_q \times N_d}, \\
 \widetilde{C}_{dp} &\coloneqq C_{dp} \cdot U_p \in \R^{N_d \times n_p}, \: && \widetilde{C}_{sq} \coloneqq C_{sq} \cdot U_q \in \R^{N_s \times n_q}, \\
 \widetilde{E}_p(\theta) &\coloneqq V_p^\T \cdot E_p(\theta) \cdot U_p \in \R^{n_p \times n_p}, && \widetilde{E}_q \coloneqq V_q^\T \cdot E_q \cdot U_q \in \R^{n_q \times n_q}, \\
 \widetilde{F}_c &\coloneqq V_q^\T \cdot F_c \in \R^{n_q},
\end{align*}
yet, the nonlinear component $\tilde{f}_q$ remains a composition operation:
\begin{align}\label{eq:nrom}
\tilde{f}_q &\coloneqq V_q^\T \cdot f_q(\bar{p} + U_p \, \tilde{p} \,, \bar{q} + U_q \, \tilde{q} \,,s_p \,,\theta) : \R^{n_p} \times \R^{n_q} \times \R^{N_s} \times \R^2 \to \R^{n_q}.
\end{align}

\subsection{Structure Preserving Model Order Reduction}
In this specific context, the term structure preserving model order reduction (SPMOR) has two meanings:
first and foremost, SPMOR refers to the separate reduction of the state components, as above in the case of gas networks,
the individual reduction of the discretized pressure $p$ and mass-flux $q$ variables.
Second, SPMOR can also refer to preserving the port-Hamiltonian form \eqref{eq:ph}.
For projection-based model reduction, the former is generally ensured by separate projectors \cite{morFre04} (or an overall block diagonal projection).
The latter is guaranteed by using Galerkin projections,
which implies stability preservation \cite{morBeaG11}, given a port-Hamiltonian full order model.
Both SPMOR interpretations are jointly fulfilled if a block-diagonal (w.r.t. $p$ and $q$) Galerkin projection is used.

\subsection{The Lifting Bottleneck and Hyper-Reduction}\label{sec:hypred}
The gas network models considered for reduction are nonlinear (and potentially non-smooth),
hence the reduced order nonlinear part $\tilde{f}_q$ involves lifting the reduced state up to the original high-dimensional space,
evaluating the nonlinearity and projecting the result back down to the reduced low-dimensional space~\eqref{eq:nrom}.
As the high-dimensional space is involved, this is typically computationally demanding and may eat up the gains from the reduction of the linear part.
To mitigate this so-called lifting bottleneck, hyper-reduction methods can be employed, which construct low-dimensional surrogates for nonlinearities.

In this work we discard (or rather defer) hyper-reduction due to the following reasoning:
The purpose of this work is to determine which method constructs the best reduced order models,
a hyper-reduction may inhibit comparability due to, for example, a dominating hyper-reduction approximation error.
Second, various hyper-reduction methods for this setting are applicable (i.e. DEIM~\cite{morChaS10}, Q-DEIM~\cite{morDrmG16}, DMD~\cite{morWilSK13} or numerical linearization~\cite{morMoo79}),
which may interact differently with the different model reduction methods.
So as a first step, the bare model reduction methods are tested here (this means: which method's linear subspaces capture the nonlinear dynamics best),
at a later stage the best hyper-reduction method can then be determined.
Lastly we note, the nonlinear part of the vector field consists exclusively of element-wise operations (see \cref{sec:ph}),
a system with repeated scalar nonlinearities (SRSN)~\cite{morChuG99}, which are less difficult to handle due to ``locality'' of the nonlinearity,
and hence, its vectorization.

\subsection{Parametric Model Reduction}\label{sec:pmor}
There are two common approaches for projection-based parametric model order reduction: averaging and accumulating~\cite{morHim21}.
For the selected data-driven methods, \emph{averaging} means that for a set of parameter samples the associated trajectories or derived quantities (such as the utilized system Gramians) are averaged,
while \emph{accumulating} refers to the concatenation of trajectories or derived quantities (such as the projectors).

Generally, each of the structure-preserving model order reduction methods in \cref{sec:methods} can be used with either,
we opted to use the averaging ansatz for all of the following methods since their computation is without exception based on parametric empirical Gramians~\cite{morHimO15a}.

\section{Model Reduction Methods}\label{sec:methods}
In this section, we briefly summarize the employed model reduction methods from a practical point of view.
For theoretical details and backgrounds we refer to the relevant works, cited in the respective subsections.
Due to the non-differentiable nonlinearity (friction), the sought projections $U_*$, $V_*$ for all tested model reduction techniques are constructed from (transformed) time-domain trajectory data obtained from numerical simulations,
which is given by discrete-time snapshots of the internal pressure nodes $\widehat{X}_p$ and mass-flux edges~$\widehat{X}_q$,
\begin{align*}
 \widehat{X}_p(t;\theta_k) &= \begin{bmatrix} \; p^1(t;\theta_k) & \dots & p^{N_s + N_d}(t;\theta_k) \;\;\end{bmatrix}, \\
 \widehat{X}_q(t;\theta_k) &= \begin{bmatrix} \; q^1(t;\theta_k) & \dots & q^{N_s + N_d}(t;\theta_k) \;\;\end{bmatrix},
\end{align*}
the external demand node pressure~$\widehat{Y}_p$ and supply node mass-flux~$\widehat{Y}_q$,
\begin{align*}
 \widehat{Y}_p(t;\theta_k) &= \begin{bmatrix} \begin{pmatrix} s_q^1(t;\theta_k) \\ d_p^1(t;\theta_k) \end{pmatrix} & \dots & \begin{pmatrix} s_q^{N_p}(t;\theta_k) \\ d_p^{N_p}(t;\theta_k) \end{pmatrix} \end{bmatrix}, \\
 \widehat{Y}_q(t;\theta_k) &= \begin{bmatrix} \begin{pmatrix} s_q^1(t;\theta_k) \\ d_p^1(t;\theta_k) \end{pmatrix} & \dots & \begin{pmatrix} s_q^{N_q}(t;\theta_k) \\ d_p^{N_q}(t;\theta_k) \end{pmatrix} \end{bmatrix},
\end{align*}
as well as dual state components $\widehat{Z}_p$ and $\widehat{Z}_q$, if available,
\begin{align*}
 \widehat{Z}_p(t;\theta_k) &= \begin{bmatrix} \; \mathfrak{p}^1(t;\theta_k) & \dots & \mathfrak{p}^{N_s + N_d}(t;\theta_k) \;\;\end{bmatrix}, \\
 \widehat{Z}_q(t;\theta_k) &= \begin{bmatrix} \; \mathfrak{q}^1(t;\theta_k) & \dots & \mathfrak{q}^{N_s + N_d}(t;\theta_k) \;\;\end{bmatrix}.
\end{align*}
The state-space trajectories $\widehat{X}_*(t;\theta_k)$, $\widehat{Z}_*(t;\theta_k)$ are obtained for perturbations of the inputs,
pressure at the $N_s$ supply boundary nodes and mass-flux at the $N_d$ demand boundary nodes,
while the output trajectories $\widehat{Y}_*(t;\theta_k)$ are computed for perturbations in the respective $N_*$ steady-state components.
Using the dual state trajectories is significantly faster than output trajectories,
as computing observability as dual reachability scales, as for the primal reachability, with the number of ports,
instead of scaling with the number of internal states.
The training parameters $\theta_k$ are sampled from a sparse grid spanning the parameter space $\Theta \subset \R^2$.

\pagebreak

All methods are prefixed ``Structured'', since the model structure of a pressure and mass-flux variable is preserved in the reduced order model.
Practically, this means while pressure and mass-flux trajectories are computed simultaneously due to their coupling,
the individual projectors for pressure and mass-flux are constructed separately.

The subsequent methods may not have been previously introduced explicitly in structured form,
yet given~\cite{morVanV08,morSanM09,morAal10} introducing structured Gramians, these are trivial extensions.
For ease of notation, we describe the computation of projectors $\{U_p, V_p\}$ and $\{U_q, V_q\}$ generically as~$\{U_*, V_*\}$.

We implemented a total of thirteen model reduction method variants, which we compare in this work.
All tested methods are data-driven and time-domain focused, as the dynamic gas network model~\eqref{eq:iosys} is nonlinear.
Furthermore, all model reduction methods construct linear subspaces and are derived from methods for linear systems,
yet differ from plain linearization:
instead, the implemented methods assemble linear subspaces of the model's phase space that, in a method specific sense, approximately enclosing the relevant nonlinear system evolution.
Moreover, all methods are SVD-based~\cite{morAnt05b}, and their majority is originally based on (empirical) system Gramian matrices,
for details see~\cite{morHim18b}.

We highlight here that the time horizon for the training data is significantly shorter than for the actual simulations the reduced order model is targeted at.
Furthermore, generic training inputs (transient boundary values), such as impulse, step or random signals, are utilized to avoid a \emph{model reduction crime}~\cite{morHim21} (comparable to an inverse crime):
\textit{Test} a reduced order model using the \textit{training} parameters or inputs.

\subsection{Empirical System Gramians}
All model reduction methods currently included in \morgen{} are computationally realized using empirical system Gramian matrices,
which are system-theoretic operators encoding reachability and observability.
From these, information on importance of linear combinations of states can be extracted.
For linear systems, these system Gramians are typically computed via matrix equations,
for general nonlinear systems there is no \emph{feasible} closed form.
However, the empirical system Gramians approximate the nonlinear Gramians based on state and output trajectory data.
In case of a port-Hamiltonian (nonlinear) systems, the approximate dual system \eqref{eq:dual} enables substituting expensive state by port perturbations, and thus severely reduce computation times.
Empirical Gramians are described in-depth in~\cite{morHim18b}; following only a brief summary is given.

\subsubsection{Empirical Reachability Gramian Matrix}
Reachability quantifies how well a system can be driven by the inputs,
which is encoded by the reachability Gramian.
The empirical reachability Gramian is an approximation based on state trajectory data~\cite{morLalMG99}:
\begin{align}\label{eq:wc}
 \widehat{W}_{R,*} \coloneqq \sum_{k=1}^{K} \sum_{m=1}^{N_s + N_d} \int_0^\infty \widehat{X}^m_*(t;\theta_k) \, \widehat{X}^m_*(t;\theta_k)^\T \D t \in \R^{N_* \times N_*}.
\end{align}
Given a suitable set of input perturbations, $\widehat{W}_{R,*}$ approximates the nonlinear reachability Gramian near a steady-state~\cite{morHahE02}.

\pagebreak

\subsubsection{Empirical Observability Gramian Matrix}
Observability quantifies how well the state can be characterized from the outputs,
which is encoded by the observability Gramian.
The empirical observability Gramian is an approximation based on output trajectory data~\cite{morLalMG99}:
\begin{align}\label{eq:wo}
 \widehat{W}_{O,*} \coloneqq \sum_{k=1}^{K} \int_0^\infty \widehat{Y}_*(t;\theta_k)^\T \, \widehat{Y}_*(t;\theta_k) \D t \in \R^{N_* \times N_*}.
\end{align}
For a suitable set of steady-state perturbations, $\widehat{W}_{O,*}$ approximates the nonlinear observability Gramian near a steady-state~\cite{morHahE02}.

Given the port-Hamiltonian structure of a discretization,
the empirical reachability Gramian of the dual system \eqref{eq:dual} can be used to compute the empirical observability Gramian \cite[Sec.~2]{morHim21} with $\widehat{Z}^m_*(t;\theta_k)$ instead of $\widehat{X}^m_*(t;\theta_k)$.

\subsubsection{Empirical Cross Gramian Matrix}
The empirical cross Gramian concurrently encodes reachability and observability,
which in conjunction quantifies redundancy also known as minimality;
however, the (empirical) cross Gramian is only applicable for square systems,
systems with the same number of inputs and outputs (which the gas network model~\eqref{eq:iosys} fortunately is),
and an approximation is based on simulated state and output trajectory data~\cite{morHimO14}:
\begin{align}\label{eq:wx}
 \widehat{W}_{X,*} \coloneqq \sum_{k=1}^{K} \sum_{m=1}^{N_s + N_d} \int_0^\infty \widehat{X}^m_*(t;\theta_k) \,\, \widehat{Y}_*(t;\theta_k) \D t \in \R^{N_* \times N_*}.
\end{align}
For a suitable set of input and steady-state perturbations, $\widehat{W}_{X,*}$ approximates the nonlinear cross Gramian near a steady-state~\cite{morHimO14}.

Given the port-Hamiltonian structure of a discretization,
the linear empirical cross Gramian exploiting the dual system \eqref{eq:dual} can be used to compute the empirical cross Gramian \cite{morBauBHetal17} with $\widehat{Z}^m_*(t;\theta_k)^\T$ instead of $\widehat{Y}_*(t;\theta_k)$.

\subsubsection{Empirical Non-Symmetric Cross Gramian Matrix}
A generalization of the empirical cross Gramian for non-square systems is the empirical non-symmetric cross Gramian, 
which is an approximation based on simulated state and (averaged) output trajectory data~\cite{morHimO16}:
\begin{align}\label{eq:wz}
 \widehat{W}_{Z,*} \coloneqq \sum_{k=1}^{K} \sum_{m=1}^{N_s + N_d} \sum_{q=1}^{N_s + N_d} \int_0^\infty \widehat{X}^m_*(t;\theta_k) \,\, \widehat{Y}^q_*(t;\theta_k) \D t \in \R^{N_* \times N_*}.
\end{align}
For a suitable set of input and steady-state perturbations, $\widehat{W}_{Z,*}$ approximates the nonlinear cross Gramian near a steady-state.
Even though the gas network model~\eqref{eq:iosys} is square, the empirical non-symmetric cross Gramian is included here,
since, heuristically, it could provide better results than the regular cross Gramian~\cite{morHimO16}.
Furthermore, an empirical non-symmetric linear cross Gramian is computable by similarly averaging over the dual states and replacing $\widehat{Y}^q_*(t;\theta_k)$ by $\widehat{Z}^q_*(t;\theta_k)^\T$.

\subsection{Structured Proper Orthogonal Decomposition}\label{sec:pod}
% Concept
Proper orthogonal decomposition (POD) is a basic data-driven method for model reduction:
given a matrix of state snapshots over time,
the dominant left singular vectors are computed as a basis via a singular value decomposition (SVD)~\cite{morMoo78}.
The basis vectors are assigned their principality with respect to the conveyed energy by the associated (relative) singular value magnitude.

In the context of this work, the POD is constructed from a system-theoretic point of view,
that connects to the system property of reachability.
Due to the overall structured approach to model reduction,
a structured POD refers in this context to the separate PODs for pressure and mass-flux variables $p$ and $q$ as in~\cite{morGruJHetal14,morGruHR16,morBenBG18}.

\subsubsection{Reachability-Gramian-Based}
The singular vectors to the principal singular values of the empirical reachability Gramian~\eqref{eq:wc} correspond to the POD modes.
To obtain a reduced order model, first, a truncated SVD (tSVD) of the empirical reachability Gramian,
\begin{align*}
 W_{R,*} \stackrel{\tsvd}{=} U_{R,*} D_{R,*} U_{R,*}^\T,
\end{align*}
reveals the principal subspace of the respective trajectories $\widehat{X}_*^m(t)$,
whereas the importance of each basis vector (column) in $U_{R,*}$ is determined by (the square-root of) the associated singular value $\sigma_i := D_{R,*\,ii}^{1/2}$:
\begin{align*}
 U_* := U_{R,*}.
\end{align*}
The matrix of basis vectors (POD modes) constitutes a Galerkin projection \linebreak $V_* := U_*$.
Notably, (structured) POD only considers the input-to-state mapping, not the state-to-output mapping,
and hence approximates the state variables, $p$ and $q$, not the output quantities of interest $s_q$ and $d_p$.

The POD could also be computed directly from an SVD of the trajectory data,
yet the computational overhead of using the empirical reachability Gramian is small compared to the trajectory simulation runtimes,
and the systematic perturbations of the empirical Gramian approach~\cite{morHim18b} are exploited.

\subsection{Structured Empirical Dominant Subspaces}\label{sec:eds}
The previous (structured) POD method considers only the reachability information,
hence the data only reflects the input-to-state mappings,
and thus the POD derived ROMs (Reduced Order Model) approximate the state variables $p$ and $q$.
To approximate the outputs $s_q$ and $d_p$,
the state-to-output mappings, encoding observability information, need to be considered, too.

The (empirical) dominant subspaces method initially developed in~\cite{morPen06},
and originally named DSPMR (Dominant Subspace Projection Model Reduction),
conjoins and compresses the dominant reachability and observability subspaces
of an input-output system, such as the gas network model~\eqref{eq:iosys}, obtained from (empirical) system Gramians.
Heuristically, this method seems to be useful for hyperbolic input-output systems~\cite{morGruHS19}.
Here, we consider three variants: first, based on the empirical reachability and observability Gramians,
second, based on the empirical cross Gramian and third, based on the empirical non-symmetric cross Gramian.

\subsubsection{Reachability- and Observability-Gramian-Based}
The singular vectors associated to the principal singular values of the empirical reachability and observability Gramians
span these dominant subspaces, which are first extracted by tSVDs,
and then, after concatenation, compressed by orthogonalization via another tSVD:
\begin{align*}
 \widehat{W}_{R,*} &\stackrel{\tsvd}{=} U_{R,*} D_{R,*} V_{R,*}^\T, \\
 \widehat{W}_{O,*} &\stackrel{\tsvd}{=} U_{O,*} D_{O,*} V_{O,*}^\T, \\
 \begin{bmatrix} (\omega_R U_{R,*} D_{R,*}) & (\omega_O U_{O,*} D_{O,*}) \end{bmatrix} &\stackrel{\tsvd}{=} U_{RO,*} D_{RO,*} V_{RO,*},
\end{align*}
from which the singular vectors $V_* = U_* := U_{RO,*}$ make up a Galerkin projection.
The weights $\omega_R := \|\widehat{W}_{R,*}\|_F^{-1}$,
and $\omega_O := \|\widehat{W}_{O,*}\|_F^{-1}$ equilibrate the potentially mismatched scales of the respective Gramians,
akin to the \textit{refined DSPMR} method from~\cite{morPen06}.

\subsubsection{Cross-Gramian-Based}
A truncated SVD of the cross Gramian also engenders the sought dominant subspaces~\cite{morBenH19}:
The (empirical) cross Gramian's left and right singular vectors (approximately) span the reachability and observability subspaces, respectively,
and their orthogonalized concatenation, via a truncated SVD,
\begin{align*}
 \widehat{W}_{X,*} &\stackrel{\tsvd}{=} U_{X,*} D_{X,*} V_{X,*}^\T, \\
 \begin{bmatrix} (U_{X,*} D_{X,*}) & (V_{X,*} D_{X,*}) \end{bmatrix} &\stackrel{\tsvd}{=} U_{RO,*} D_{RO,*} V_{RO,*},
\end{align*}
yields a Galerkin projection $V_* = U_* := U_{RO,*}$.
The cross-Gramian-based variant does not need to additionally weight the reachability subspace $U_{X,*}$ and observability subspace $V_{X,*}$,
as they are both extracted from the same matrix.

This cross-Gramian-based empirical dominant subspaces method seamlessly extends to the empirical non-symmetric cross Gramian $\widehat{W}_{Z,*}$.

\subsection{Structured Empirical Balanced Truncation}\label{sec:ebt}
% Concept
The dominant subspaces method combines separately quantified input-to-state and state-to-output energies,
but not the actual input-to-output energy.
Such can be accomplished by balanced truncation and based on the Hankel operator,
which maps past inputs to future outputs.
This operator's singular values measure the sought input-to-output energy,
and the singular vectors constitute a basis.
To obtain the Hankel operator's truncated SVD, first, the underlying system needs to be balanced,
and then singular vectors associated to small magnitude Hankel singular values are truncated.

Balanced truncation is the reference model reduction method for \emph{linear} input-output systems,
due to stability preservation in the ROMs and a computable error bound.
For (control-affine) nonlinear input-output systems, such as the gas network model~\eqref{eq:iosys},
balanced truncation can be generalized to \emph{empirical balanced truncation}~\cite{morLalMG99},
while a structured variant is introduced as \emph{interconnected system balanced truncation}~\cite{morVanV08},
which we combine.

Again, we consider three variants: First, based on the empirical reachability and observability Gramian,
second, based on the empirical cross Gramian and third, on the empirical non-symmetric cross Gramian;
additionally, the reachability and observability-based balanced POD variant is included.

\subsubsection{Reachability- and Observability-Gramian-Based}
The original balanced truncation method is based on the reachability and observability Gramians~\cite{morMoo81}.
A transformation into a balanced coordinate system in which both system Gramians are diagonal and equal is obtained via simultaneous diagonalization.
Various balancing algorithms are available for this task~\cite{morVar91b},
in this setting we selected the general balancing algorithm~\cite{morSafC88a,morSafC89},
which utilizes the magnitude-based truncated eigenvalue-decomposition (tEVD) of the Gramians:
The matrices $U_*$, $V_*$ constitute a Petrov-Galerkin projection,
whereas the importance of each column is determined by the approximate Hankel singular values, the (diagonal) elements of $D_{B,*}$:
\begin{align}\label{eq:btbal}
\begin{split}
 \widehat{W}_{R,*} \widehat{W}_{O,*} T_{R,*} &\stackrel{\tevd}{=} \Lambda_{R,*} T_{R,*}, \\
 \big(\widehat{W}_{R,*} \widehat{W}_{O,*}\big)^\T T_{O,*} &\stackrel{\tevd}{=} \Lambda_{O,*} T_{O,*}, \\
 T_{O,*}^\T T_{R,*} &\stackrel{\tsvd}{=} U_{B,*} D_{B,*} V_{B,*}, \\
 U_* &:= T_{R,*} V_{B,*} D_{B,*}^{-\frac{1}{2}}, \\
 V_* &:= T_{O,*} U_{B,*} D_{B,*}^{-\frac{1}{2}}.
\end{split}
\end{align}

\subsubsection{Cross-Gramian-Based}
% Concept
For linear, symmetric systems, alternatively a balanced and truncated reduced order model can be computed via the cross Gramian.
Yet, the gas network model~\eqref{eq:iosys} is neither linear nor symmetric, not even in a nonlinear sense of symmetric systems, i.e. gradient systems~\cite{morIonFS11},
but the considered input-output system is square.
Hence, an empirical cross-Gramian-based reduced model is computable,
but it will differ from the reduced model obtained by reachability- and observability-Gramian-based balanced truncation.

Given a cross Gramian with full rank, an (approximate) balancing projection is computable in a similar manner as for the (empirical) balanced truncation~\eqref{eq:btbal},
but based on the left and right eigen spaces of the cross Gramian~\cite{morJiaQY19}:
\begin{align}\label{eq:wxbal}
\begin{split}
 \widehat{W}_{X,*} T_{R,*} &\stackrel{\tevd}{=} \Lambda_* T_{R,*} \\
 \widehat{W}_{X,*}^\T T_{O,*} &\stackrel{\tevd}{=} \Lambda_* T_{O,*} \\
 T_{O,*}^\T T_{R,*} &\stackrel{\tsvd}{=} U_{X,*} D_{X,*} V_{X,*}, \\
 U_* &:= T_{R,*} V_{X,*} D_{X,*}^{-\frac{1}{2}}, \\
 V_* &:= T_{O,*} U_{X,*} D_{X,*}^{-\frac{1}{2}}.
\end{split}
\end{align}
The matrices $U_*$, $V_*$ again constitute a Petrov-Galerkin projection,
and the importance of each column is determined by the absolute value of the (diagonal) elements of $D_{X,*}$,
which are only equal to the Hankel singular values for linear and symmetric systems.

\pagebreak

As for the empirical dominant subspaces method,
the cross-Gramian-based empirical balanced truncation variant directly extends to the non-symmetric cross Gramian $\widehat{W}_{Z,*}$.

One could assume that if only one system Gramian has to be computed,
instead of two for dominant subspaces or balanced truncation, that the cross-Gramian-based computation is significantly faster,
but the overall number of simulated trajectories is the same for both methods, which causes the dominant fraction of computational cost.
Thus, the empirical cross Gramian computation is merely somewhat quicker than the computation of both empirical reachability and observability Gramians.

\subsubsection{Structured Balanced POD}
Instead of balancing the system using the product of the (empirical) reachability and observability Gramians,
the dominant subspaces of the respective Gramians can be used to approximately balance the system.
This variant of reachability- and observability-Gramian-based balanced truncation is called balanced POD~\cite{morWilP02},
and the approximate balancing algorithm reads:
\begin{align}\label{eq:bpod}
\begin{split}
 \widehat{W}_{R,*} T_{R,*} &\stackrel{\tevd}{=} \Lambda_{R,*} T_{R,*}, \\
 \widehat{W}_{O,*} T_{O,*} &\stackrel{\tevd}{=} \Lambda_{O,*} T_{O,*}, \\
 T_{O,*}^\T T_{R,*} &\stackrel{\tsvd}{=} U_{B,*} D_{B,*} V_{B,*}, \\
 U_* &:= T_{R,*} V_{B,*} D_{B,*}^{-\frac{1}{2}}, \\
 V_* &:= T_{O,*} U_{B,*} D_{B,*}^{-\frac{1}{2}},
\end{split}
\end{align}
with the matrices $U_*$, $V_*$ inducing a Petrov-Galerkin projection.

Here, we categorized balanced POD as a variant of balanced truncation,
due the algorithmic similarity to the balancing algorithm~\eqref{eq:btbal}.
Alternatively, we could have classified balanced POD as a variant of POD,
since it can be described as POD with the observability Gramian defining the POD's inner product~\cite{morRow05}.

We leave it to the reader to test other balancing algorithms, e.g.~\cite{morVar91b,morBen09},
while we excluded the modified POD~\cite{morOrSK12,morHim21},
as it is not a Petrov-Galerkin method.

\subsection{Structured Empirical Balanced Gains}\label{sec:ebg}
The empirical balanced gains method is a variant of the empirical balanced truncation method~\cite{morDav86,morHim21}:
While balanced truncation selects principality of subspaces based on the Hankel singular value magnitude,
balanced gains \cite{morDav86 sorts the balanced basis vectors} by the impulse response norm:
\begin{align*} 
 \|s_q\|_2^2 &\approx \tr(C_{sq} \widehat{W}_{R,q} C_{sq}^\T) \approx \tr(B_{qs}^\T \widehat{W}_{O,q} B_{qs})
           \approx    |\tr(C_{sq} \widehat{W}_{X,q} B_{qs})|, \\[2ex]
 \|d_p\|_2^2 &\approx \tr(C_{dp} \widehat{W}_{R,p} C_{dp}^\T) \approx \tr(B_{pd}^\T \widehat{W}_{O,p} B_{pd})
           \approx    |\tr(C_{dp} \widehat{W}_{X,p} B_{pd})|,
\end{align*}
for a system in \emph{balanced} form.
Hence, either method, balanced truncation and balanced gains compute the same balancing transformation,
via~\eqref{eq:btbal} or~\eqref{eq:wxbal}, but the sequence of basis vectors differs due to the variant measures.

We note, that due to the linear input and output operators,
the balanced gains method can be directly applied for the nonlinear gas network models~\eqref{eq:iosys}.

\pagebreak

\subsubsection{Structured Goal-Oriented Proper Orthogonal Decomposition}\label{sec:gopod}
Similar to the balanced gains method, (structured) POD basis vectors can also be sorted,
instead of by their singular value magnitude $\sigma_k$,
in terms of their impulse response, akin to the simplified balanced gains in \cite{morDav86}, by the index $d_k$
\begin{align*}
 d_k := \tilde{c}_k \tilde{c}_k^\T \sigma_k,
\end{align*}
for rows $\tilde{c}_k$ of the POD-transformed output matrix $\widetilde{C}$.
Given the reachability-based POD,
this variant is related to the concept of output-reachability and $H_2$-norm model reduction \cite{morPolv10}.

\subsection{Structured Dynamic Mode Decomposition - Galerkin}\label{sec:dmd}
In addition to the four energy-based method classes (Sections~\ref{sec:pod}, \ref{sec:eds}, \ref{sec:ebt}, \ref{sec:ebg}),
also an alternative method from~\cite{morAllK17}, based on system identification, is investigated.
Yet, it is adapted as a structured projection-based model reduction method.
After summarizing the parent system identification method, the derived model reduction technique is presented.

\subsubsection{Dynamic Mode Decomposition}
Dynamic mode decomposition (DMD) identifies a discrete-time operator from (discrete) trajectory time-series data,
preserving certain modal behavior~\cite{morRowMBetal09}.
This system identification method is based on the Koopman operator, which is an infinite dimensional, but linear operator,
mapping (a transformation, or observable of the) state $x^k$ at discrete-time step $k$ to $(k+1)$.
DMD yields a (linear) finite-dimensional approximation of the Koopman operator preserving its dominant eigenmodes.
Here, using the identity observable and given time series data \linebreak $X = \begin{bmatrix} x^0 & \dots & x^K \end{bmatrix}$,
DMD identifies an operator $\widehat{A}$, via least-squares:
\begin{subequations}
\begin{align}
 x^{k+1} &\approx \widehat{A} x^{k}, \\
 \widehat{A} &:= \begin{bmatrix} x^1 & \dots & x^K \end{bmatrix} \begin{bmatrix} x^0 & \dots & x^{K-1} \end{bmatrix}^+. \label{eq:dmd}
\end{align}
\end{subequations}
Since the underlying model is a control system, DMD with control (DMDc)~\cite{morProBK16} (and known input operator $B$),
i.e. using $X^c := \begin{bmatrix} x^0 - B u^0 & \dots & x^K - B u^K \end{bmatrix}$ instead of $X$ for DMD,
is applicable, yet, due to the perturbed steady-state training of ROMs (see \cref{sec:workflow}) not beneficial.

\subsubsection{Reachability-Gramian-Based DMD-Galerkin}
DMD is rather a system identification than a model reduction method.
To fit into the projection setting, we utilize the DMD-Galerkin method~\cite{morAllK17},
which forms a Galerkin projection $U$ from the orthogonalized dominant eigenvectors (based on the associated eigenvalue magnitudes),
\begin{align*}
 \widehat{A} T &\stackrel{\evd}{=} T \Lambda, \\
 \to T \Lambda &\stackrel{\svd}{=} U D V^\T.
\end{align*}
Since only state-space trajectories are utilized, as for the POD, the DMD-Galerkin method approximates the state, not the output.
Curiously, we note, that this method uses discrete-time information to assemble projections for continuous-time systems.

\pagebreak

Practically, $U_*$ are computed as a subset of singular vectors for the largest magnitude singular values,
instead of orthogonalized eigenvectors.
In the structured setting at hand $\widehat{X}_* = \begin{bmatrix} x_*^0 & \dots & x_*^K \end{bmatrix}$,
with $X_*$ representing the pressure and mass-flux states $X_p$, $X_q$, the respective Galerkin projection $U_* = V_*$ is given by:
\begin{align*}
 \begin{bmatrix} x_*^1 & \dots & x_*^K \end{bmatrix} \begin{bmatrix} x_*^0 & \dots & x_*^{K-1} \end{bmatrix}^+ &\stackrel{\tsvd}{=} U_* D_* \widetilde{V}_*^\T.
\end{align*}

Effectively, the approximate Koopman operator $\widehat{A}$ is computed via an empirical reachability Gramian,
yet instead of the standard inner product, the DMD-``kernel''~\eqref{eq:dmd} is used.
Additionally, and in line with the original empirical Gramians~\cite{morLalMG99},
the utilized trajectories are centered, following~\cite{morHirHKetal20}.
This means theoretically, DMD-Galerkin is POD with a specific kernel, and practically,
that by computation via an empirical Gramian the systematic perturbation properties can also be exploited for DMD-Galerkin.

\vskip1em
\begin{figure}[h!]\centering
 \includegraphics[width=\textwidth]{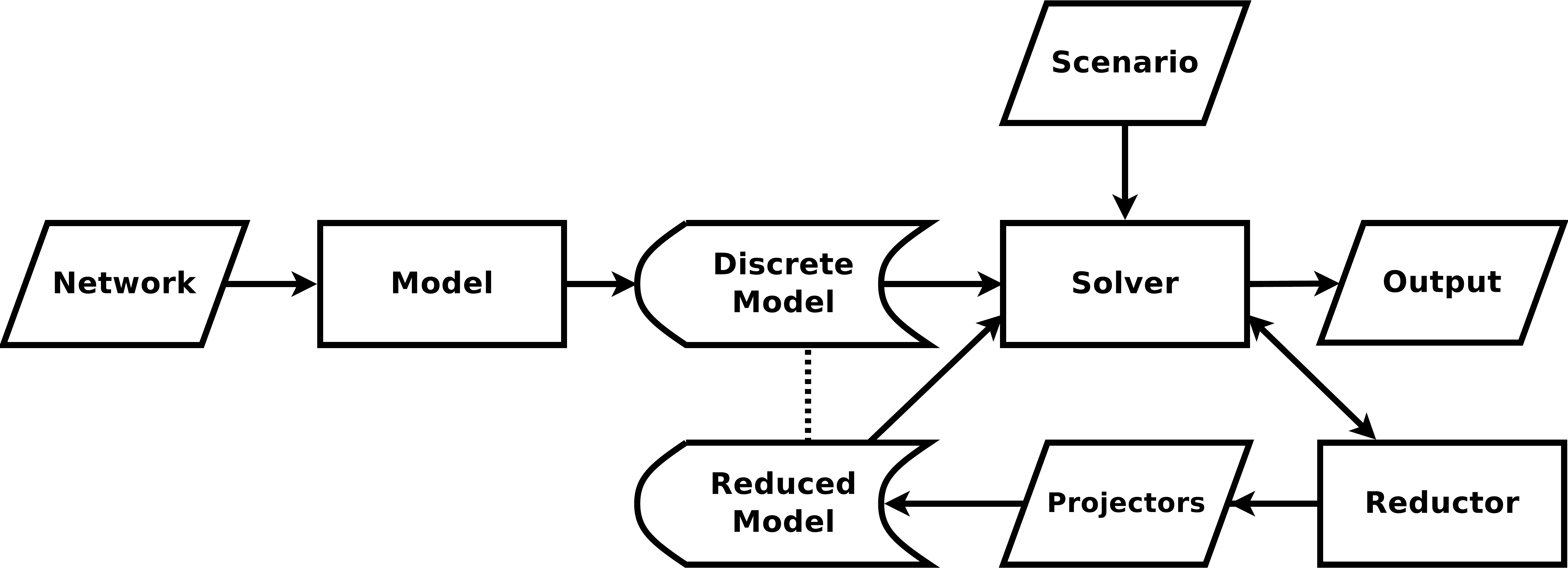}
 \caption{Internal data flow and process of \morgen{}.}\label{fig:morgen}
\end{figure}
\vskip-1em

\section{\morgen{}}\label{sec:morgen}
The \morgen{} (Model Order Reduction for Gas and Energy Networks) platform~(1.0) implements the mathematical methods presented above,
in MATLAB ($\geq$ 2020b) and compatible to Octave ($\geq$ 6.1).
Compared to, for example~\cite{HerDDetal09}, \morgen{} does not feature a graphical user interface or Simulink integration,
since it is designed for batch testing and multi-query use on (headless) workstations.

The source code is organized into five main components:
\begin{itemize}[leftmargin=5em]\itemsep.1ex

 \item[\texttt{networks}] \emph{(holds network and scenario data-sets)}

 \item[\texttt{models}] \emph{(discretizes networks and assembles input-output systems)}

 \item[\texttt{solvers}] \emph{(computes solution time series for discrete models and scenarios)}

 \item[\texttt{reductors}] \emph{(reduces state variables of discrete models)}

 \item[\texttt{tests}] \emph{(defines experiments for data-model-solver-reductor combinations)}.

\end{itemize}
These components are briefly described in the following.
For an illustration of the internal structure of \morgen{}, see \cref{fig:morgen}.
Further code that is used by the main function or multiple components is contained in an \texttt{utils} folder,
and the stand-alone network format converters\footnote{Converters for \textit{GasLib} \texttt{XML}, \textit{SciGRID\_gas} \texttt{CSV},
and \textit{MathEnergy} \texttt{JSON} are available.} are stored in a \texttt{tools} folder.

\subsection{Networks}
The \texttt{networks} directory stores the network \texttt{.net} files,
and a folder for each network with the same name as the associated network file's base name,
which hold the scenario definition files.
The \texttt{.net} file, in comma-separated value (CSV) format, defines the network topology and gas network components through an edge list.
Each row encodes one edge through the information: edge type (pipe, shortcut, compressor, valve), ``from''-node, ``to''-node, length, diameter, incline, roughness;
the latter three are only relevant for pipe edges. 

Note that boundary nodes have to be leaf nodes of the network graph in order to be identified by \morgen{} as such.
Furthermore, if a boundary node shall act as supply and demand,
it needs to be artificially split into two leaf nodes, as the edge connecting a supply node has to be directed from the leaf,
while the edge connecting a demand node has to be directed towards the leaf (\cref{sec:network}).

\subsubsection{Scenario}
A scenario is defined via a set of key-value pairs in an \texttt{.ini} file,
whereas sequence of pairs does not matter.
Each network has at least a generic \texttt{training.ini} scenario.
The following keys are mandatory for a scenario definition:
\begin{itemize}\itemsep.1ex
 \item[\texttt{T0}] \emph{Mean temperature}
 \item[\texttt{RS}] \emph{Mean specific gas constant}
 \item[\texttt{tH}] \emph{Time horizon}
 \item[\texttt{ut}] \emph{List of time instances}
 \item[\texttt{up}] \emph{List of supply node pressures  at time instances} \texttt{ut}
 \item[\texttt{uq}] \emph{List of demand node mass-fluxes  at time instances} \texttt{ut}.
\end{itemize}
Depending on the network composition, the following keys may need to be provided:
\begin{itemize}\itemsep.1ex
 \item[\texttt{cp}] \emph{List of compressor discharge pressures}
 \item[\texttt{vs}] \emph{List of valve settings}.
\end{itemize}
Note that amongst other configurations, the parameter ranges of temperature and specific gas constant for the parametric model reduction are set in the global \texttt{morgen.ini} file.

\subsection{Models}
A model encodes a spatial discretization of the simplified Euler equations~\eqref{eq:iosys} on a gas network topology
in a structure with the members: $A$, $B$, $C$, $E$, $F$, $f$ and the Jacobian $J$ (\cite{StoM18,AltZ19}),
together with the system dimensions in terms of number of pressure-at-nodes and mass-flux-on-edges states, and total number of boundary/port nodes.
It is ensured during the assembly of the model that the sparsity of the model components is preserved.

The model interface is given by the following signature:
\begin{center}
 \texttt{discrete = model(network,config);}
\end{center}

While the linear components $A$, $B$, $C$, and $F$ are provided as sparse matrices,
the parameter-dependent linear component $E$ is a closure\footnote{A closure is a pair of a function together with its scoped environment.} returning a sparse matrix,
and the nonlinear component $f$ and the (nonlinear) Jacobian $J$ are closures returning the application of a state (as well as steady-state, input, parameters, compressibility).

Two spatial discretizations are currently provided in \morgen{}:

\begin{itemize}[leftmargin=5em]\itemsep.1ex

 \item[\texttt{ode\_mid}] \emph{Midpoint ODE discretization} (\cref{sec:midpoint})

 \item[\texttt{ode\_end}] \emph{Endpoint ODE discretization} (\cref{sec:endpoint}).

\end{itemize}

Even though both provided models are ODEs,
DAE models can also be implemented, given a solver (and reductor) is available.

\subsection{Solvers}\label{sec:solvers}
The \morgen{} platform provides four solvers:
An adaptive step-size method, a fixed step-size explicit method and two fixed step-size implicit-explicit methods.

While explicit solvers only require vector field evaluations at the cost of smaller time steps,
implicit solvers have to solve a root-finding problem in each time step for a nonlinear model.
An IMEX with singly DIRK (SDIRK) solver turns out to be the most efficient for this class of models,
simplifying the nonlinear problem to a linear problem solvable by a single matrix decomposition per trajectory.
Due to the non-diagonal mass matrix,
even in the case of a sufficiently stable and accurate explicit method, at least one linear problem per trajectory would have to be solved.

The solver interface is given by the following signature:
\begin{center}
 \texttt{solution = solver(discrete,scenario,config);}
\end{center}
with the return value \texttt{solution}, being a structure,
and the arguments \texttt{discrete} (model), \texttt{scenario}, and \texttt{config}(uration).

All provided fixed step-size solvers cache matrix decompositions.
The initial steady-state is also cached, as is the QR decomposition used to compute it.

Even though the overall model~\eqref{eq:iosys} has a two dimensional structure and the reductors exploit this structure,
the simulations itself can be performed on a lumped model (we omit parametrization here for ease of notation):
\begin{align}\label{eq:lumped}
\begin{split}
 E\dot{x}(t) &= Ax(t) + Bu(t) + f(x(t),u(t)), \\
        y(t) &= Cx(t).
\end{split}
\end{align}

\subsubsection{Second-Order Adaptive Implicit Solver: \texttt{generic}}\label{sec:adapt}
For validation purposes, the adaptive step-size solver for stiff systems \texttt{ode23s}\footnote{\url{https://mathworks.com/help/matlab/ref/ode23s.html} (accessed: 2020-11-18)}
included in MATLAB (and Octave)~\cite[Sec.~2.3]{ShaR97},
based on a modified second order Rosenbrock formula is used and encapsulated as a generic solver.
Due to preferential performance demonstrated in~\cite{morGruHR16}, \texttt{ode23s} is preferred over alternatives such as \texttt{ode15s} (or \texttt{ode45}).

\subsubsection{Fourth-Order ``Classic'' Runge-Kutta Solver: \texttt{rk4}}\label{sec:rk4}
Since in~\cite{Osi84,Osi87}, the fourth-order explicit Runge-Kutta (RK) method~\cite{Kut01} is employed, \morgen{} provides it, too.
This method is strong stability preserving~\cite{GotST01}, however it is not SSP-optimal and it works only for small time-steps.

\subsubsection{First-Order Implicit-Explicit Solver: \texttt{imex1}}\label{sec:imex1}
The lumped gas network model~\eqref{eq:lumped} can be split into a linear and a nonlinear part,
of which the linear part is numerically stiff and hence should be solved with an implicit solver,
while an explicit solver is preferred for the nonlinear part as to avoid solving a root-finding (optimization) problem in each time step.

An IMEX method allows this separate treatment of operators and is thus suitable for this hyperbolic and nonlinear system.
Combining the first-order explicit Euler's method with the first order implicit Euler's method yields the first order IMEX method:
\begin{align*}
                    E h^{-1}(x_{k+1} - x_k) &= (1 - \gamma) A x_k + \gamma A x_{k+1} + B u_k + f(x_k,u_k) \\
 \Rightarrow E x_{k+1} - \gamma h A x_{k+1} &= E x_k + (1 - \gamma) h A x_k + h B u_k + h f(x_k,u_k) \\
                        \Rightarrow x_{k+1} &= x_k  + (E - \gamma h A)^{-1} h \big(A x_k + B u_k + f(x_k,u_k)\big),
\end{align*}
with the associated Butcher tableaus \cref{tab:imex1}.
Even though this IMEX method is not a Runge-Kutta method~\cite{AscRS97},
it was successfully applied to gas network models in~\cite{morGruJ15} with a relaxation parameter set to $\gamma = 1$.

\begin{table}[h]\centering
 \begin{tabular}{r|c}
  \multicolumn{2}{c}{Explicit:} \\[1ex]
  0 & 0 \\ \hline
    & 1
 \end{tabular}~~~~~~~~~~~~~
 \begin{tabular}{r|c}
  \multicolumn{2}{c}{Implicit:} \\[1ex]
  1 & 1 \\ \hline
    & 1
 \end{tabular}~~~~~~~
 \caption{Butcher tableaus for the 1st-order IMEX method (\cref{sec:imex1}).}
 \label{tab:imex1}
\end{table}

\subsubsection{Second-Order Implicit-Explicit Runge-Kutta Solver: \texttt{imex2}}\label{sec:imex2}
A second-order (two-stage) IMEX Runge-Kutta solver is provided, based on the combination of a second-order explicit SSP Runge-Kutta method~\cite{GotST01},
and a second-order DIRK method.
Following~\cite{ParG05}, such an IMEX-SSP2(2,2,2) method with relaxation $\gamma$ is given by:
\begin{align*}
 z_1 &= (E - h \gamma \lambda A)^{-1} \phantom{(} E x_k, \\
 z_2 &= (E - h \gamma \lambda A)^{-1} (E x_k + h B u_k  + h f(x_k,u_k) + h \gamma (1 - 2 \lambda) A z_1), \\
 x_{k+1} &= x_k + E^{-1} \frac{h}{2} (B u_k + f(x_k,u_k) + \gamma A z_1 + B u_{k+1} + f(z_1,u_{k+1}) + \gamma A z_2).
\end{align*}
The explicit component of this IMEX method is SSP-optimal~\cite{GotST01},
while depending on the choice for the free parameter $\lambda$,
different properties of the implicit component can be achieved (see \cref{tab:lambda}).
Practically, we found $\lambda = {\textstyle \frac{1}{2}}$, making the implicit part SDIRK and stiffly accurate, to work best.
Additionally, we would like to highlight passive Runge-Kutta methods~\cite{Och01},
specifically PDIRK (passive DIRK), as implicit IMEX component,
which have various desirable stability and conservation properties~\cite{FraO01}.

\begin{table}[H]\centering
 \rowcolors{1}{white}{lightgray}
 \begin{tabular}{lcc}
  Parameter & Property & Source \\ \hline
  $\lambda = 0.24$ & ``Efficient'' & \cite{KupHHetal12} \\
  $\lambda \geq \frac{1}{4}$ & A-stable & \cite{IzzJ17} \\
  $\lambda = \frac{2 - \sqrt{2}}{2} = 1 - \frac{1}{\sqrt{2}}$ & L-Stable & \cite{ParG05} \\
  $\lambda = \frac{1}{2}$ & Stiffly accurate & \morgen, cf.~\cite[Sec.~5.1]{IzzJ17} \\ % TODO check if correct!
  $\lambda = \frac{3 + \sqrt{3}}{6} = \frac{1}{2} + \frac{1}{\sqrt{12}}$ & Passive & \cite{Och01}, cf.~\cite[Sec.~3.2.2]{Ban15}
 \end{tabular}
 \caption{Parameter choices for the implicit 2nd-order IMEX-RK component.}
 \label{tab:lambda}
\end{table}

The associated Butcher tableaus are given in \cref{tab:imex2}.

In our experiments, the first-order IMEX method \cref{sec:imex1} allowed larger time-steps and exhibited less numerical oscillations or artifacts
compared to the second-order IMEX-RK methods.
The generic (adaptive) method \cref{sec:adapt} also solves sufficiently accurate, but takes longer to compute.
Thus, by default, we recommend the first-order IMEX integrator for gas network simulations.

\begin{table}[H]\centering
 \begin{tabular}{c|cc}
  \multicolumn{3}{c}{Explicit:} \\[1ex]
  $0$ & $0$ & $0$ \\
  $1$ & $1$ & $0$ \\ \hline
    & $\frac{1}{2}$ & $\frac{1}{2}$
 \end{tabular}~~~~~~
 \begin{tabular}{c|cc}
  \multicolumn{3}{c}{Implicit:} \\[1ex]
  $\lambda$ & $\lambda$ & $0$ \\
  $1-\lambda$ & $1-2\lambda$ & $\lambda$ \\ \hline
    & $\frac{1}{2}$ & $\frac{1}{2}$
 \end{tabular}
 \caption{Butcher tableaus for the 2nd-order IMEX-RK method (\cref{sec:imex2}).}
 \label{tab:imex2}
\end{table}

\subsection{Reductors}
The reductor module provides methods that compute structured (struct.) projectors for a given discretization.
These projectors can be stored on disk for reuse.
Currently, the reductors, described in \cref{sec:methods} are included:
\begin{itemize}[leftmargin=4em]\itemsep.1ex

 \item[\texttt{pod}] \emph{Struct. proper orthogonal decomposition}

 \item[\texttt{eds}] \emph{Struct. empirical dominant subspaces}

 \item[\texttt{bpod}] \emph{Struct. balanced proper orthogonal decomposition}

 \item[\texttt{ebt}] \emph{Struct. empirical balanced truncation}

 \item[\texttt{gopod}] \emph{Struct. goal-oriented proper orthogonal decomposition}

 \item[\texttt{ebg}] \emph{Struct. empirical balanced gains}

 \item[\texttt{dmd}] \emph{Struct. dynamic mode decomposition Galerkin.}

\end{itemize}
Each reductor variant has a suffix characterizing the employed empirical Gramians:
\texttt{\_r} for reachability, \texttt{\_ro} for reachability and observability,
\texttt{\_wx} for cross, and \texttt{\_wz} for non-symmetric cross Gramian.
For each reductor utilizing observability information, a linear variant using the dual system is available,
and signified with the additional suffix \texttt{\_l}.
Originally, all methods were also tested in an unstructured variant, but showed insufficient accuracy.

The reductors have the interface:
\begin{center}
 \texttt{[proj,name] = reductor(solver,discrete,scenario,config);}
\end{center}
returning a (cell) array of \texttt{proj}ectors with maximum configured column rank, as well as the reductor's full name.
These projectors specific to model and solver defining a ROM are then stored in a \texttt{.rom} file.

\pagebreak

\subsubsection{Empirical Gramian Framework}
The compute back-end for the model reduction methods is \texttt{emgr} -- empirical Gramian framework~\cite{morHim18b};
currently in version~5.9~\cite{morHim21a}.
This (open-source) \linebreak Octave and MATLAB toolbox computes the empirical Gramians,
which are essential to construct the reduced order models via the Gramian-based model reduction methods from \cref{sec:methods},
including the structured DMD-Galerkin method.

\subsection{Tests}
A test is a script defining an experiment by specifying network, scenario, model, solver and reductors.
The \texttt{tests} component is a collection of test scripts probing primarily model reduction for various networks.
A typical test contains two calls to the main \texttt{morgen} function.
The first call computes the reduced order model (offline phase):
\begin{center}
 \texttt{morgen(network,training\_scenario,model,solver,reductors);}
\end{center}
which computes a ROM from a short, generic, steady-state \texttt{training\_scenario}.
The projectors defining the reduced order models are then stored in \texttt{rom\_files}.
The second call tests the reduced order model(s) on a longer \texttt{test\_scenario} (online phase):
\begin{center}
 \texttt{morgen(network,test\_scenario,model,solver,rom\_files);}
\end{center}
Generally, a model reduction method can also be tested in a single call,
disregarding that scenario's boundary value time series,
yet for productive use, a reduced order model is constructed once (first call),
and then employed for many different scenarios (second call).

Included in \morgen{} are two types of tests: First, tests prefixed with ``\texttt{sim$\_$}'' only simulate the test scenarios,
second, tests prefixed with ``\texttt{mor$\_$}'' compute the reduced order models using training scenarios,
and benchmark these ROMs on the test scenarios.

\section{Numerical Experiments}\label{sec:numex}
In the following, we present three sets of numerical experiments,
with the purpose of demonstrating the reducibility of gas network models via the data-driven, parametric, system-theoretic model order reduction algorithms from \cref{sec:methods},
and illustrating the capabilities of the \morgen{} framework summarized in \cref{sec:morgen}.
The first set uses a pipeline ``network'', which is interesting in the context of model reduction,
while the second set tests an academic toy network as a sanity check and a simple functionality test.
Lastly, a realistic gas network topology is evaluated.

We note that various further networks are included for testing in \morgen{}; among others:
the \textsc{Canvey-Leeds} network~\cite{Guy67,Kiu94},
the Belgium transport network~\cite{DeWS00,morMohHHetal04}
and a part of the \textsc{Fermaca} network~\cite{RodSd18}.
A synthetic pipeline model and associated simulation results were provided by the \textsc{PSI~Software~AG}
for validation of \morgen{} against the commercial \texttt{PSIganesi}\footnote{\url{https://www.psigasandpipelines.com}} solver.

\pagebreak

\subsection{Workflow}\label{sec:workflow}
For each of the numerical tests, the same workflow is employed,
which is composed of a training phase (offline phase), in that the ROMs are computed,
using a generic test scenario, with a (virtual) time horizon of $1$h,
and boundary value input functions typical for system identification,
i.e. Dirac impulse, step signal, random-binary signal or Gauss noise~\cite[Ch.~16]{Nel01}.
In the test phase (online phase), the ROMs are tested on scenarios with a (virtual) time horizon of $24$h \cite{CheFBetal15,MakVZetal19},
(starting at 6am~\cite{EhrS05}).
In addition to shorter offline phases,
this difference in training and test time horizons emphasizes generality of the ROMs.

To verify models and solvers~\cite{OsiC10} this offline/online procedure is performed for all combinations of:
\begin{itemize}[itemsep = -.5ex]

 \item Models: \texttt{ode\_mid}, \texttt{ode\_end};

 \item Solvers: \texttt{imex1}, \texttt{imex2};

 \item Reductors: \texttt{pod\_r}, \,\,\, \texttt{eds\_ro}, \texttt{eds\_wx}, \texttt{eds\_wz}, \\
       \phantom{Reductors:} \texttt{bpod\_ro}, \texttt{ebt\_ro}, \texttt{ebt\_wx}, \texttt{ebt\_wz}, \\
       \phantom{Reductors:} \texttt{gopod\_r}, \texttt{ebg\_ro}, \texttt{ebg\_wx}, \texttt{ebg\_wz}, \\
       \phantom{Reductors:} \texttt{dmd\_r};
\end{itemize}
whereas the port-Hamiltonian \texttt{ode\_end} model is tested with the nonlinear as well as the linear reductor variant (\texttt{\_l} suffix) if available,
while the \texttt{ode\_mid} model is only tested with the nonlinear reductor variant.

The models are specialized by the \textit{Schifrinson} friction, and the \textit{AGA88} compressibility factor formula.
We excluded the \texttt{generic} and \texttt{rk4} solvers in this comparison as they are too slow or too fragile, respectively.
Yet, the test scenario visualizations in \cref{fig:numex1:scen} and \cref{fig:numex2:scen}
are computed by the \texttt{generic} solver.

For the parametric model reduction,
the temperature range for training and testing is set to $[0^\circ,15^\circ]$C,
while the specific gas constant range is chosen as $[500,600]\frac{\text{J}}{\text{kg}\,\text{K}}$.
During training, samples from the parameter space are drawn from a sparse grid,
whereas for the tests, parameters are drawn from a uniform random distribution.
For either test and training, five parameters are sampled.
The input perturbations for the steady-state training scenario are selected to be a step function,
which heuristically works well for hyperbolic systems~\cite{morGruHS19}.

The reduced order models are compared via the approximate, discrete, parametric, $(L_2 \otimes L_2)$-norm of the output error~\cite{morHim21}:
\begin{align*}
 \| y - \tilde{y} \|_{L_2 \otimes L_2} \approx \sqrt{\sum_{\theta_h \in \Theta_h} \Delta t \big\|\!\operatorname{vec}\big(y_h(\theta_h) - \tilde{y}_h(\theta_h)\big)\big\|_2^2},
\end{align*} 
for a finite sample $\Theta_h$ of the parameter space $\Theta$ and discrete output samples $y_h(\theta_h)$, $\tilde{y}_h(\theta_h)$.
This energy norm is chosen, since all methods are at least related to an energy-based method.
However, \morgen{} can also provide the errors in the \emph{approximate} parametric $(L_k \otimes L_\ell)$ parameter-space-state-space norms
for $k \in \{1,2,\infty\}$, $\ell \in \{0,1,2,\infty\}$, cf.~\cite{morGruHKetal13}.
Note, that due to nonlinearity of the considered models, and the averaging nature of the norm, a monotonic error decay cannot be expected.
To enhance comparability of the results, also the reducibility measure \textsc{MORscore}~\cite{morHim21} for each experiment is computed.
The \textsc{MORscore} for a certain method and model is essentially the area above the method's error graph in the relative error plot such as \cref{fig:numex1:mid1}.

Lastly, we note that the following numerical experiments are conducted using a computer with an
AMD \textsc{Ryzen} 4500U @ 2.3Ghz hexa-core processor and 16GiB memory
running MATLAB~2021a on \textsc{Ubuntu}~20.04 \textsc{Linux}.

\begin{figure}[H]\centering

 \includegraphics[width=\textwidth]{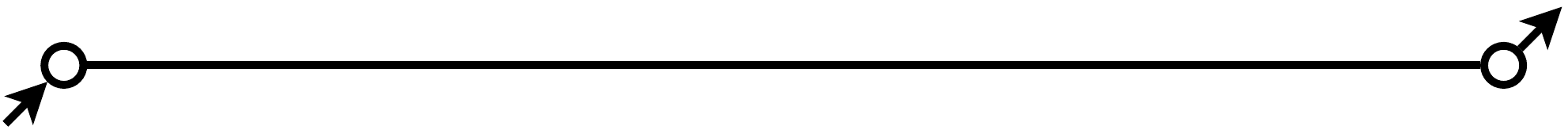}

 \caption{Pipeline topology.}
 \label{fig:pipeline:topo}
\end{figure}

\subsection{Yamal-Europe Pipeline}\label{sec:numex1}
First, a pipeline is tested,
which is an interesting test case, since the trivial topology (\cref{fig:pipeline:topo}) comprises little redundancy,
hence pipelines are a useful benchmark for model reduction methods.

The Yamal-Europe pipeline connects gas fields in Russia with western Europe\footnote{See also \url{https://en.wikipedia.org/wiki/Yamal-Europe\_pipeline}.}.
A section of this pipeline was also benchmarked in~\cite{Cha09,OsiC10,BroGH11,BenGHetal18},
from which the technical properties and test scenario are taken.
The considered pipeline section is $363$km long, has a diameter of $1.422$m, no (reported) inclination, and a pipe roughness of $0.01$mm.
A steady-state, used as initial state, is set by a supply pressure of $84$bar and demand mass-flux of $46.3\frac{\text{kg}}{\text{s}}$.

The semi-discrete nonlinear state-space system has two inputs and outputs as well as $908$ states;
and a time step width of $20$s is used.
The employed test scenario is taken from~\cite{Cha09}, compressed to $24$h, and shown in \cref{fig:numex1:scen},
the associated model reduction errors are given in \cref{fig:numex1:mid1}, \cref{fig:numex1:end1}, \cref{fig:numex1:mid2}, and \cref{fig:numex1:end2},
for up to reduced order $150$, while the resulting \textsc{MORscore}s are listed in \cref{tab:numex1:morscore}.

Generally, the choice of solver is more relevant than the choice of model:
while the \textsc{MORscore}s for different models but same solver are similar,
for the same model but different solver, they are significantly dissimilar.
Also, the tested balancing (Petrov-Galerkin) methods perform worse than the Galerkin methods.

For both models, and the first-order IMEX solver, the structured empirical dominant subspaces methods perform best,
followed (closely) by the DMD-Galerkin and (goal-oriented) POD method;
then, among the balancing methods, the balanced POD and cross-Gramian-based variants.
The most overall accurate method is the cross-Gramian-based dominant subspaces method.

For both models, in combination with the second-order IMEX-RK solver,
the structured POD and goal-oriented methods lead, followed by the DMD-Galerkin reductor.
For both solvers, the endpoint model performs better than the midpoint model.
In case of the port-Hamiltonian endpoint model,
the linear Galerkin reductors are about as accurate as the nonlinear Galerkin reductors.

We note that the cross-Gramian-based dominant subspaces methods produce the lowest errors,
and since for the linear reductors used in combination with the endpoint model,
the dominant subspaces methods are as efficient as the purely reachability-based DMD-Galerkin,
and (goal-oriented) POD methods.

Interestingly, while the second-order IMEX-RK solver is better suited for simulations of the full order model,
in terms of data-driven model reduction and/or reduced order model simulation it is significantly worse than the first-order IMEX solver.
This is also demonstrated in the subsequent experiments.

\pagebreak %temp

\begin{figure}[H]

\begin{subfigure}{0.49\textwidth}
\includegraphics[width=\textwidth]{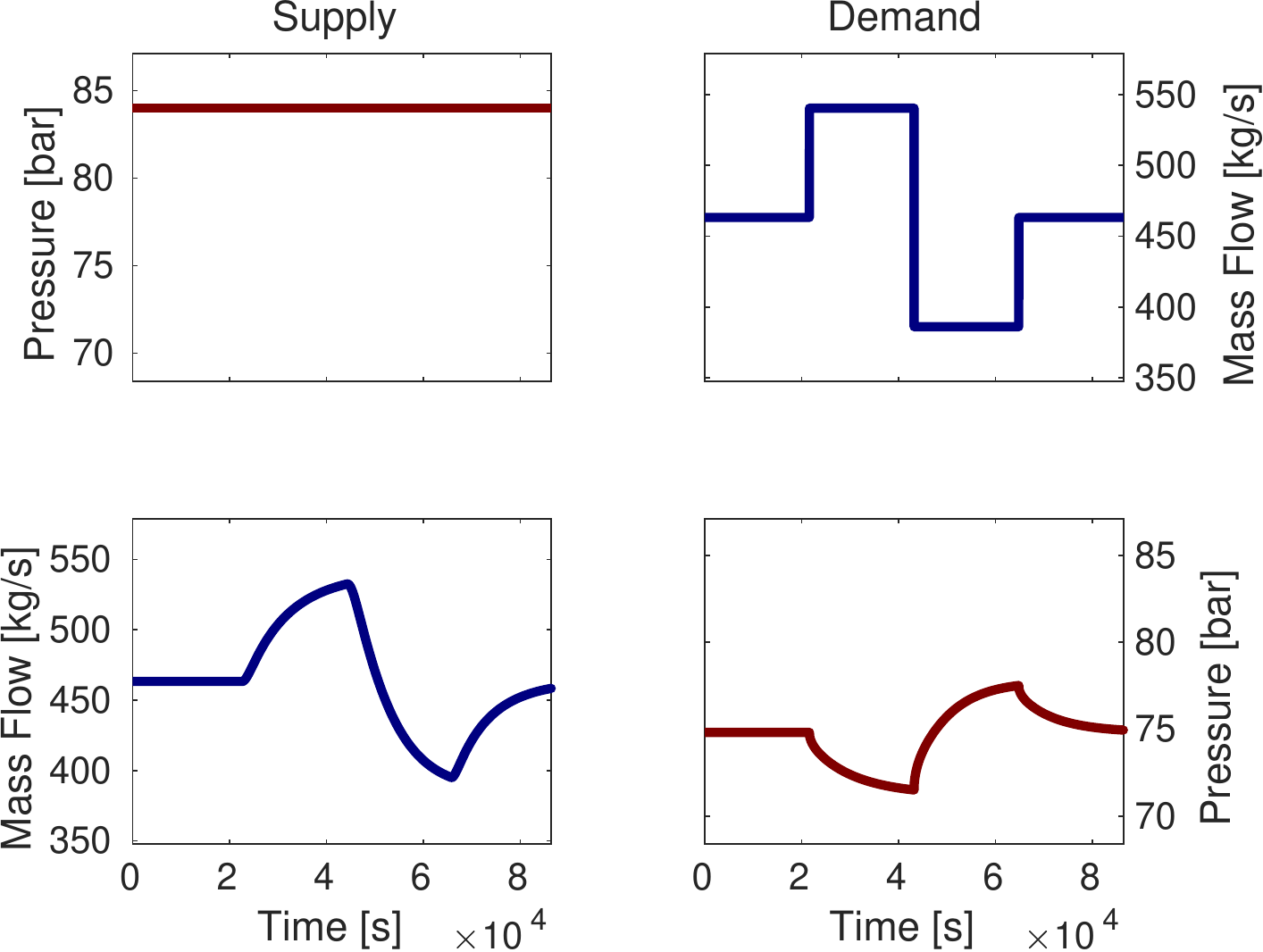}
\caption{Boundary values and quantities of \linebreak interest plots for the test scenario.}
\label{fig:numex1:scen}
\end{subfigure}
~~
\begin{subfigure}{0.49\textwidth}

\vspace{2em}

\includegraphics[width=\textwidth]{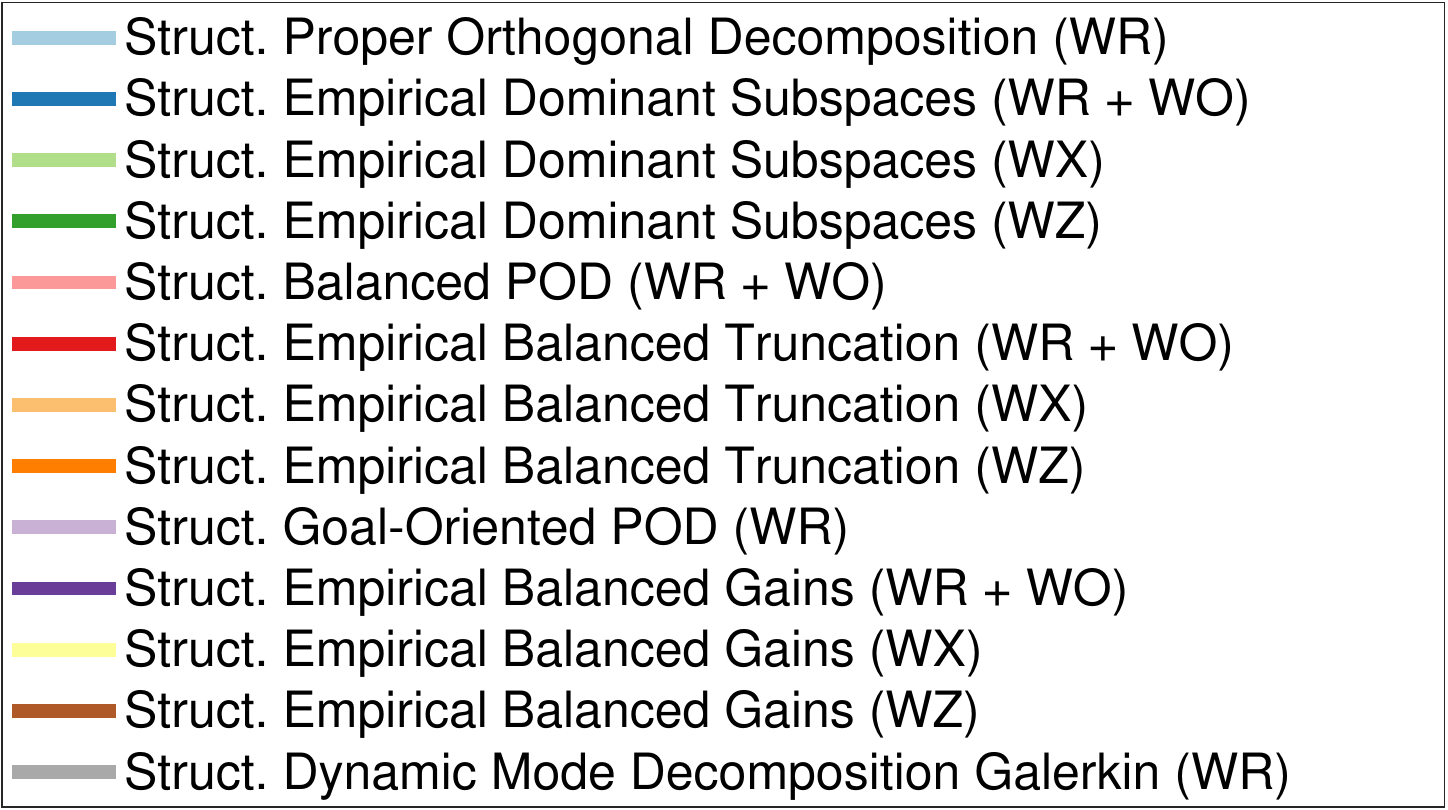}

\vspace{2em}

\caption{Common legend for the error plots \cref{fig:numex1:mid1}, \cref{fig:numex1:end1}, \cref{fig:numex1:mid2}, \cref{fig:numex1:end2}.}
\label{fig:numex1:legend}
\end{subfigure}

\begin{subfigure}{0.49\textwidth}
\includegraphics[width=\textwidth]{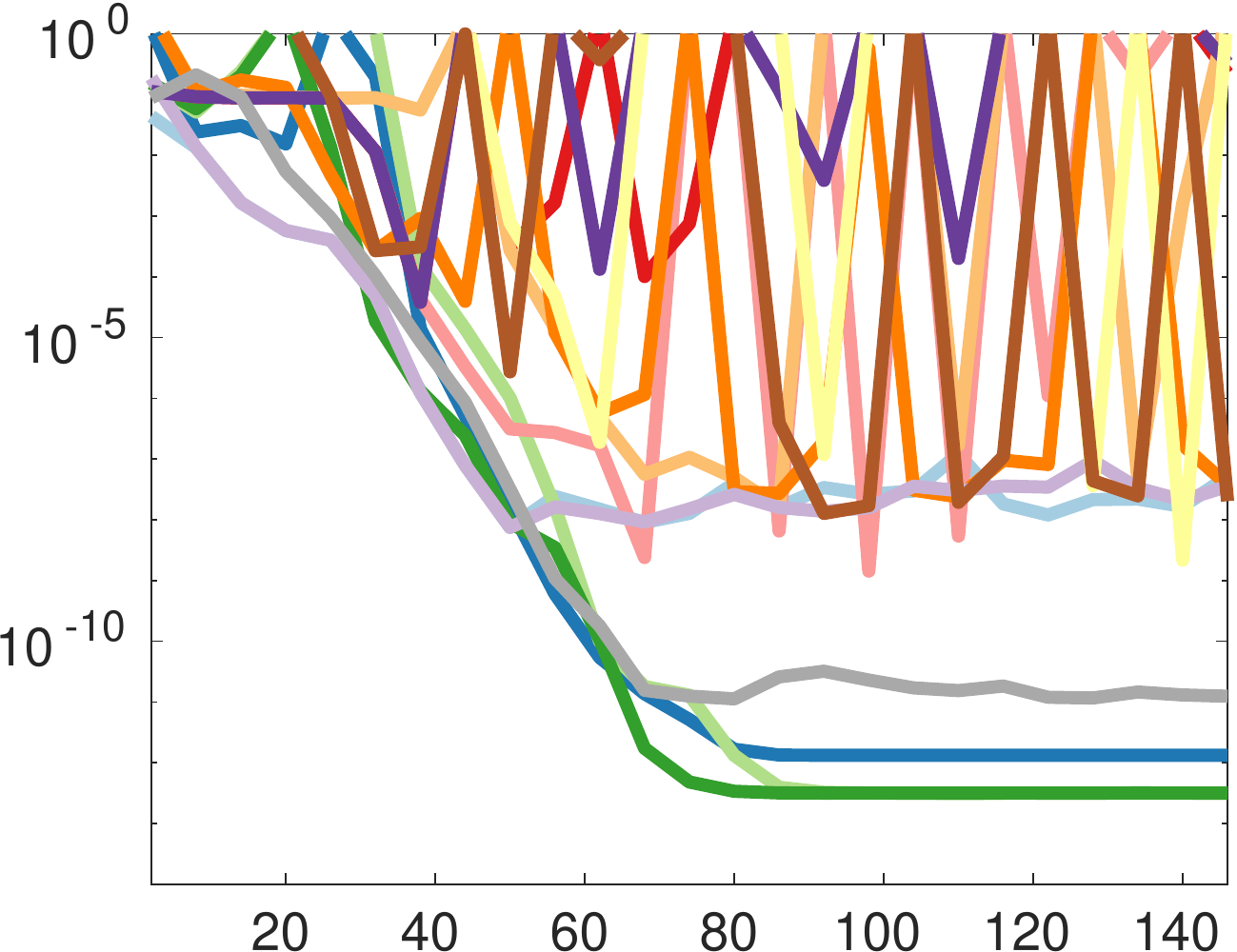}
\caption{$L_2 \otimes L_2$ error between ROM and FOM for the \texttt{ode\_mid} model, \texttt{imex1} solver, and \emph{nonlinear} reductors versus reduced order.}
\label{fig:numex1:mid1}
\end{subfigure}
~~
\begin{subfigure}{0.49\textwidth}
\includegraphics[width=\textwidth]{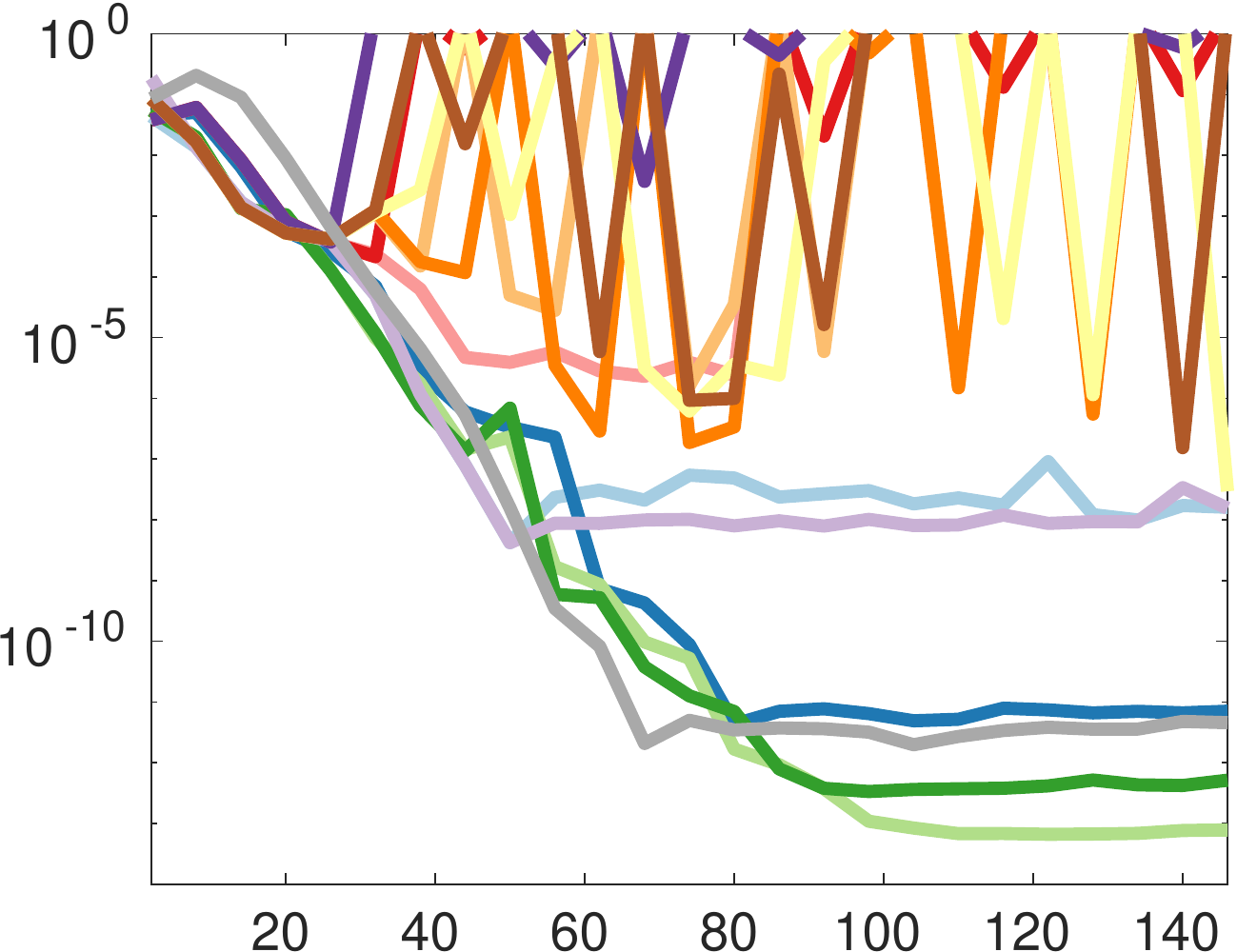}
\caption{$L_2\,\otimes\,L_2$ error between ROM and FOM for the \texttt{ode\_end} model, \texttt{imex1} solver, and \emph{linear} reductors versus reduced order.}
\label{fig:numex1:end1}
\end{subfigure}

~\\

\begin{subfigure}{0.49\textwidth}
\includegraphics[width=\textwidth]{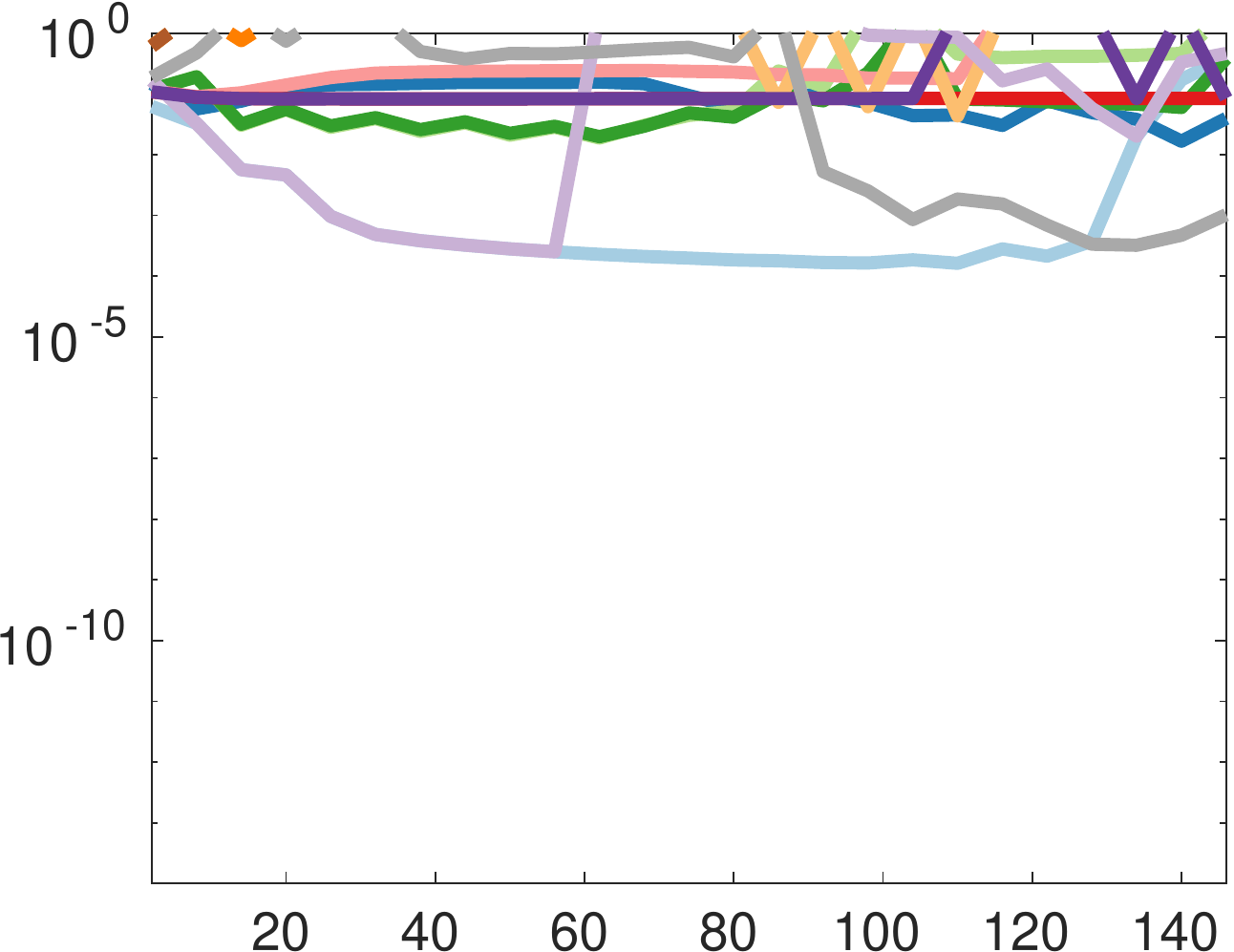}
\caption{$L_2\,\otimes\,L_2$ error between ROM and FOM for the \texttt{ode\_mid} model, \texttt{imex2} solver, and \emph{nonlinear} reductors versus reduced order.}
\label{fig:numex1:mid2}
\end{subfigure}
~~
\begin{subfigure}{0.49\textwidth}
\includegraphics[width=\textwidth]{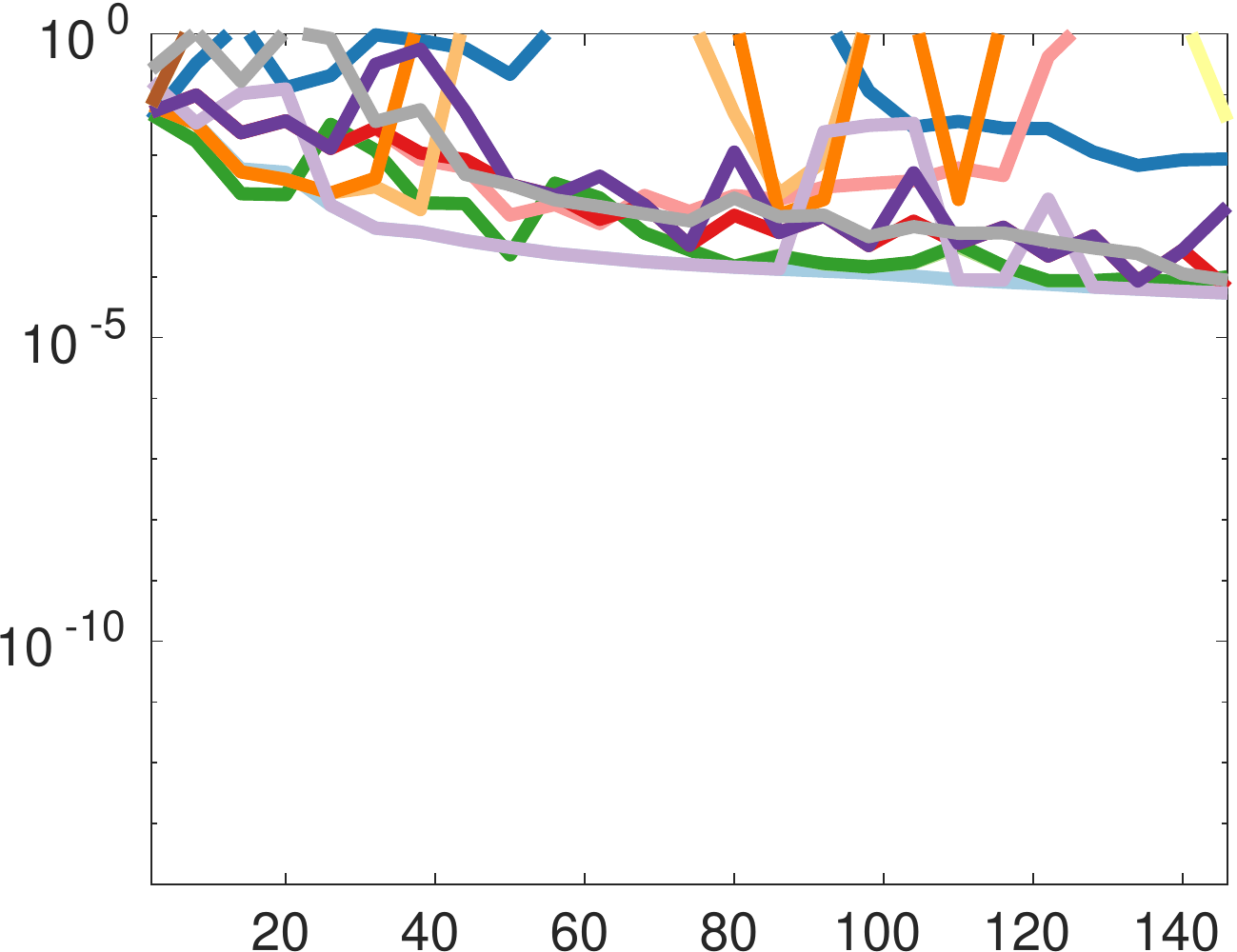}
\caption{$L_2\,\otimes\,L_2$ error between ROM and FOM for the \texttt{ode\_end} model, \texttt{imex2} solver, and \emph{linear} reductors versus reduced order.}
\label{fig:numex1:end2}
\end{subfigure}

\caption{Visualization of the test scenario, and model reduction errors of the tested ROMs for the Yamal-Europe pipeline from \cref{sec:numex1}.}
\label{fig:numex1}
\end{figure}

\newpage %temp

\begin{figure}[H]\centering

 \vspace{2em}

 \includegraphics[width=\textwidth]{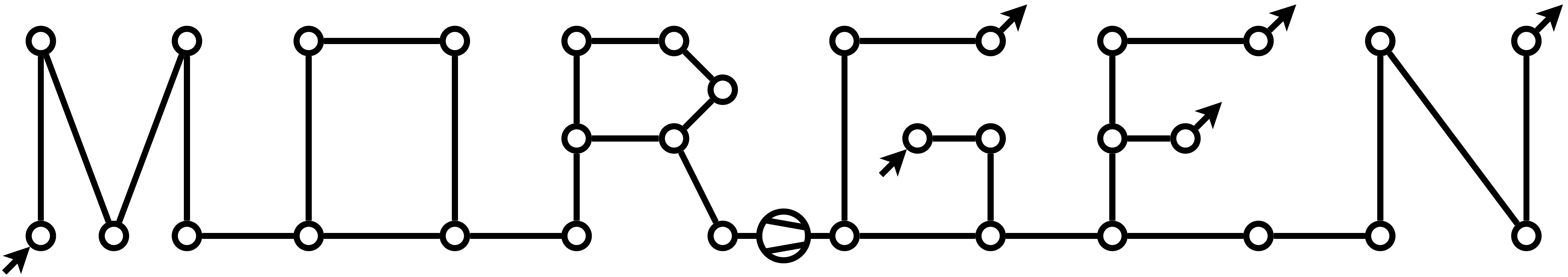}

 \vspace{2em}

 \caption{MORGEN network topology.}
 \label{fig:MORGEN:topo}
\end{figure}

\subsection{MORGEN Network}\label{sec:numex2}
The second set of tests encompasses a synthetic network for testing \morgen{}'s capabilities.
This ``MORGEN'' network, with topology as in \cref{fig:MORGEN:topo},
tests the interaction of simplified compressors from \cref{sec:comp} for various network features,
such as cycles, multiple supply and demand nodes, and is in the spirit of a test network from~\cite{morEggKLetal18}.

Specifically, six sub-networks (in the shape of letters) are connected,
the second and third sub-network contain a cycle,
a compressor connects the third and fourth sub-network,
and the fourth and fifth sub-network contain additional supply and demand nodes.
The edges vary in length between $20$km and $60$km,
while the diameter and roughness are consistently $1$m and $0.01$mm, respectively.
A steady-state, used as initial state, is set by supply (and discharge) pressures of $50$bar at both supply nodes and the compressor,
and demand mass-fluxes of $30\frac{\text{kg}}{\text{s}}$ at all demand nodes.

In semi-discrete form, the nonlinear state-space system features six inputs and outputs as well as $901$ states;
and a time discretization with $60$s time steps.
The employed $24$h test scenario is made from hourly standard load profiles~\cite{Hel03} and shown in \cref{fig:numex2:scen},
the associated model reduction errors are given in \cref{fig:numex2:mid1}, \cref{fig:numex2:end1}, \cref{fig:numex2:mid2}, and \cref{fig:numex2:end2},
for up to reduced order $200$, while the resulting \textsc{MORscore}s are listed in \cref{tab:numex2:morscore}.

Again for this comparison, the choice of solver is more relevant than the choice of model,
yet the \textsc{MORscore}s are much lower, due to the complexities (cycles, compressor, multiple demands) of the network.

Both models in conjunction with the first-order IMEX solver only produce workable results with Galerkin methods.
The endpoint ROMs are again more accurate than the midpoint ROMs.
Notably, the linear reachability-and-observability dominant subspaces method for the endpoint model, is leading the \textsc{MORscore}s.

The ROMs for both models with the second-order IMEX-RK solver perform worse,
with the exception of the balanced truncation \texttt{ebt\_ro} and balanced gains \texttt{ebg\_ro} variants.
However, the second-order IMEX-RK solver related ROMs are of no practical use due to the high(er) error.

Overall, the (linear) \texttt{eds\_ro} dominant subspaces reductor produces the lowest error,
followed by the DMD-Galerkin, (goal-oriented) POD, and cross-Gramian-based dominant subspaces methods.
As for the pipeline, the endpoint model is better suited for model reduction,
while the first-order IMEX solver results in significantly more accurate ROMs than the second-order IMEX-RK solver.

\pagebreak %temp

\begin{figure}[H]
\begin{subfigure}{0.49\textwidth}
\includegraphics[width=\textwidth]{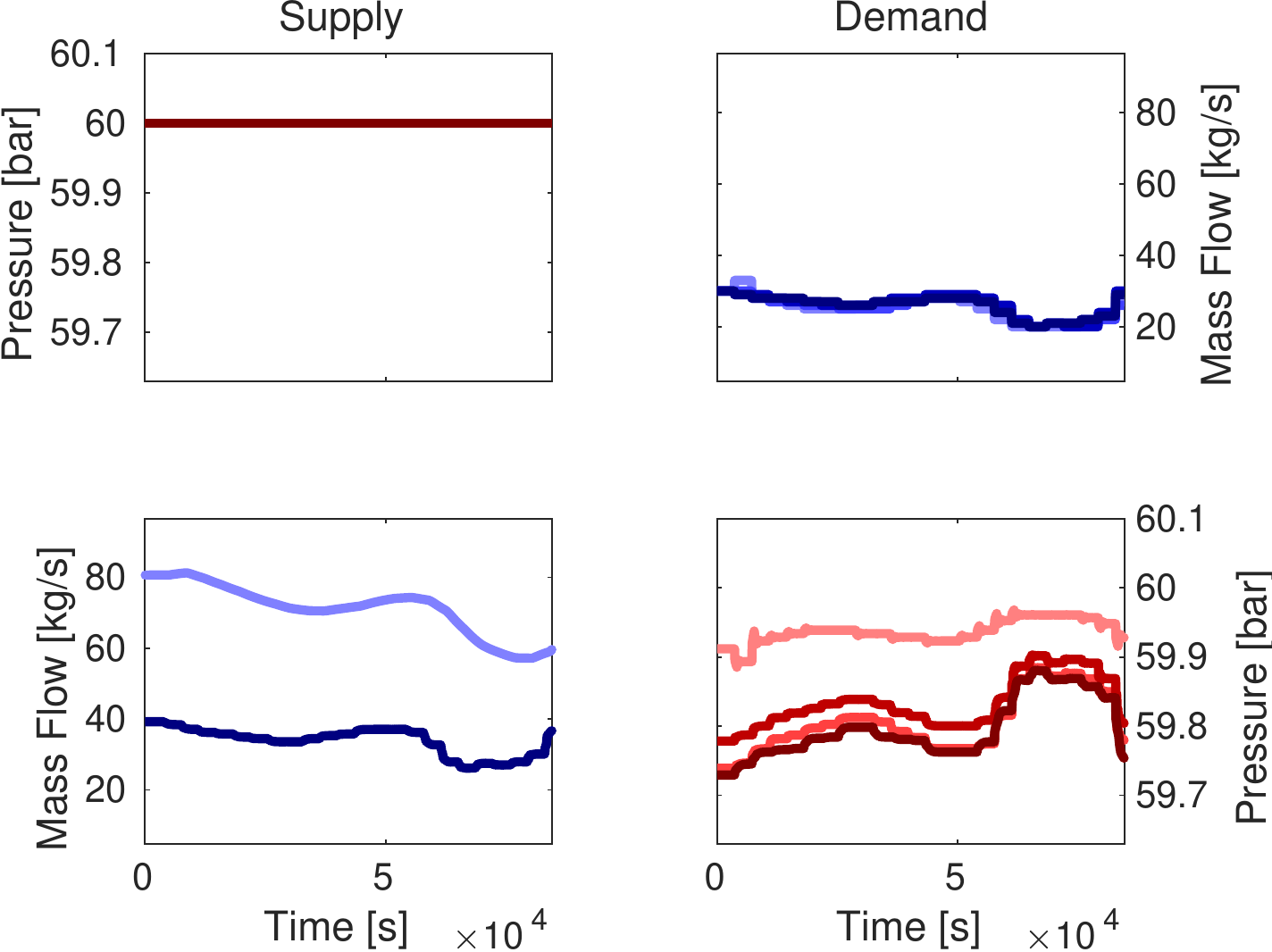}
\caption{Boundary values and quantities of \linebreak interest plots for the test scenario.}
\label{fig:numex2:scen}
\end{subfigure}
~~
\begin{subfigure}{0.49\textwidth}

\vspace{2em}

\includegraphics[width=\textwidth]{plots/legend.pdf}

\vspace{2em}

\caption{Common legend for the error plots \cref{fig:numex2:mid1}, \cref{fig:numex2:end1}, \cref{fig:numex2:mid2}, \cref{fig:numex2:end2}.}
\label{fig:numex2:legend}
\end{subfigure}

~\\

\begin{subfigure}{0.49\textwidth}
\includegraphics[width=\textwidth]{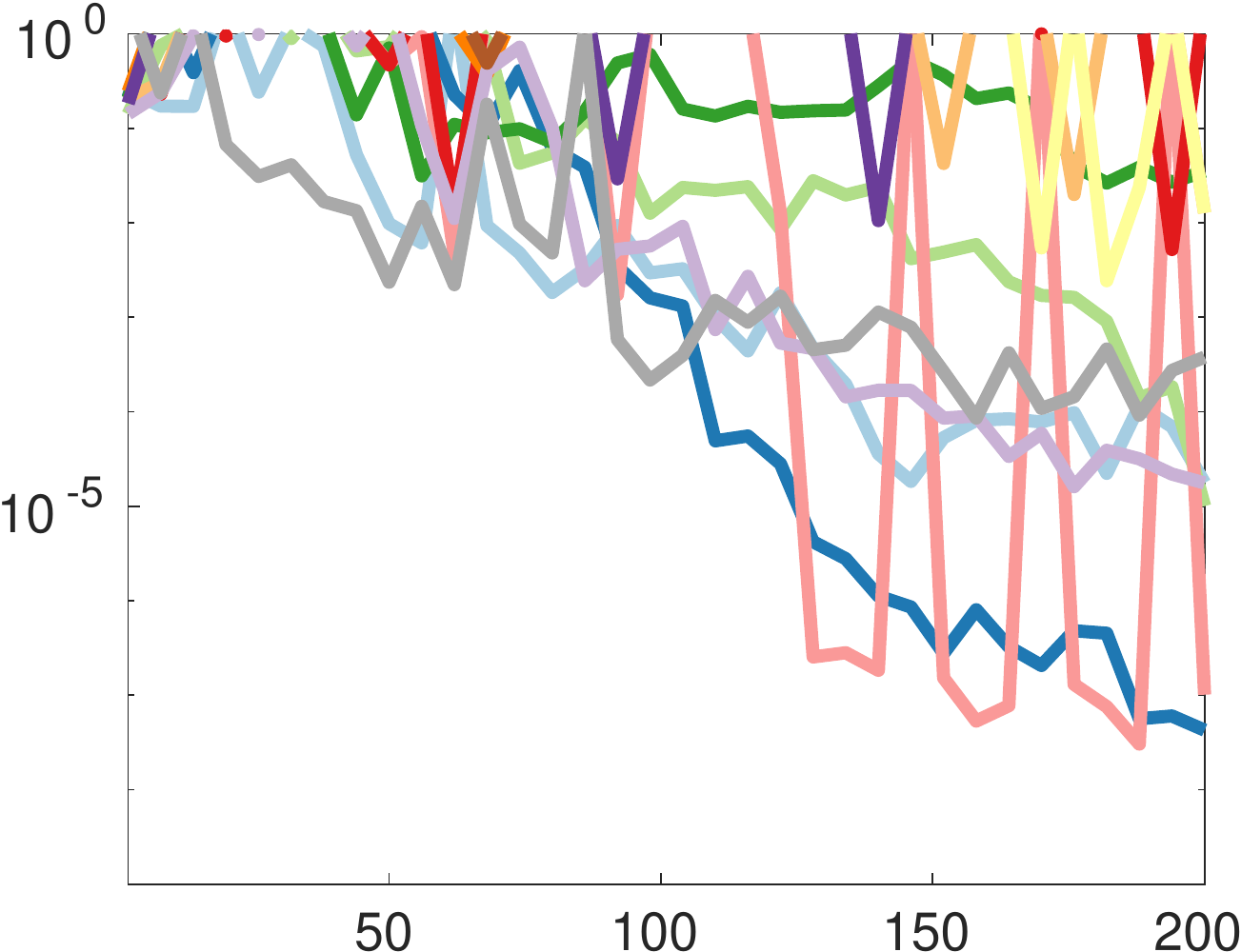}
\caption{$L_2\,\otimes\,L_2$ error between ROM and FOM for the \texttt{ode\_mid} model, \texttt{imex1} solver, and \emph{nonlinear} reductors versus reduced order.}
\label{fig:numex2:mid1}
\end{subfigure}
~~
\begin{subfigure}{0.49\textwidth}
\includegraphics[width=\textwidth]{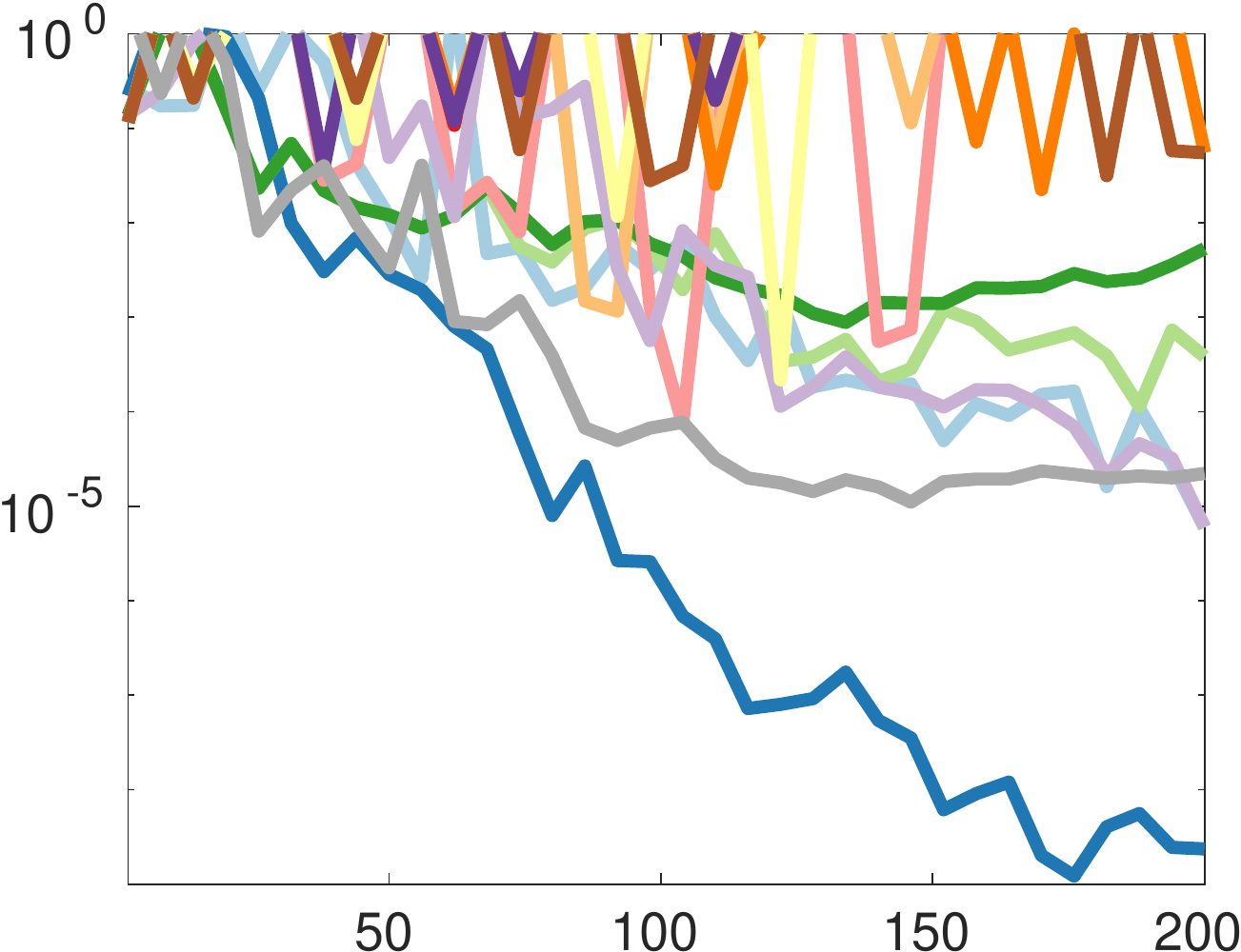}
\caption{$L_2\,\otimes\,L_2$ error between ROM and FOM for the \texttt{ode\_end} model, \texttt{imex1} solver, and \emph{linear} reductors versus reduced order.}
\label{fig:numex2:end1}
\end{subfigure}

~\\

\begin{subfigure}{0.49\textwidth}
\includegraphics[width=\textwidth]{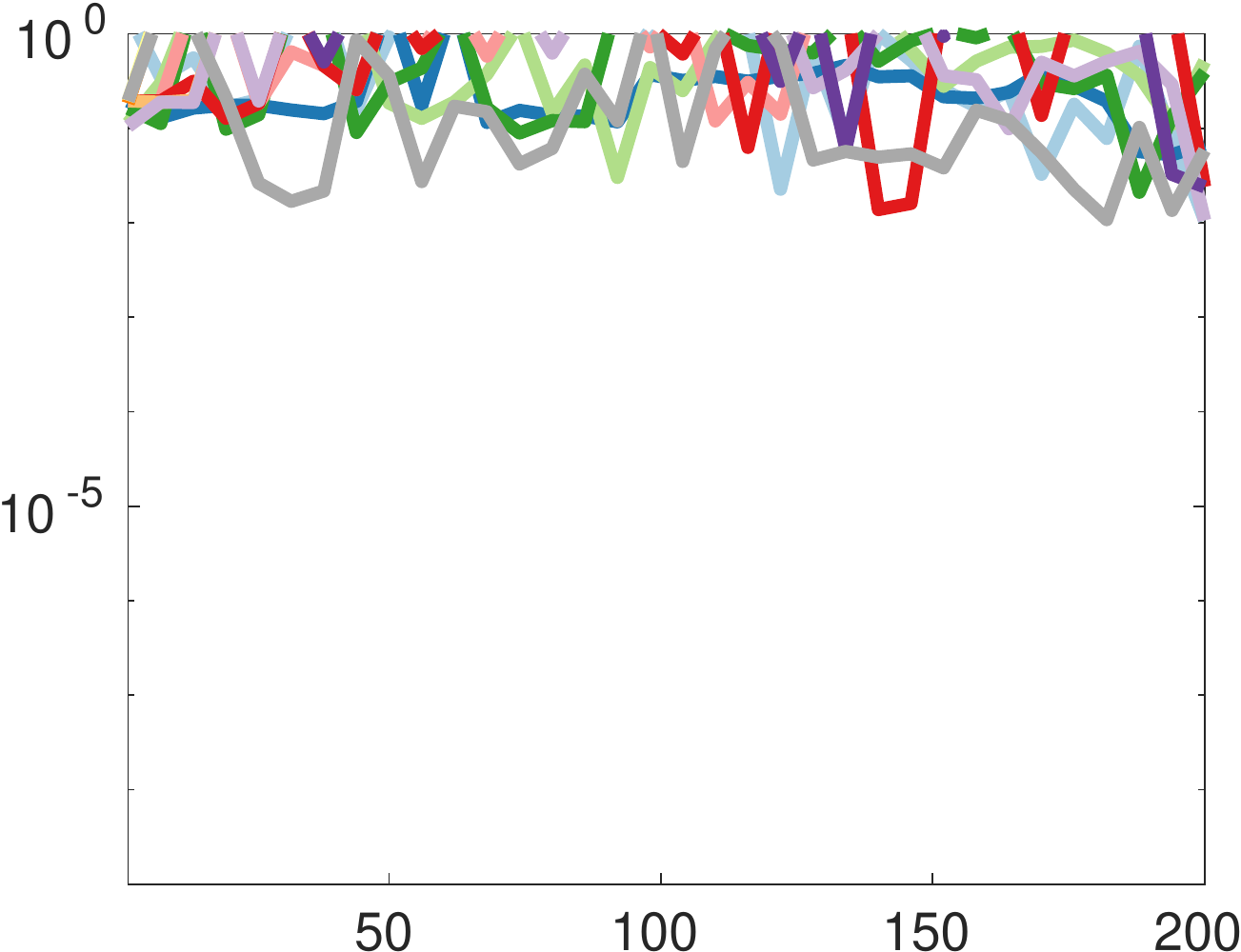}
\caption{$L_2\,\otimes\,L_2$ error between ROM and FOM for the \texttt{ode\_mid} model, \texttt{imex2} solver, and \emph{nonlinear} reductors versus reduced order.}
\label{fig:numex2:mid2}
\end{subfigure}
~~
\begin{subfigure}{0.49\textwidth}
\includegraphics[width=\textwidth]{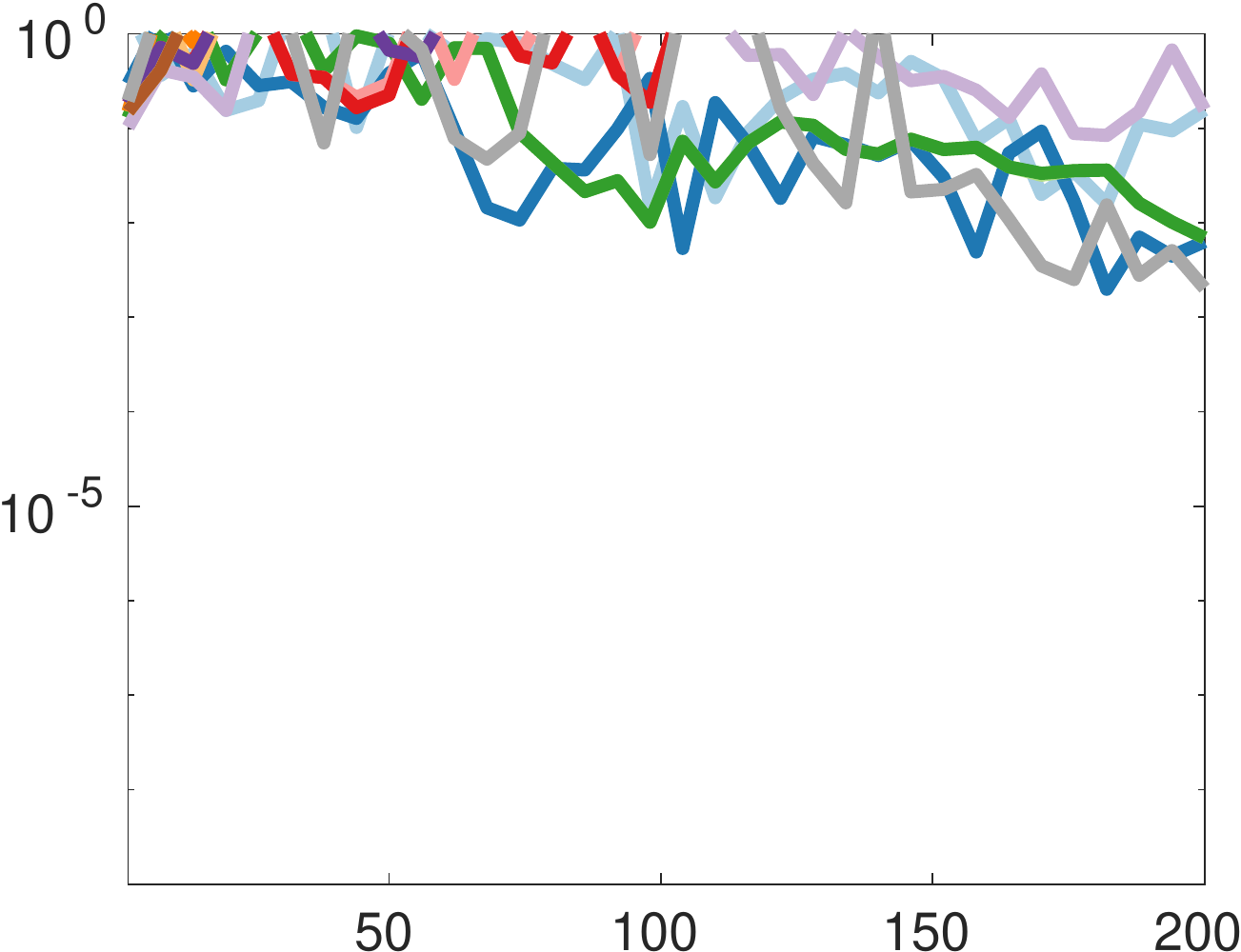}
\caption{$L_2\,\otimes\,L_2$ error between ROM and FOM for the \texttt{ode\_end} model, \texttt{imex2} solver, and \emph{linear} reductors versus reduced order.}
\label{fig:numex2:end2}
\end{subfigure}

\caption{Visualization of the test scenario, and model reduction errors of the tested ROMs for the MORGEN network from \cref{sec:numex2}.}
\label{fig:numex2}
\end{figure}

\pagebreak

\begin{figure}[H]\centering
 \includegraphics[width=.39\textwidth]{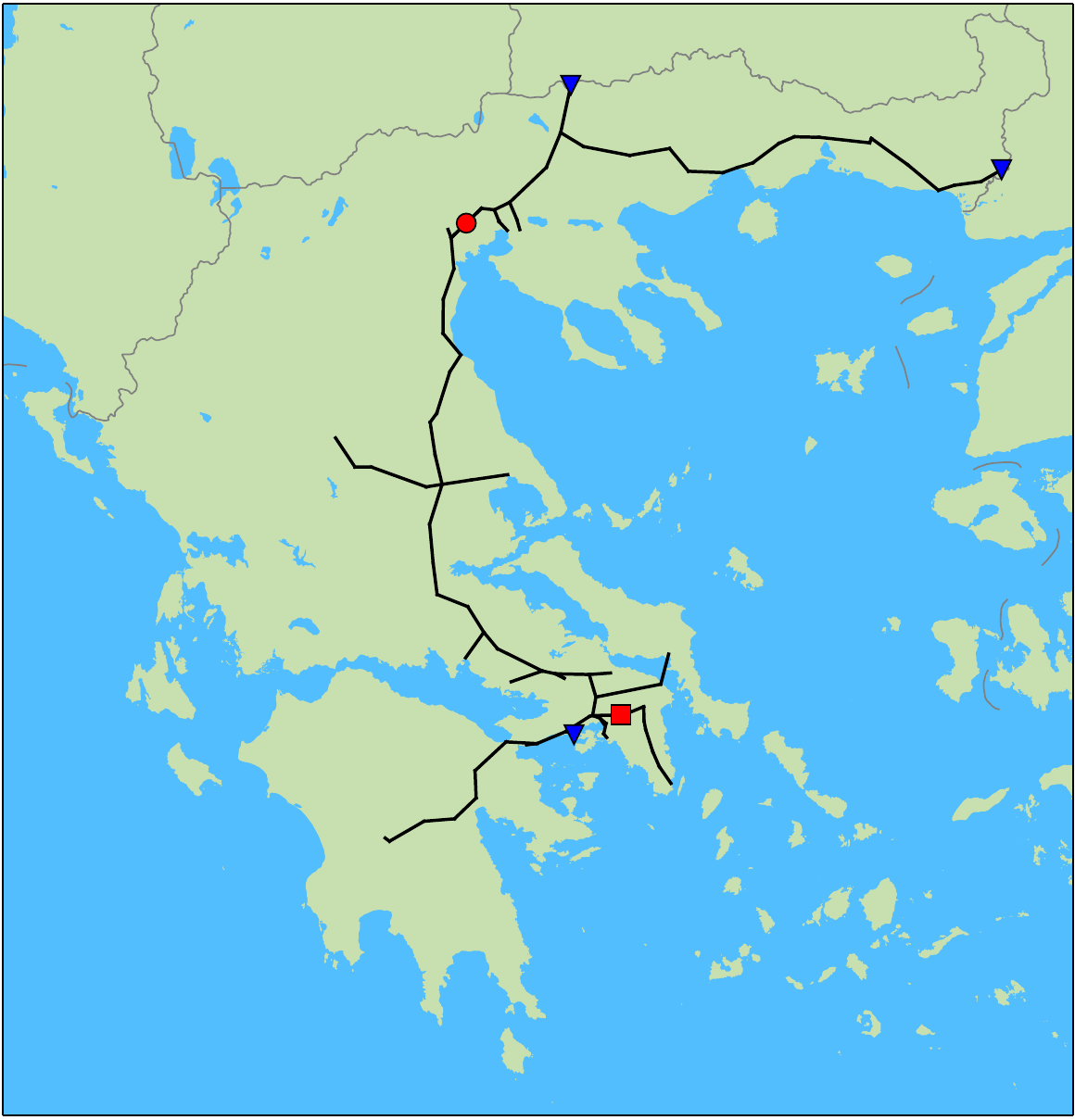}
 \caption{GasLib134v2 network topology modeling a part of the Greek gas network {\small(taken from \url{https://gaslib.zib.de}, licensed under \texttt{CC-BY}.)}.}
 \label{fig:gaslib:topo}
\end{figure}

\subsection{GasLib Network}\label{sec:numex3}
Lastly, a network topology derived from real-life is tested.
The GasLib-134v2 network~\cite{SchABetal17}, modeling a part of the Greek natural gas transport system,
is overlayed on a map of Greece in \cref{fig:gaslib:topo}.
The network has a total length of $1412$km and features a compressor.
A steady-state, used as initial state, is set by supply (and discharge) pressures of $80$bar at supply nodes and the compressor,
and demand mass-fluxes up to $16\frac{\text{kg}}{\text{s}}$ at all demand nodes.

In semi-discrete form, the nonlinear state-space system has $48$ inputs and outputs as well as $2682$ states;
and $30$s time steps are employed.
For testing, a random ($24$h) load profile is generated, by adding samples from a scaled uniform random distribution to the steady-state\footnote{\morgen{} can generate such profiles for all included networks.},
shown in \cref{fig:numex3:scen}, the associated model reduction errors are given in \cref{fig:numex3:mid1}, \cref{fig:numex3:end1}, \cref{fig:numex3:mid2}, and \cref{fig:numex3:end2},
for up to reduced order $250$, while the resulting \textsc{MORscore}s are listed in \cref{tab:numex3:morscore}.

As before, the choice of solver is more relevant than the choice of model.
Challenges in this network, beyond the compressor, are the high number of boundary nodes, which are predominantly demand nodes ($N_d = 45$).

First, we note that only Galerkin methods produce consistently stable ROMs.
Furthermore, in comparison with the previous experiments, the dominant subspace methods perform worse,
and all variants based on reachability and observability Gramians perform relatively better.
The endpoint model seems to be better suited for the tested model reduction methods than the midpoint model.
And as for the other experiments, the first order IMEX solver outmatches the second order IMEX-RK solver.

Considering all experiments, the DMD-Galerkin method performs best in terms of \textsc{MORscore}, accuracy and efficiency, followed by the dominant subspaces methods.
We also note that the purely reachability-based as well as the linear reductors exploiting the port-Hamiltonian structure are the most efficient.
Surprisingly, Galerkin methods perform better than the tested Petrov-Galerkin methods in terms of accuracy and stability,
while in an unstructured, non-parametric, linearized setting all tested Petrov-Galerkin methods would be stability preserving.
Yet, structured balancing methods are explicitly not guaranteed to be stability-preserving \cite{morVanV08,morSanM09}.

\pagebreak %temp

\begin{figure}[H]
\begin{subfigure}{0.49\textwidth}
\includegraphics[width=\textwidth]{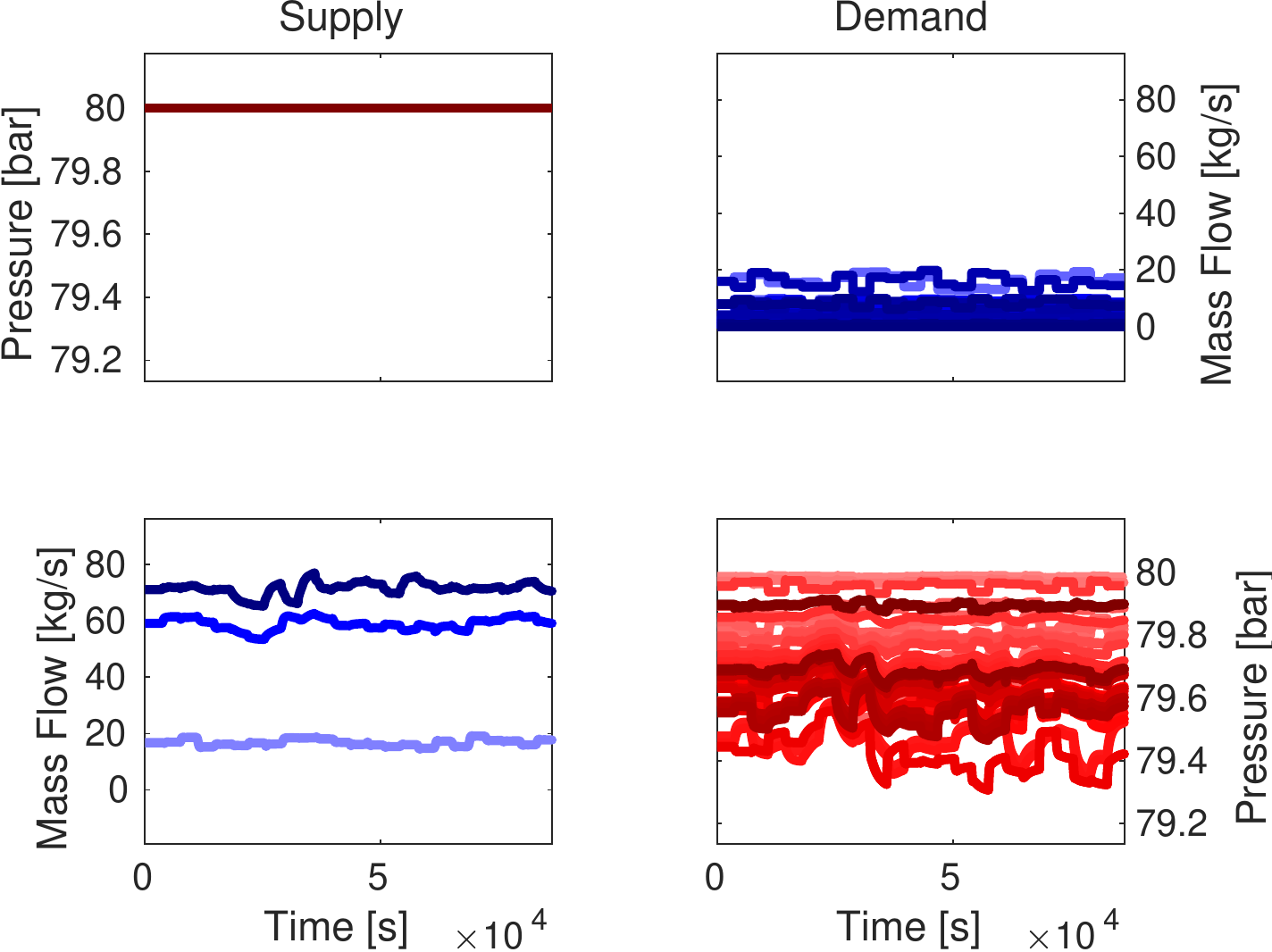}
\caption{Boundary values and quantities of \linebreak interest plots for the test scenario.}
\label{fig:numex3:scen}
\end{subfigure}
~~
\begin{subfigure}{0.49\textwidth}

\vspace{2em}

\includegraphics[width=\textwidth]{plots/legend.pdf}

\vspace{2em}

\caption{Common legend for the error plots \cref{fig:numex3:mid1}, \cref{fig:numex3:end1}, \cref{fig:numex3:mid2}, \cref{fig:numex3:end2}.}
\label{fig:numex3:legend}
\end{subfigure}

\begin{subfigure}{0.49\textwidth}
\includegraphics[width=\textwidth]{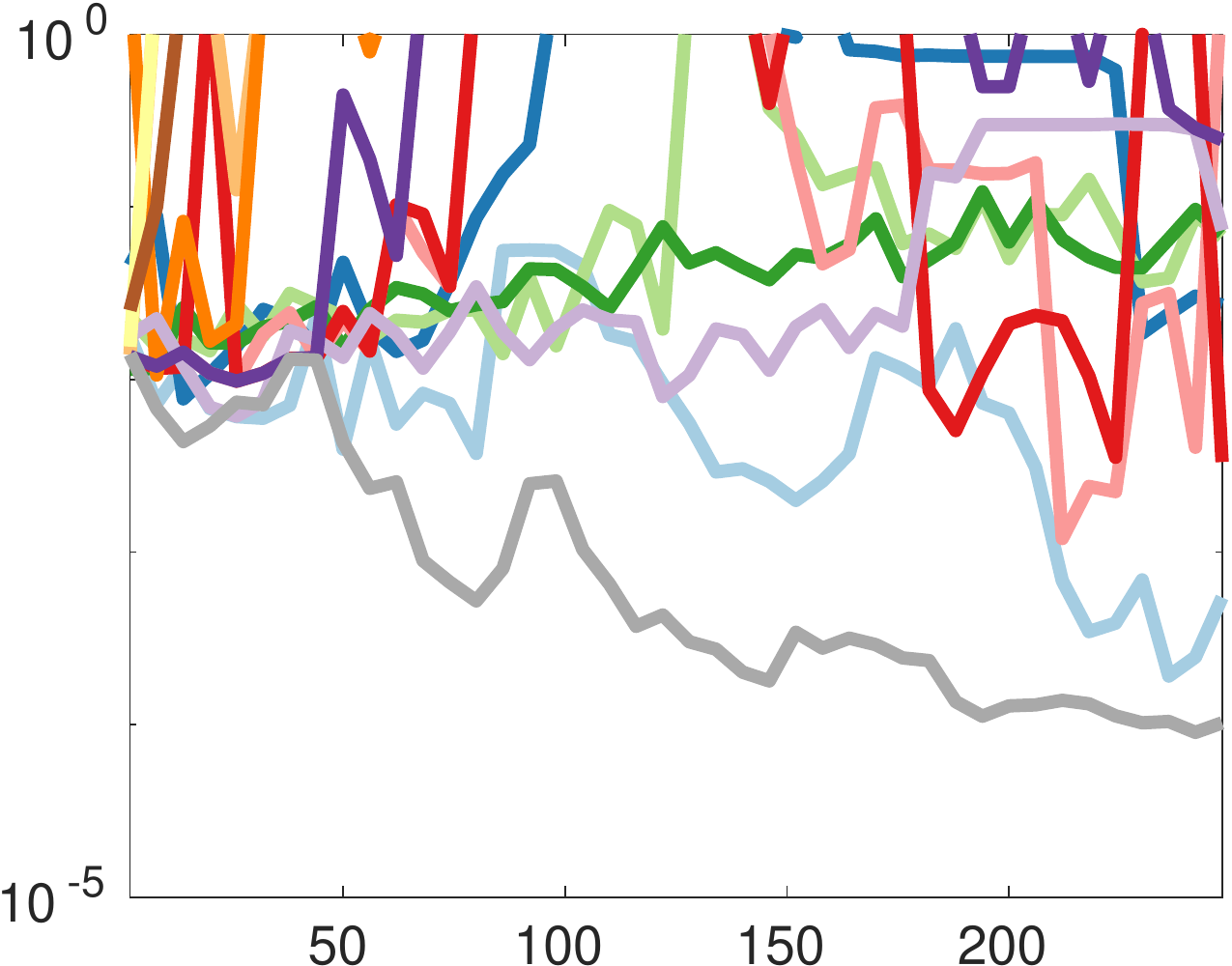}
\caption{$L_2\,\otimes\,L_2$ error between ROM and FOM for the \texttt{ode\_mid} model, \texttt{imex1} solver, and \emph{nonlinear} reductors versus reduced order.}
\label{fig:numex3:mid1}
\end{subfigure}
~~
\begin{subfigure}{0.49\textwidth}
\includegraphics[width=\textwidth]{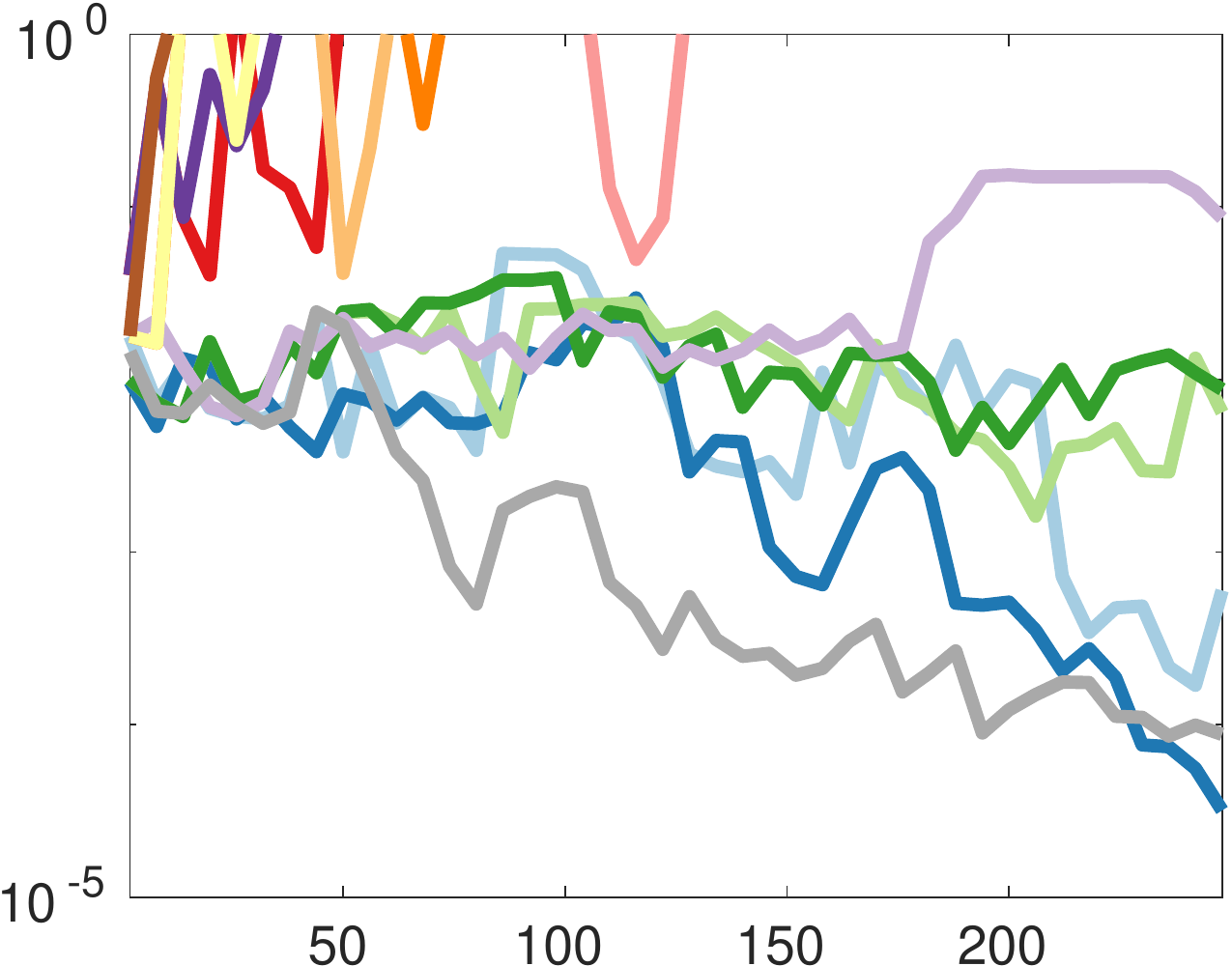}
\caption{$L_2\,\otimes\,L_2$ error between ROM and FOM for the \texttt{ode\_end} model, \texttt{imex1} solver, and \emph{linear} reductors versus reduced order.}
\label{fig:numex3:end1}
\end{subfigure}

~\\

\begin{subfigure}{0.49\textwidth}
\includegraphics[width=\textwidth]{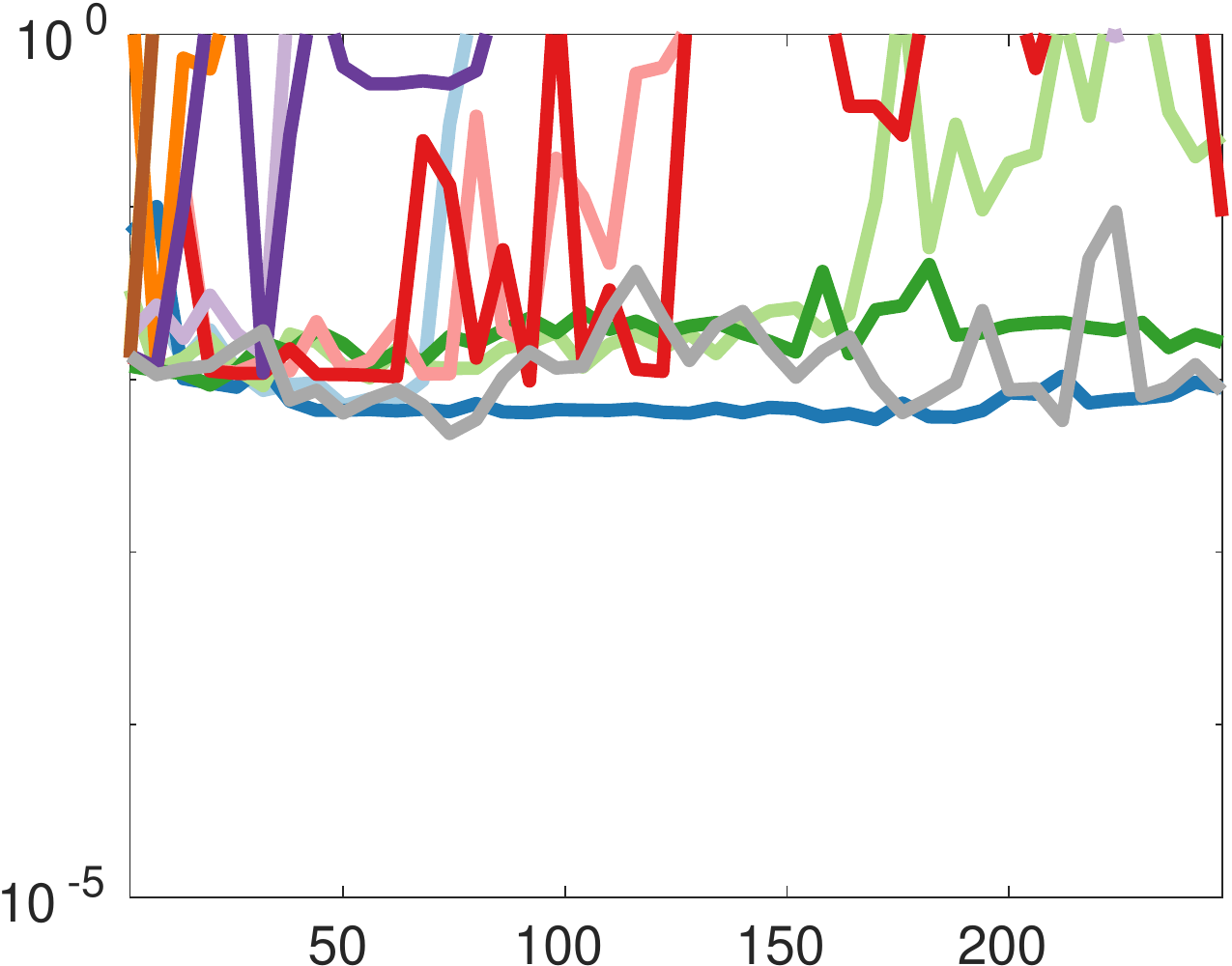}
\caption{$L_2\,\otimes\,L_2$ error between ROM and FOM for the \texttt{ode\_mid} model, \texttt{imex2} solver, and \emph{nonlinear} reductors versus reduced order.}
\label{fig:numex3:mid2}
\end{subfigure}
~~
\begin{subfigure}{0.49\textwidth}
\includegraphics[width=\textwidth]{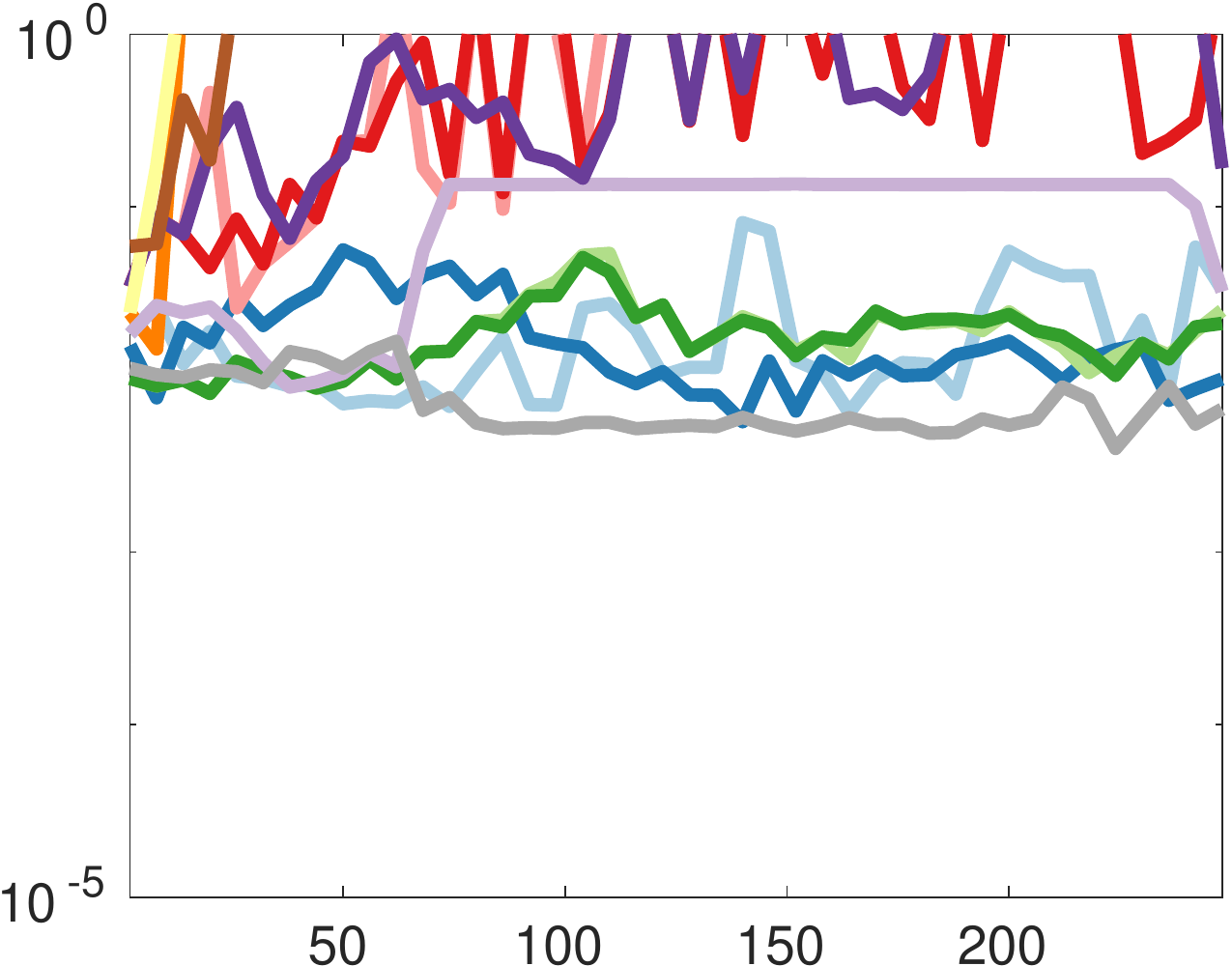}
\caption{$L_2\,\otimes\,L_2$ error between ROM and FOM for the \texttt{ode\_end} model, \texttt{imex2} solver, and \emph{linear} reductors versus reduced order.}
\label{fig:numex3:end2}
\end{subfigure}

\caption{Visualization of the test scenario, and model reduction errors of the tested ROMs for the GasLib-134v2 network from \cref{sec:numex3}.}
\label{fig:numex3}
\end{figure}

\pagebreak

\begin{table}[H]\centering\small
\begin{tabular}{r|c|c|c|c|c|c}
& \shortstack{\footnotesize\texttt{ode\_mid} \\ \footnotesize\texttt{imex\_1}}
& \shortstack{\footnotesize\texttt{ode\_end} \\ \footnotesize\texttt{imex\_1}}
& \shortstack{\footnotesize\texttt{ode\_end} \\ \footnotesize\texttt{imex\_1*}}
& \shortstack{\footnotesize\texttt{ode\_mid} \\ \footnotesize\texttt{imex\_2}}
& \shortstack{\footnotesize\texttt{ode\_end} \\ \footnotesize\texttt{imex\_2}}
& \shortstack{\footnotesize\texttt{ode\_end} \\ \footnotesize\texttt{imex\_2*}} \\
\hline \rowcolor{lightgray}
\texttt{pod\_r~}   & 0.40 & \multicolumn{2}{c|}{0.40} & 0.19 & \multicolumn{2}{c}{0.22} \\
\texttt{eds\_ro}   & 0.52 & 0.54 & 0.51 & 0.07 & 0.09 & 0.04 \\\rowcolor{lightgray}
\texttt{eds\_wx}   & 0.51 & 0.55 & 0.57 & 0.06 & 0.13 & 0.20 \\
\texttt{eds\_wz}   & 0.55 & 0.58 & 0.56 & 0.07 & 0.14 & 0.20 \\ \hdashline \rowcolor{lightgray}
\texttt{bpod\_ro}  & 0.19 & 0.28 & 0.14 & 0.03 & 0.08 & 0.11 \\
\texttt{ebt\_ro}   & 0.05 & 0.06 & 0.03 & 0.07 & 0.12 & 0.17 \\ \rowcolor{lightgray}
\texttt{ebt\_wx}   & 0.17 & 0.11 & 0.10 & 0.00 & 0.00 & 0.04 \\
\texttt{ebt\_wz}   & 0.24 & 0.30 & 0.15 & 0.00 & 0.00 & 0.04 \\ \hdashline \rowcolor{lightgray}
\texttt{gopod\_r~} & 0.40 & \multicolumn{2}{c|}{0.41} & 0.08 & \multicolumn{2}{c}{0.19} \\
\texttt{ebg\_ro}   & 0.05 & 0.11 & 0.02 & 0.05 & 0.11 & 0.15 \\ \rowcolor{lightgray}
\texttt{ebg\_wx}   & 0.08 & 0.23 & 0.14 & 0.00 & 0.00 & 0.00 \\
\texttt{ebg\_wz}   & 0.17 & 0.29 & 0.11 & 0.00 & 0.00 & 0.00 \\ \hdashline \rowcolor{lightgray}
\texttt{dmd\_r~}   & 0.50 & \multicolumn{2}{c|}{0.53} & 0.08 & \multicolumn{2}{c}{0.15}
\end{tabular}
\vskip-1ex
\caption{\textsc{MORscore}s $\mu(150,\epsilon_{\operatorname{mach}(16)})$ in the $L_2 \otimes L_2$ error norm
         for the \linebreak ``Yamal-Europe'' pipeline network from \cref{sec:numex1}; \texttt{*} notes linear reductors.}
\label{tab:numex1:morscore}
\end{table}

\vskip-2.5ex

\begin{table}[H]\centering\small
\begin{tabular}{r|c|c|c|c|c|c}
& \shortstack{\footnotesize\texttt{ode\_mid} \\ \footnotesize\texttt{imex\_1}}
& \shortstack{\footnotesize\texttt{ode\_end} \\ \footnotesize\texttt{imex\_1}}
& \shortstack{\footnotesize\texttt{ode\_end} \\ \footnotesize\texttt{imex\_1*}}
& \shortstack{\footnotesize\texttt{ode\_mid} \\ \footnotesize\texttt{imex\_2}}
& \shortstack{\footnotesize\texttt{ode\_end} \\ \footnotesize\texttt{imex\_2}}
& \shortstack{\footnotesize\texttt{ode\_end} \\ \footnotesize\texttt{imex\_2*}} \\
\hline \rowcolor{lightgray}
\texttt{pod\_r~}   & 0.16 & \multicolumn{2}{c|}{0.16} & 0.02 & \multicolumn{2}{c}{0.04} \\
\texttt{eds\_ro}   & 0.20 & 0.20 & 0.32 & 0.04 & 0.04 & 0.08 \\ \rowcolor{lightgray}
\texttt{eds\_wx}   & 0.09 & 0.10 & 0.15 & 0.01 & 0.03 & 0.06 \\
\texttt{eds\_wz}   & 0.04 & 0.09 & 0.13 & 0.02 & 0.04 & 0.06 \\ \hdashline \rowcolor{lightgray}
\texttt{bpod\_ro}  & 0.13 & 0.13 & 0.03 & 0.00 & 0.02 & 0.00 \\
\texttt{ebt\_ro}   & 0.00 & 0.00 & 0.00 & 0.01 & 0.02 & 0.00 \\ \rowcolor{lightgray}
\texttt{ebt\_wx}   & 0.00 & 0.00 & 0.00 & 0.00 & 0.00 & 0.00 \\
\texttt{ebt\_wz}   & 0.00 & 0.00 & 0.00 & 0.00 & 0.00 & 0.00 \\ \hdashline \rowcolor{lightgray}
\texttt{gopod\_r~} & 0.14 & \multicolumn{2}{c|}{0.14} & 0.01 & \multicolumn{2}{c}{0.01} \\
\texttt{ebg\_ro}   & 0.00 & 0.00 & 0.00 & 0.00 & 0.01 & 0.00 \\ \rowcolor{lightgray}
\texttt{ebg\_wx}   & 0.00 & 0.00 & 0.00 & 0.00 & 0.00 & 0.00 \\
\texttt{ebg\_wz}   & 0.00 & 0.00 & 0.01 & 0.00 & 0.00 & 0.00 \\ \hdashline \rowcolor{lightgray}
\texttt{dmd\_r~}   & 0.16 & \multicolumn{2}{c|}{0.21} & 0.06 & \multicolumn{2}{c}{0.05}
\end{tabular}
\vskip-1ex
\caption{\textsc{MORscore}s $\mu(200,\epsilon_{\operatorname{mach}(16)})$ in the $L_2 \otimes L_2$ error norm
         for the \linebreak ``MORGEN'' test network from \cref{sec:numex2}; \texttt{*} notes linear reductors.}
\label{tab:numex2:morscore}
\end{table}

\vskip-2.5ex

\begin{table}[H]\centering\small
\begin{tabular}{r|c|c|c|c|c|c}
& \shortstack{\footnotesize\texttt{ode\_mid} \\ \footnotesize\texttt{imex\_1}}
& \shortstack{\footnotesize\texttt{ode\_end} \\ \footnotesize\texttt{imex\_1}}
& \shortstack{\footnotesize\texttt{ode\_end} \\ \footnotesize\texttt{imex\_1*}}
& \shortstack{\footnotesize\texttt{ode\_mid} \\ \footnotesize\texttt{imex\_2}}
& \shortstack{\footnotesize\texttt{ode\_end} \\ \footnotesize\texttt{imex\_2}}
& \shortstack{\footnotesize\texttt{ode\_end} \\ \footnotesize\texttt{imex\_2*}} \\
\hline \rowcolor{lightgray}
\texttt{pod\_r~}   & 0.14 & \multicolumn{2}{c|}{0.14} & 0.02 & \multicolumn{2}{c}{0.11} \\
\texttt{eds\_ro}   & 0.04 & 0.07 & 0.16 & 0.13 & 0.12 & 0.11 \\ \rowcolor{lightgray}
\texttt{eds\_wx}   & 0.07 & 0.07 & 0.12 & 0.08 & 0.12 & 0.11 \\
\texttt{eds\_wz}   & 0.09 & 0.08 & 0.12 & 0.11 & 0.12 & 0.11 \\ \hdashline \rowcolor{lightgray}
\texttt{bpod\_ro}  & 0.06 & 0.06 & 0.00 & 0.04 & 0.05 & 0.01 \\
\texttt{ebt\_ro}   & 0.05 & 0.05 & 0.00 & 0.05 & 0.07 & 0.02 \\ \rowcolor{lightgray}
\texttt{ebt\_wx}   & 0.00 & 0.00 & 0.00 & 0.00 & 0.00 & 0.00 \\
\texttt{ebt\_wz}   & 0.00 & 0.00 & 0.00 & 0.00 & 0.00 & 0.00 \\ \hdashline \rowcolor{lightgray}
\texttt{gopod\_r~} & 0.09 & \multicolumn{2}{c|}{0.10} & 0.01 & \multicolumn{2}{c}{0.07} \\
\texttt{ebg\_ro}   & 0.02 & 0.03 & 0.00 & 0.00 & 0.02 & 0.01 \\ \rowcolor{lightgray}
\texttt{ebg\_wx}   & 0.00 & 0.00 & 0.00 & 0.00 & 0.00 & 0.00 \\
\texttt{ebg\_wz}   & 0.00 & 0.00 & 0.00 & 0.00 & 0.00 & 0.00 \\ \hdashline \rowcolor{lightgray}
\texttt{dmd\_r~}   & 0.20 & \multicolumn{2}{c|}{0.19} & 0.12 & \multicolumn{2}{c}{0.13}
\end{tabular}
\vskip-1ex
\caption{\textsc{MORscore}s $\mu(250,\epsilon_{\operatorname{mach}(16)})$ in the $L_2 \otimes L_2$ error norm
         for the \linebreak ``GasLib-134v2'' benchmark network from \cref{sec:numex3}; \texttt{*} notes linear reductors.}
\label{tab:numex3:morscore}
\end{table}

\pagebreak

With regard to the computational complexity of the offline and online phase, we reiterate,
that due to the absence of hyper-reduction, the online runtimes are not competitive (see \cref{sec:hypred}),
thus, we focus on the offline phase.
Yet, due to the practical reducibility of the state-space dimension by more than one order of magnitude
in the numerical experiments using the first order IMEX solver, a considerable speed-up is to be expected.

For the tested data-driven (time-domain) model reduction methods, the number of vector-field evaluations,
or relatedly, the number of simulated trajectories measures the complexity,
as these constitute their principal fraction.
The empirical reachability Gramian requires $N_s + N_d$ (number of ports) trajectories.
The empirical observability Gramian requires $N_p + N_q$ (number of states) trajectories.
The empirical cross Gramian requires $N_s + N_d + N_p + N_q$ trajectories,
and the linear empirical cross Gramian requires $2(N_s + N_d)$ trajectories.
For the tested reductors this amounts to $N_s + N_d$ trajectories for the POD, goal-oriented POD, and DMD-Galerkin method,
while the port-Hamiltonian variants of the dominant subspaces, balanced POD, balanced truncation and balanced gains methods need $2(N_s + N_d)$ trajectories,
and their non-port-Hamiltonian variants require $N_s + N_d + N_p + N_q$ trajectories.

These predicted complexities are reflected in the offline runtimes, when computed sequentially.
As the computation of trajectories is embarrassingly parallel, all trajectories are however computable simultaneously. 
Nonetheless, the complexities of the reachability-Gramian-only and port-Hamiltonian reductors are independent from the discretization,
and thus most relevant for large-scale gas networks.

\section{Outlook}\label{sec:outlook}
The next stage in the development of \morgen{} involves testing larger real-life networks,
such as the deliverable of the \emph{SciGRID\_gas}\footnote{\url{https://www.gas.scigrid.de}} (Open Source Model of the European Gas Network) project.
Yet, various further venues of linked modeling and model reduction questions are still not covered by \morgen{}.

In terms of model reduction, an interesting issue are intraday switchable valves,
which change the topology of the gas network graph and likely require to extend the utilized model reduction methods towards these switched systems.

Another interesting question in need of further investigation is the minimal time horizon of the training phase.
A lower bound is the time step times the longest path from a supply to a demand node,
but this is likely not sufficient.

Besides, an additional hyper-reduction module (\cref{sec:hypred}) post-processing the reduced order models,
a decoupler module pre-processing (DAE) models as described in \cite{morBanGB20,morBanAGetal20} is projected.

Also as detailed in \cref{sec:pmor}, the pipe roughness is a relevant parameter for (transient) simulations~\cite{SunZ19},
yet the entailing high-dimensional parameter space, due to the locally differing roughness and attrition rates, would have to be treated, too.
This in turn would raise the question for combined state and parameter reduction~\cite{morHim17},
and is postponed to future investigations.

Finally, using a tunable efficiency factor~\cite{Osi87,OsiC10,PfeFGetal15} that scales the model's friction term,
can be used to tune the models to match real data.

\section{Conclusions}\label{sec:conclusion}
In more than half a century of computational transient gas network simulation research and industrial use,
\morgen{} seems to be the first open-source platform (modularly) covering modeling, simulation, and model order reduction of gas (and energy) networks.
The target applications for \morgen{} are finding the best model reduction method or best reduced order model for a network by heuristic comparison,
as well as comparing model-solver-reductor simulation ensembles.

From a mathematical point of view, a next generation gas network simulation stack should consist of
a (endpoint) port-Hamiltonian model, a (first-order) IMEX solver, and a block-diagonal Galerkin projection reductor,
which is confirmed by the numerical results.

This results in the following heuristically determined but theoretically explainable recommended combination:
The endpoint model together with the first-order IMEX solver, and a Galerkin reductor,
specifically a structured dominant subspaces or structured DMD-Galerkin,
exhibiting the highest \textsc{MORscores} in the numerical experiments.
The performance of structured balanced truncation and the related structured balanced gains
may be improved in terms of stability(-preservation) either by a variant of the technique~\cite{morLuMM16},
a stabilizing inner product \cite{morSerLGetal12}, an (energy-)stable inner product~\cite{morKalBAetal14},
or an optimization-based post-processing as in~\cite{morBenHM18}.

Lastly, we invite researchers, engineers and users to provide their reductors, solvers, networks and scenarios for expansion and testing with \morgen{} for a broader view of this comparison.

\section*{Availability of Data and Materials}
\begin{minipage}{\linewidth} % minipage to avoid pagebreak inside the framed window
  \begin{framed}
    The Matlab language source code of the \morgen{} platform~1.0 is licensed under \textsc{BSD-2-Clause License},
    can be obtained from:
    \begin{center}
      \href{http://doi.org/10.5281/zenodo.4288509}{\texttt{doi:10.5281/zenodo.4288509}}
    \end{center}
    and is authored by: \textsc{C.~Himpe} and \textsc{S.~Grundel}.
  \end{framed}
\end{minipage}

\section*{Funding}
This work is supported by the German Federal Ministry for Economic Affairs and Energy,
in the joint project:
``\textbf{MathEnergy} -- Mathematical Key Technologies for Evolving Energy Grids'',
sub-project: Model Order Reduction (Grant number: 0324019\textbf{B}),

and by the National Science Foundation under Grant number: DMS-1439786
while the author was in residence at the Institute for Computational and Experimental Research in Mathematics in Providence, RI,
during the ``\textbf{Model and dimension reduction in uncertain and dynamic systems}'' program.

\section*{Appendix}\label{sec:appendix}
\subsection*{Model Fact Sheet}

\begin{tabular}{rl}
\textbf{Basis:} & Euler equations for cylindrical pipes \\[1ex]
\textbf{Assumptions} \\
Long pipes: & One spatial dimension \\
Kinetic term: & Removed due to slow subsonic velocities \\
Boundary values: & Low-frequency (sum of) step functions \\[1ex]
\textbf{Simplifications} \\
Temperature: & Isothermal {\small(temperature is parameter)} \\
Gas composition: & Constant global {\small(specific gas constant is parameter)} \\
Compressibility: & Constant global {\small(derived from steady-state)} \\
Compressors: & Affine / Additive \\[1ex]
\textbf{Modularization} \\
Friction:         & {\small Hofer \orelse Nikuradse \orelse Altshul \orelse Schifrinson \orelse PMT1025 \orelse IGT} \\
Compressibility:  & Ideal \orelse DVGW-G-2000 \orelse AGA88 \orelse Papay \\[1ex]
\textbf{Discretization} \\
Spatial:  & 1st order upwind finite differences \\
Temporal: & RK-4 \orelse IMEX-1 \orelse IMEX-RK-2 \orelse Rosenbrock-2
\end{tabular}

\section*{Abbreviations}
\begin{tabular}{rl}
CFL & Courant-Friedrichs-Levy \\
CSV & Comma Separated Value \\
DAE & Differential Algebraic Equation \\
DEIM & Discrete Empirical Interpolation Method \\
DIRK & Diagonally Implicit RK \\
DMD & Dynamic Mode Decomposition \\
DMDc & DMD with Control \\
DSPMR & Dominant Subspace Projection Model Reduction \\
EVD & Eigenvalue Decomposition \\
IMEX & Implicit-Explicit \\
MOR & Model Order Reduction \\
ODE & Ordinary Differential Equation \\
PDAE & Partial DAE \\
PDE & Partial Differential Equation \\
PDIRK & Passive DIRK \\
POD & Proper Orthogonal Decomposition \\
RK & Runge-Kutta \\
ROM & Reduced Order Model \\
SDIRK & Singly DIRK \\
SPMOR & Structure-Preserving Model Order Reduction \\
SRSN & System with Repeated Scalar Nonlinearities \\
SSP & Strong Stability Preserving \\
SVD & Singular Value Decomposition \\
tEVD & truncated EVD \\
tSVD & truncated SVD
\end{tabular}

\section*{Competing Interest}
The authors declare that they have no competing interests.

%\section*{Acknowledgements}
%Not applicable.

%\section*{Authors' Contributions}
%\textsc{C.~Himpe} has authored the manuscript, implemented the \morgen{} platform and set up the numerical experiments.
%\textsc{S.~Grundel} ...
%\textsc{P.~Benner} ...

%bib
\bibliographystyle{plainurl}
\bibliography{mor,csc,software,tmp}

\end{document}